\documentclass{compositio}

\usepackage{scalerel} 

\newcommand{\smallheartsuit}{\scalebox{0.5}{$\heartsuit$}}

\newcommand{\smalldagger}{\scalebox{0.6}{$\dagger$}}

\usepackage{amsthm,amsfonts,amssymb,amsmath,amsxtra,mathtools}
\usepackage[dvipsnames]{xcolor}
\usepackage[all]{xy}
\SelectTips{cm}{}
\usepackage{tikz}
\usepackage{tikz-cd}
\usepackage{verbatim}
\usepackage{todonotes}
\usepackage{mathrsfs}
\usepackage{booktabs,makecell}
\usepackage{diagbox}

\RequirePackage{xspace}
\RequirePackage{etoolbox}
\RequirePackage{varwidth}
\RequirePackage{enumitem}
\RequirePackage{tensor}
\RequirePackage{mathtools}
\RequirePackage{longtable}
\RequirePackage{multirow}
\RequirePackage{makecell}
\RequirePackage{bigints}

\usepackage[backref=page]{hyperref}%

\newcommand{\inv}{\mathrm{inv}}
\newcommand{\Inv}{\mathrm{Inv}}

\newcommand{\disc}{{\mathrm{disc}}}

\DeclareMathOperator{\Gal}{Gal}
\DeclareMathOperator{\Hom}{Hom}
\DeclareMathOperator{\End}{End}
\DeclareMathOperator{\Res}{Res}

\newcommand{\Sh}{\mathrm{Sh}}

\let\Im\relax
\DeclareMathOperator{\im}{im}
\DeclareMathOperator{\Im}{Im}

\DeclareMathOperator{\Lie}{Lie}

\newcommand{\val}{{\mathrm{val}}}

\DeclareMathOperator{\tr}{Tr}

\newcommand{\SO}{\mathrm{SO}}

\newcommand{\isoarrow}{%
	\ifbool{@display}{\overset{\sim}{\longrightarrow}}{\xrightarrow\sim}%
}

\newcommand{\ab}{\mathrm{ab}}
\newcommand{\der}{\mathrm{der}}

\newcommand{\tL}{\widetilde{L}}

\DeclareFontFamily{U}{matha}{\hyphenchar\font45}
\DeclareFontShape{U}{matha}{m}{n}{
	<5> <6> <7> <8> <9> <10> gen * matha
	<10.95> matha10 <12> <14.4> <17.28> <20.74> <24.88> matha12
}{}
\DeclareSymbolFont{matha}{U}{matha}{m}{n}
\DeclareFontFamily{U}{mathx}{\hyphenchar\font45}
\DeclareFontShape{U}{mathx}{m}{n}{
	<5> <6> <7> <8> <9> <10>
	<10.95> <12> <14.4> <17.28> <20.74> <24.88>
	mathx10
}{}
\DeclareSymbolFont{mathx}{U}{mathx}{m}{n}

\DeclareMathSymbol{\obot}         {2}{matha}{"6B}

\newcommand{\bX}{\mathbb{X}}

\newcommand{\bV}{\mathbb{V}}

\newcommand{\Gr}{\mathrm{Gr}}

\newcommand{\Fil}{\mathrm{Fil}}

\newcommand{\bL}{\mathbb{L}}
\newcommand{\R}{\mathbb{R}}

\newcommand{\Q}{\mathbb{Q}}
\newcommand{\Z}{\mathbb{Z}}

\newcommand{\F}{\mathbb{F}}

\newcommand{\spa}{\mathrm{Span}}

\newcommand{\sm}{\ensuremath{\mathrm{sm}}}

\newcommand{\fka}{\ensuremath{\mathfrak{a}}}

\newcommand{\fkf}{\ensuremath{\mathfrak{f}}}

\newcommand{\fkm}{\ensuremath{\mathfrak{m}}}
\newcommand{\fkn}{\ensuremath{\mathfrak{n}}}

\newcommand{\fkp}{\ensuremath{\mathfrak{p}}}

\newcommand{\fks}{\ensuremath{\mathfrak{s}}}

\newcommand{\bbA}{\ensuremath{\mathbb{A}}}

\newcommand{\bbC}{\ensuremath{\mathbb{C}}}
\newcommand{\bbD}{\ensuremath{\mathbb{D}}}

\newcommand{\bbF}{\ensuremath{\mathbb{F}}}
\newcommand{\bbG}{\ensuremath{\mathbb{G}}}

\newcommand{\bbL}{\ensuremath{\mathbb{L}}}
\newcommand{\bbM}{\ensuremath{\mathbb{M}}}

\newcommand{\bbQ}{\ensuremath{\mathbb{Q}}}
\newcommand{\bbR}{\ensuremath{\mathbb{R}}}
\newcommand{\bbS}{\ensuremath{\mathbb{S}}}

\newcommand{\bbV}{\ensuremath{\mathbb{V}}}
\newcommand{\bbW}{\ensuremath{\mathbb{W}}}
\newcommand{\bbX}{\ensuremath{\mathbb{X}}}

\newcommand{\bbZ}{\ensuremath{\mathbb{Z}}}

\newcommand{\brF}{\ensuremath{\breve{F}}}
\newcommand{\brV}{\ensuremath{\breve{V}}}

\newcommand{\Gspin}{\ensuremath{\operatorname{GSpin}}}

\newcommand{\brD}{\ensuremath{\breve{D}}}
\newcommand{\brL}{\ensuremath{\breve{L}}}

\newcommand{\rmF}{\ensuremath{\mathrm{F}}}
\newcommand{\rmG}{\ensuremath{\mathrm{G}}}

\newcommand{\rmI}{\ensuremath{\mathrm{I}}}

\newcommand{\rmK}{\ensuremath{\mathrm{K}}}

\newcommand{\rmM}{\ensuremath{\mathrm{M}}}

\newcommand{\rmX}{\ensuremath{\mathrm{X}}}

\newcommand{\bfx}{\ensuremath{\mathbf{x}}}

\newcommand{\scrD}{\ensuremath{\mathscr{D}}}

\newcommand{\scrF}{\ensuremath{\mathscr{F}}}
\newcommand{\scrG}{\ensuremath{\mathscr{G}}}

\newcommand{\scrS}{\ensuremath{\mathscr{S}}}

\newcommand{\scrV}{\ensuremath{\mathscr{V}}}

\newcommand{\scrX}{\ensuremath{\mathscr{X}}}

\newcommand{\Mloc}{\ensuremath{\mathbb{M}^{\mathrm{loc}}}}
\newcommand{\calA}{\ensuremath{\mathcal{A}}}
\newcommand{\calB}{\ensuremath{\mathcal{B}}}

\newcommand{\calG}{\ensuremath{\mathcal{G}}}

\newcommand{\calI}{\ensuremath{\mathcal{I}}}

\newcommand{\calL}{\ensuremath{\mathcal{L}}}
\newcommand{\calM}{\ensuremath{\mathcal{M}}}
\newcommand{\calN}{\ensuremath{\mathcal{N}}}
\newcommand{\calO}{\ensuremath{\mathcal{O}}}

\newcommand{\calR}{\ensuremath{\mathcal{R}}}

\newcommand{\calT}{\ensuremath{\mathcal{T}}}

\newcommand{\calV}{\ensuremath{\mathcal{V}}}

\newcommand{\calY}{\ensuremath{\mathcal{Y}}}
\newcommand{\calZ}{\ensuremath{\mathcal{Z}}}

\newcommand{\OGr}{\ensuremath{\mathrm{OGr}}}
\newcommand{\tW}{\ensuremath{\widetilde{W}}}
\newcommand{\Adm}{\ensuremath{\operatorname{Adm}}}

\newcommand{\an}{\ensuremath{\mathrm{an}}}
\newcommand{\Span}{\ensuremath{\mathrm{Span}}}

\newcommand{\cV}{\mathcal{V}}
\newcommand{\cZ}{\mathcal{Z}}

\newcommand{\cY}{\mathcal{Y}}

\newcommand{\cG}{\mathcal{G}}

\newcommand{\cN}{\mathcal{N}}

\newcommand{\GL}{\mathrm{GL}}

\newcommand{\Nilp}{\mathrm{Nilp}\, }
\newcommand{\GSpin}{\mathrm{GSpin}}

\newcommand{\Spf}{\mathrm{Spf}\, }

\newcommand{\red}{\mathrm{red}}

\newcommand{\diag}{\mathrm{Diag}}

\newcommand{\Fl}{\mathrm{Fl}}

\newcommand{\brQ}{\breve{\mathbb{Q}}}
\newcommand{\brZ}{\breve{\mathbb{Z}}}

\newcommand{\p}{p}

\newcommand{\K}{K}

\newtheorem{theorem}{Theorem}[section]

\newtheorem{proposition}[theorem]{Proposition}
\newtheorem{lemma}[theorem]{Lemma}
\newtheorem{remark}[theorem]{Remark}

\theoremstyle{definition}
\newtheorem{definition}[theorem]{Definition}

\numberwithin{equation}{subsubsection}

\title{On the basic locus of GSpin Shimura varieties with vertex stabilizer level}

\begin{document}

\author{Qiao He}
\email{heqiaode@icloud.com}
\address{Zhizhen School of Interdisciplinary Mathematical Sciences,
	The Chinese University of Hong Kong,
	Shatin, New Territories, Hong Kong}

\author{Rong Zhou}
\email{rz240@dpmms.cam.ac.uk}
\address{Department of Pure Mathematics and Mathematical Statistics, University of Cambridge, Cambridge, UK, CB3 0WA}


\begin{abstract}We study the basic locus of Shimura varieties associated to the group of spinor similitudes of a quadratic space over $\bbQ$ with level structure given by the stabilizer of a vertex lattice. We give a description of the underlying reduced scheme of the associated Rapoport--Zink space, generalizing results of Howard--Pappas  and Oki, in the case of self-dual, and almost self-dual level structure.
	
\end{abstract}

    	\maketitle
	\tableofcontents

\section{Introduction}

 This paper is concerned with the study of the basic locus of certain Shimura varieties. Specifically, we consider Shimura varieties  associated to the group of spinor similitudes of a quadratic space over $\bbQ$ whose level structure  at an odd prime $p$ is given by the stabilizer of a vertex lattice. In this setting, good integral models were constructed in \cite{KisinPappas}, \cite{KPZ}, and our main result gives a decomposition of the basic locus in its special fiber into a union of classical Deligne--Lusztig varieties, generalizing work of \cite{HowardPappas} and \cite{Oki} in the special case of self-dual, and almost self-dual, level structure.  

To give some background, the basic locus of Shimura varieties is a generalization of the supersingular locus in the special fiber  of modular curves. Its study has a long history and has led to many important arithmetic applications. For example, the computation of intersection numbers of cycles in the basic locus forms a crucial input to various arithmetic intersection problems such as the Kudla--Rapoport conjecture; we refer to \cite{LZ},\cite{HLSY}, \cite{LZ2} and \cite{LiZhu2018} for applications along these lines. The description of the basic locus is also an essential ingredient for proving `arithmetic level raising' type results, which in turn are important pieces for the proof of cases of  the Beilinson--Bloch--Kato conjecture and Iwasawa main conjecture; see  for example \cite{LT},\cite{ZhouMotivic2023},\cite{LTXZZ}, \cite{LTX}. Moreover, as another striking application, the cycle classes coming from the basic locus often serve as natural candidates to prove  cases of the Tate conjecture; see \cite{TX1,HTX}.

In many of the above applications, a key point is to have an explicit description of the basic locus in terms of more familiar arithmetic objects such as Deligne--Lusztig varieties. Such a description was first discovered by 
 Vollaard \cite{vollaard2010supersingular} and Vollaard--Wedhorn \cite{vollaardwedhorn} in the case of Shimura varieties for the general unitary group $\mathrm{GU}(n-1,1)$ at an inert prime with hyperspecial level. The situation was later clarified by \cite{GHN19}, who gave a group theoretic criterion of when the basic locus is a union of classical Deligne--Lusztig varieties, but only after perfection. 
 For many arithmetic applications, it is crucial to have this description before taking perfection, and moreover to have a more explicit description of the isomorphism with the Deligne--Lusztig variety.
 In the case of Shimura varieties associated to unitary groups such as those considered in \cite{vollaardwedhorn}, \cite{cho2018basic}, this is achieved by leveraging the interpretation of the integral models as a moduli space of abelian varieties with extra structure. See also \cite{RTW,wu2016supersingular,HLY} for ramified unitary cases.

The Shimura varieties associated to spinor similitude groups that we consider in this paper no longer admit such a moduli interpretation, and so this causes difficulties when trying to prove the corresponding result in this setting. Howard--Pappas \cite{HowardPappas} were able to overcome this in the case of hyperspecial level, but a crucial result used in loc. cit. concerning triviality of torsors for reductive group schemes (cf. \cite[Theorem 5.2]{Nisnevich}) breaks down in our setting.
 We are thus required to introduce some new methods along the way, which we hope can be applied in more general non-moduli theoretic settings to study problems of this kind.

\subsection{Main results}
We now discuss our main theorems in more detail. We aim to characterize the irreducible components of the basic locus as the reduced subscheme of certain special cycles  and establish  isomorphisms between them and certain generalized Deligne--Lusztig varieties without passing to perfection.   
As in previous works, we approach this problem by investigating the corresponding Rapoport--Zink space; the analogous results for the Shimura variety are then deduced from this via the Rapoport--Zink uniformization; cf. Proposition \ref{prop: RZ uniformization}.

Let $p$ be an odd prime and let $V$ be a vector space over $\bbQ_p$ equipped with a symmetric bilinear form $(\ ,\ )$. We call a $\Z_p$-lattice $L\subset V$ a vertex lattice of type $t$ if $p L^\vee \subset L \subset L^{\vee}$ with $[L^\vee:L]=t$.\footnote{Our convention for vertex lattices is dual to the one in \cite{HowardPappas}.} Here $$L^{\vee}=\{x\in L\otimes_{\Z_p} \Q_p\mid (x,y) \in \Z_p, \, \forall y \in L\}$$ is the dual lattice. We use $n$ to denote the rank of $L$ and assume $n\ge 3$. From now on, we will fix a vertex lattice $L$ over $\Z_p$ of type $h$.

We let $G=\GSpin(V)$ be the group of spinor similitudes attached to $V$ and $\cG$ the parahoric of $G$ corresponding to $L$.\footnote{Note that $\cG$ is not necessarily a maximal parahoric, cf. \cite[Remark 2.5.2]{PappasZachos}.} Let $\brQ_p$ be the completion of the maximal unramified extension of $\bbQ_p$, $\brZ_p$ its ring of integers and $k$ the residue field of $\brZ_p$. Let $\mu$ be the cocharacter of $G$ defined as in \S \ref{Subsec: GSpin Shimura datum} and $b\in G(\brQ_p)$ a representative of the basic element in the neutral acceptable set $B(G,\mu)$. Then according to \cite{HowardPappas,HK19,PRshtukas}, we have a corresponding Rapoport--Zink space $\cN_L$, which is a formal scheme defined over $\mathrm{Spf}\breve{\Z}_p$.  The choice of $L$ gives rise to a choice of the framing object $\bX$ (see \S \ref{para: RZ space GSp}), and hence the level structure of $\cN_L$.  

\subsubsection{Special cycles}
In the global setting, one can define special cycles on GSpin Shimura varieties which are of great arithmetic interest. As a local analogue of Shimura variety, we can also define special cycles on $\cN_L$. 
Let $\bbV$ be the $\bbQ_p$-vector space given by taking the $\Phi$-invariants $V_{\breve{\Q}_p}^\Phi$ of $V_{\breve{\Q}_p}$, where $\Phi$ is the Frobenius given via twisting by the basic element $b$. Then by considering $\bbV$ as a space of quasi-endomorphisms of the framing object $\bbX$, we may define for any lattice $\bL\subset \bV$,  formal subschemes $\cZ(\bL)$ and $\calY(\bL)$ of $\calN_L$ (see \S\ref{para: Z and Y cycles}); we call these $\calZ$-cycles and $\calY$-cycles respectively.

We have a natural decomposition of $\cN_L$ into a disjoint union $\sqcup_{\ell \in \Z}\cN_L^{(\ell)}$  where $\ell$ is related to the degree of the universal quasi-isogeny. 
We use $\cN_{L,\red}$ (resp. $ \cN_{L,\red}^{(\ell)}$)  to denote the reduced locus of $\cN_L$ (resp. $ \cN_L^{(\ell)}$). The first main result of the current work is the following description of $\cN_{L,\red}^{(\ell)}$ in terms of the reduced locus of special cycles. To describe the results, we let $\cV^{\ge h}$ (resp. $\cV^{\le h}$) denote the set of vertex lattices in $\bV$ of type $t\ge h$ (resp. $t\le h$), and for $\Lambda\in \calV^{\geq h}\cup \calV^{\leq h}$, we write $t_\Lambda$ for the type of $\Lambda.$ 
\begin{theorem}\label{thm: intro 1}\hfill
 \begin{enumerate}
     \item  $\cN_L^{(\ell)}$ is connected
     \item  We have the following bijections
     \begin{align*}
        &\begin{cases}
  \cN_L^{(0)}(k)\sqcup \cN_L^{(1)}(k)   \stackrel{\sim}{\rightarrow} \{\text{special lattices } M \subset V_{\brQ_p}: M\stackrel{h}{\subset}M^\vee\} & \text{ if $h=0$ or $n$},\\
    \cN_L^{(0)}(k) \stackrel{\sim}{\rightarrow} \{\text{special lattices } M \subset V_{\brQ_p}: M\stackrel{h}{\subset}M^\vee\} & \text{ otherwise.}
\end{cases} 
     \end{align*}
 Here $M$ is a special lattice if $[M+\Phi(M):M]\le 1$ when $h\neq 0,n$ and $[M+\Phi(M):M]= 1$ when $h= 0,n$.
\item     We have a decomposition
    \begin{align*}
        \cN_{L,\red}^{(\ell)} = \Big (\cup_{\Lambda \in \cV^{\ge h}} \cZ_\red^{(\ell)}(\Lambda) \Big ) \cup \Big(\cup_{\Lambda \in \cV^{\le h}} \cY_\red^{(\ell)}(\Lambda^\vee) \Big).
    \end{align*}
Moreover, if we use $t_\mathrm{max}$ (resp. $t_\mathrm{min})$ to denote the maximal type (resp. minimal type) of a vertex lattice, then $\cZ_\red^{(\ell)}(\Lambda_{t_\mathrm{max}})$ and $\cY_\red^{(\ell)}(\Lambda_{t_\mathrm{min}}^\vee)$ are irreducible generalized Deligne-Lusztig varieties of dimensions $\frac{t_{\mathrm{max}}+h}{2}-1$ and  $n-\frac{t_{\mathrm{min}}+h}{2}-1$ respectively.
 \end{enumerate}
\end{theorem}
When $h=0$ or $n$ ($L$ is self-dual or $p$-modular), this was proved in \cite{HowardPappas}.\footnote{Strictly speaking \cite{HowardPappas} only deals with $h=0$; the case $h=n$ can be reduced to this, for example by using the duality isomorphism Proposition \ref{prop: duality for Z(Lambda)}.} Note that in these cases, we have $\cY(\Lambda)=\emptyset$ (resp. $\cZ(\Lambda)=\emptyset$) when $h=0$ (resp. $h=n$) so that there is only a single family of Deligne--Lusztig varieties appearing.

 Now we give more details.    
 First, for a vertex lattice $\Lambda\subset \bbV$ of type $t_\Lambda$ with quadratic form $q(\,)$, we set
\begin{align*}
    V_{\Lambda^\vee} \coloneqq \Lambda/p \Lambda^\vee \quad \text{ and } \quad V_{\Lambda }\coloneqq \Lambda^\vee/ \Lambda
\end{align*}
with quadratic form induced from $q(\,)$ and $p\cdot   q(\,)$ respectively. Note that $\dim_{\F_p}V_{\Lambda^\vee}=n-t_\Lambda$ and $\dim_{\F_p}V_{\Lambda}=t_\Lambda$.   
Set
	\begin{align*} 
	 \Omega_{\Lambda }&\coloneqq V_{\Lambda } \otimes_{\mathbb{F}_p} k  \stackrel{\sim}{\rightarrow} \Lambda_{\breve{\Z}_p}^\vee / \Lambda_{\breve{\Z}_p},\\
    \Omega_{\Lambda^\vee}&\coloneqq V_{\Lambda^\vee} \otimes_{\mathbb{F}_p} k  \stackrel{\sim}{\rightarrow} \Lambda_{\breve{\Z}_p} /p \Lambda_{\breve{\Z}_p}^{\vee}
  \end{align*} 
	with   Frobenius operator  induced by $\Phi$.   
	 
For a quadratic space $\Omega$ over $k$, we let $\mathrm{OGr}(i,\Omega)$ be the orthogonal Grassmannian that parametrizes $i$-dimensional isotropic subspaces of $\Omega$. Let  $S_{\Lambda}^{[h]}$ be the reduced closed subscheme of $\mathrm{OGr}(\frac{t_\Lambda-h}{2},\Omega_\Lambda)$ whose $k$-points are described as follows
	\begin{align*}
		S_{\Lambda}^{[h]}(k) & = \{  \mathscr{V} \in \mathrm{OGr}(\frac{t_\Lambda-h}{2},\Omega_\Lambda)(k)\mid  \operatorname{dim}_k (\mathscr{V}+\Phi(\mathscr{V}))\le \operatorname{dim}_k (\mathscr{V})+1 \} \\
		& \stackrel{\sim}{\rightarrow} \{\text { special lattices } M \subset V_{\brQ_p}: \Lambda_{\breve{\Z}_p} \subset M \subset   M^\vee \subset \Lambda_{\breve{\Z}_p}^\vee\}.
	\end{align*}

Similarly, let $R_{\Lambda}^{[h]}  $ be the reduced closed subscheme of $\mathrm{OGr}(\frac{h-t_\Lambda}{2},\Omega_{\Lambda^\vee})$  with $k$-points
	\begin{align*}
		R_{\Lambda}^{[h]}(k) & = \{   \mathscr{V}  \in \mathrm{OGr}(\frac{h-t_\Lambda}{2},\Omega_{\Lambda^\vee})(k) \mid  \operatorname{dim}_k (\mathscr{V}+\Phi(\mathscr{V}))\le \operatorname{dim}_k (\mathscr{V})+1  \} \\
		& \stackrel{\sim}{\rightarrow} \{\text { special lattices } M \subset V_{\brQ_p}: p\Lambda_{\breve{\Z}_p}^\vee \subset p M^\vee \subset M \subset \Lambda_{\breve{\Z}_p}\}.
	\end{align*}
We drop the superscript when there is no confusion about the implicit $h$.
\begin{theorem}\label{thm: intro 2}
    For $\Lambda \in \cV^{\ge h}$, we have the following isomorphism between $k$-schemes
    \begin{align*}
         \cZ_{\red}^{(\ell)}(\Lambda)\cong S_\Lambda.
    \end{align*}
Similarly, for $\Lambda \in \cV^{\le h}$, we have
 \begin{align*}
      \cY_{\red}^{(\ell)}(\Lambda^\vee)\cong R_\Lambda.
    \end{align*}
Here $S_\Lambda$ and $R_\Lambda$ are irreducible of dimension $\frac{t_\Lambda+h}{2}-1$ and  $n-\frac{t_\Lambda+h}{2}-1$ respectively. Moreover,  $S_\Lambda$ and $R_\Lambda$ are unions of certain generalized Deligne-Lusztig varieties.
\end{theorem}
We refer to Theorem \ref{thm: decom pf SLambda} for the details about the identification between $S_\Lambda$ and generalized Deligne--Lusztig varieties. In fact, we only need to prove the first statement concerning $S_\Lambda$ in the above theorem since the second statement will follow from the first as follows.

Let $L^\sharp$ denote the quadratic lattice $(L^\vee)$ with quadratic form $q(\,)_{L^\sharp}=p\cdot q(\,)$. Note that if $L$ is a vertex lattice of type $h$, then $L^\sharp$ is a vertex lattice of type $n-h$. In particular, one can consider the corresponding RZ space $\cN_{L^\sharp}$ and the corresponding special cycles on $\cN_{L^\sharp}$. Let  $\widetilde{\pi}=\sqrt{p}$ and set $E=\bbQ_{p^2}(\widetilde{\pi})$. Then we may identify the corresponding space of quasi-endormorphisms $\bbV^\sharp$ with the subspace  $\widetilde{\pi}\bbV\subset\bbV_E$.
Thus for $\Lambda\subset \bV$ a vertex lattice of type $t$,  $\Lambda^\sharp:=\widetilde{\pi}\Lambda$ is a vertex lattice in $\bbV^\sharp$ of type $n-t$, and we have the associated special cycle $\cZ_{L^\sharp}(\Lambda^\sharp)$ of $\cN_{L^\sharp}$. Here the subscript is used to emphasize that we regard $\cZ_{L^\sharp}(\Lambda^{\sharp})$ as a $\cZ$-cycle of $\cN_{L^\sharp}$.

Motivated by some group theoretic considerations and the unitary case considered in \cite{cho2018basic}, we have the following (see Proposition \ref{prop: duality for Z(Lambda)}).
\begin{proposition}\label{prop: intro} There is a natural isomorphism $$\Upsilon_{L,L^\sharp}:\calN_L\xrightarrow{\sim}\calN_{L^\sharp}$$
such that for any non-degenerate lattice $\Lambda\subset \bV$, $\Upsilon_{L,L^\sharp}$ induces an isomorphism
 \begin{align*}
        \cY_L(\Lambda^\vee)= \cZ_{L^\sharp}(\Lambda^\sharp).
    \end{align*} 
\end{proposition}
Similarly, one can show there is a natural isomorphism between $R_\Lambda^{[h]}$  and $S_{\Lambda^\sharp}^{[n-h]}$. Combining this with the above proposition, one can deduce the description of the $\calY$-cycles on $\calN_L$ from the description of the $\calZ$-cycles on $\calN_{L^\sharp}.$

We believe Proposition \ref{prop: intro} should be of independent interest and have further applications. For example, it should be useful for formulating the general Kudla-Rapoport conjecture for the RZ space considered here (cf. \cite{cho2022special}). It also manifests a particular symmetry that should play a vital role in many arithmetic geometry problems (cf. \cite{ZhiyuZHang,CHZ}). 

 In our setting, the formulation of this duality isomorphism is complicated by  the lack of a moduli interpretation. It should be possible to show the existence of $\Upsilon_{L,L^\sharp}$ by showing that $\cN_L$ and $\cN_{L^\sharp}$ represent the same $v$-sheaf following the method in \cite{PRLSV}.  However, it seems difficult to prove the compatibility with special cycles in this approach. Instead, we use a more direct approach for which the compatibility with special cycles follows essentially from the construction. A key observation to make this work is that in the construction of $\calN_L$, we need to use the local Hodge embedding given by the composition:
 $$\mathrm{GSpin}(V)\rightarrow \mathrm{\GL}(C(V))\rightarrow \Res_{E/\bbQ_p}\GL(C(V)_E)\rightarrow \GL(W).$$
Here $C(V)$ is the Clifford algebra associated to $V$, $C(V)_E$ its base change to $E$ and $W$ is the $\bbQ_p$-vector space underlying $C(V)_E$, with the last map being induced by the forgetful functor from $E$-vector spaces to $\bbQ_p$-vector spaces. It is interesting to note that this enlarged Hodge embedding also appears in relation to the notion of \emph{very good} Hodge embedding in the sense of \cite[Definition 5.2.5]{KPZ}, which is needed to apply the deformation theoretic results of \cite{KisinPappas} and \cite{KPZ} for our proof of Theorem \ref{thm: intro 2}.

For the reader's convenience, we summarize all the main results in Theorem \ref{thm: main thm}. 

In the following table, we list some previously studied cases and some PEL type Shimura varieties that our results specialize to.   
\begin{align*}
    \begin{array}{c|c}
\hline \text { PEL type special cases and previously studied cases} & \text {Isometry class of  } L \\
\hline \text{Modular curve with level $\Gamma_0(pN)$} \quad (p \nmid N) \text {  } & L=H_2^+\obot \Lambda_1  \\
\text { Shimura curve over a ramified prime \cite[\S 2]{KRshimuracurve} } & L=H_2^-\obot \Lambda_1   \\
\text{Product of Modular curve with level $\Gamma_0(pN)$}\quad (p \nmid N) \text {  } & L=H_2^+\obot \Lambda_2^+\\
\text{Product of  Shimura curve over a ramified prime }  \text {  } & L=H_2^-\obot \Lambda_2^-\\
\text{Hilbert modular surface with Iwahori level \cite{Stamm}} & L=H_2^-\obot \Lambda_2^+\\
\text {Hilbert modular surface over a ramified prime \cite{BachmatGoren}  } &   L=H_3\obot \Lambda_1\\
\text {Siegel threefold with paramodular level \cite{YuParamodular}, \cite[\S 7]{Wang} } & L=H_4^+\obot \Lambda_1 \\
\text {Quaternionic unitary Shimura variety \cite{OkiQuaternionic}, \cite[\S 5]{Wang}] } & L=H_4^-\obot \Lambda_1   \\
\text {GU(2,2) Shimura variety over a ramified prime \cite{OkiGU22}   } & L=H_5\obot \Lambda_1 \\
\text {Hyperspecial GSpin Shimura variety \cite{HowardPappas}} & L=H_n \\
\text {GSpin Shimura variety with almost self-dual level \cite{Oki}} & L=H_{n-1}\obot \Lambda_1 \\
\hline
\end{array}
\end{align*}
Here $H_i$ denotes a self-dual quadratic lattice over $\Z_p$ of rank $i$. When $i$ is even, $H_i^+$   denotes the quadratic lattice that contains a rank $\frac{i}{2}$ totally isotropic sublattice and $H_i^-$ denotes the one which does not contain such sublattice. Similarly,   $\Lambda_i$ denotes a $p$-modular quadratic lattice over $\Z_p$. When $i$ is even, $\Lambda_i^+$ (resp.  $\Lambda_i^-$)  is the one which contains (does not contains) a totally isotropic sublattice of rank $\frac{i}{2}$.    See \S \ref{sec: collection of background} for background about quadratic lattices.

\subsubsection{Proof strategy and organization of the paper}
We now discuss the proof of the main theorem. The starting point is to obtain an explicit description of $\cN_L(k)$. In the unitary case,  one can describe the $k$-points of the RZ space as a set of lattices satisfying certain properties via Dieudonn\'e theory. In our non-moduli setting, we use a more indirect approach by leveraging the relation between RZ-spaces and affine Deligne-Lusztig varieties. More precisely, we identify $\cN_L^{(\ell)}(k)$ with the $k$-points of a  relevant affine Deligne-Lusztig variety parametrized by the corresponding (parahoric) $\mu$-admissible set $\Adm(\mu)_J.$ The description of the $k$-points then boils down to obtaining an explicit description of $\Adm(\mu)_J$ in terms of relative positions of lattices.

The key to showing the decomposition in part $(3)$ of Theorem \ref{thm: intro 1} is Proposition \ref{prop: key prop}, which is an analogue of the ``crucial lemma" in many previously studied cases. More precisely, we show that given a special lattice $M$ of type $h$, then either $\Lambda \subset M$ for a vertex lattice $\Lambda\subset \bV$ of type $\ge h$ or  $\Lambda^\vee \subset M^\vee$ for a vertex lattice $\Lambda\subset \bV$ of type $\le h$. In fact, this notion of special lattice also applies to the lattices used to describe the $k$-points of many other RZ spaces including the unramified/ramified unitary RZ spaces with maximal parahoric level structures. Since our argument is a general argument that works for any special lattices, it can be applied to obtain the crucial lemmas for these cases too (see \cite{HLY} for an application in the ramified unitary case with vertex level).

The geometric properties of $\cZ_{\red}^{(\ell)}(\Lambda)$ and $\cY_{\red}^{(\ell)}(\Lambda^\vee)$ including dimension and irreducibility often can be reduced to the ones of a generalized Deligne-Lusztig variety once we can prove the isomorphism $ \cZ_{\red}^{(\ell)}(\Lambda)\cong S_\Lambda$ (equivalently   $\cY_{\red}^{(\ell)}(\Lambda^\vee)\cong R_\Lambda$) as in Theorem \ref{thm: intro 2}. First we construct a morphism $\Psi_{\cZ,\Lambda}$ (see Proposition \ref{prop: morphism BT strata to DLV}) from $\cZ_{\red}^{(\ell)}(\Lambda)$ to $S_\Lambda$ that induces a bijection on $k$-points. When $L$ is a self-dual lattice, \cite{HowardPappas} proves $\Psi_{\calZ,\Lambda}$ is an isomorphism using  a result of Nisnevich \cite[Theorem 5.2]{Nisnevich} on a special case of a conjecture of Grothendieck. The analogue of this result is not known for parahorics, so instead we use a different method inspired by an alternative proof in the unitary case (see \cite{LTXZZ}). The idea is to  show that, in addition to being a bijection on $k$-points, $\Psi_{Z,\Lambda}$ also induces  isomorphisms on tangent spaces at closed points. In order to carry this out, we need to have an explicit description of the  tangent space to the Shimura variety at closed points in the special fiber. This seems difficult to prove in general since the Hodge embedding has a very large dimension. Nevertheless, it turns out that using the deformation theoretic methods developed in \cite{KisinPappas} and \cite{KPZ} (in particular the notion of very good Hodge embedding mentioned above), we can obtain such a description, at least over the smooth locus of the Shimura variety. Applying the argument for this smooth locus shows that the restriction of $\Psi_{\calZ,\Lambda}$ to an open dense subspace $\calZ^\circ_{\red}(\Lambda$) is an isomorphism. An abstract argument using the normality of $S_\Lambda$  then implies that $\Psi_{\calZ,\Lambda}$ is an isomorphism.

We now explain the organization of the paper. \S2 contains mainly background and recollections on quadratic spaces and vertex lattices. In \S3, we describe the admissible set in our setting in terms of relative positions of lattices, which is then used in \S4 to study the $k$-points of  the relevant affine Deligne--Lusztig variety. In \S5 we introduce the varieties $S_\Lambda$ and $R_\Lambda$ and  relate them to classical Deligne--Lusztig  varieties. In $\S6$ we introduce integral models of GSpin Shimura varieties and the associated Rapoport--Zink space, and recall the deformation theoretic results from \cite{KisinPappas} and \cite{KPZ} needed in the sequel. Finally in \S7 we introduce  special cycles on the Rapoport--Zink space, and prove the relationship with the varieties $S_\Lambda$ and $R_\Lambda$ using the approach outlined above. In the last subsection \S7.6, we summarize the main results and give the application to the basic locus of Shimura varieties.

\emph{Acknowledgements:} The authors would like to thank Xuhua He, Chao Li, Yifeng Liu, Yu Luo, Andreas Mihatsch,   Yousheng Shi, Qingchao Yu, Wei Zhang, and Baiqing Zhu for helpful discussion and comments. Q.H. is supported by the AMS-Simons Travel Grant Program.
R.Z. was supported by the Engineering and Physical Sciences Research Council grant no. EP/Y030648/1.

\section{Collection of background on quadratic lattice and vertex lattices}\label{sec: collection of background}

In this section, we collect some results about quadratic lattices and vertex lattices that will be needed later.

\subsection{Notations on quadratic lattices}
\subsubsection{} Let $p$ be an odd prime and $F$ be a local field  over $\Q_p$.  We use $O_F$ to denote the ring of integers of $F$ and $k_F$ its residue field. Let $\breve F$ denote the completion of the maximal unramified extension $F^{\mathrm{ur}}$ of $F$,  $W=O_{\breve{F}}$ its ring of integers, and $k$ the residue field of $W$. 
  We always fix  a uniformizer $\pi$ of $F$. We let $V$ be a quadratic space over $F$ of dimension $n$ with non-degenerate symmetric form $(\,,\,)$ and quadratic form $q(\,)$.  

We define the determinant of $V$ to be
$$
\operatorname{det}(V):=\operatorname{det}\left(\left(x_i, x_j\right)_{1\le i, j\le n}\right) \in F^{\times} /\left(F^{\times}\right)^2
$$
where $\left\{x_1, \ldots, x_n\right\}$ is a basis of $V$, and the discriminant of $V$   to be
$$
\operatorname{disc}(V):=(-1)^{\binom{n}{2}} \cdot \operatorname{det}(V) \in F^{\times} /\left(F^{\times}\right)^2.
$$

Let $\operatorname{val}(V)=\operatorname{val}(\operatorname{det}(V))\in \Z/2\Z$ and $\chi=(\frac{\cdot}{\pi})_{F}: F^{\times} /\left(F^{\times}\right)^2 \rightarrow\{ \pm 1,0\}$ be the quadratic residue symbol. Thus $\chi(\alpha)=0$  if  $\val(\alpha)$ is odd, and $\chi(\alpha)$ is identified with the image of $\alpha/\pi^{\val(\alpha)}$ under the isomorphism $O_F^\times/ (O_F^\times)^2 \cong \{\pm\}$ if   $\val(\alpha)$ is even. Define
$$
\chi(V):=\chi(\operatorname{disc}(V)) \in\{ \pm 1,0\},\quad 
\chi'(V):=\chi(\pi \operatorname{disc}(V)) \in\{ \pm 1,0\}.
$$
Then $\disc(V)$ is determined by $\chi(V)$ and $\chi'(V)$. Note that only one of $\chi(V)$ and $\chi'(V)$ is nontrivial.

Let $\{v_1,\ldots,v_n\}$ be an orthogonal basis of $V$. We define the Hasse invariant of $V$ to be
\begin{align*}
    \epsilon(V)\coloneqq \prod_{1\le i< j\le n} (q(v_i),q(v_j))_{F}\in \{\pm 1\}.
\end{align*}
Here $(\, , \,)_{F}$ is the Hilbert symbol.
It is well-known that quadratic spaces over $F$ are classified by its dimension, discriminant and Hasse invariant.

For a quadratic lattice $L\subset V$ of rank $n$, we   define $\chi(L)=\chi(V),\chi'(L)=\chi'(V)$ and $\epsilon(L)=\epsilon(V)$.

\subsection{An explicit classification of quadratic spaces over $F$}
In this subsection, we discuss all the possible isomorphism classes of vertex lattices. 
	\subsubsection{}

 Let $V$ be a quadratic space over $F$.
 Consider the following decomposition of $V$:
  \begin{equation*}
      V=V_0\obot V_{\an}=\spa_{F}\{e_1,f_1,\ldots,e_r,f_r\}\obot V_{\an},
  \end{equation*}
where $(e_i,e_j)=(f_i,f_j)=0$, $(e_i,f_j)=\delta_{ij}$, $V_{\an}$ is anisotropic, and $\obot$ denotes the  orthogonal sum. Note that $\dim_{F}(V_{\an})\le 4$ since any quadratic space over $F$ of dimension at least $5$ contains an isotropic vector. 

Let $L\subset V$ be a lattice. We let $L^\vee=\{x\in L\otimes_{O_F} F\mid (x,y) \in O_F, \, \forall y \in L\}$ denote the dual lattice. We say that  $L$ is \emph{a vertex lattice} if we have $\pi L \subset L \subset L^{\vee}$.  If $L$ is a vertex lattice, its type $t=t(L)$ is defined to be the dimension of the $k_F$-vector space $L^\vee/L.$ Note that our convention of vertex lattice is different with the one used in \cite{HowardPappas}. The dual lattice of a vertex lattice in  \cite{HowardPappas} is our vertex lattice.

We will describe all the possible $V_{\an}$ and the vertex lattices in $V_{\an}$ following \cite[\S 5.1]{HowardPappas}.

\subsubsection{}
If $\operatorname{dim}\left(V_{\an}\right)=1$, then it is isomorphic to $F$ with the quadratic form $q(x)=d x^2$ for some $d \in F^{\times} /\left(F^{\times}\right)^2$. So there are four cases. Then $V$ contains a self-dual lattice if and only if $\val(d)$ is even. If we write $d=u\pi^{\val(d)}$, then $\chi(V)=\chi(u)$. The vertex lattice in $V_{\an}$ is $\spa_{O_F}\{x\}$ for $\val(q(x))=0$ (resp. $1$) if $\val(d)$ is even (resp. odd).

If $\operatorname{dim}\left(V_{\an}\right)=2$, then we can identify $V_{\an}$ with a quadratic extension $E/F$  whose quadratic form is given by $q(x)=d\cdot \mathrm{Nm}(x)$ for some $d \in E^{\times} / \mathrm{Nm}\left(E^{\times}\right)$. Without loss of generality, we can choose $d$ such that $\val(d)=0$ or $1$ when $E$ is unramified and $d\in (O_F^\times)^2$ or $d\in O_F^\times \setminus (O_F^\times)^2$ when $E$ is ramified. If $E$ is an unramified extension, then the vertex lattice in $V_{\an}$  has moment matrix $\diag(u_1 d, u_2 d)$ such that $\chi(-u_1u_2)=-1$. If $E$ is a ramified extension, then the vertex lattice in $V_{\an}$  has moment matrix $\diag(d , -d \pi )$.

If $\operatorname{dim}\left(V_{\an}\right)=3$, there are exactly four anisotropic quadratic spaces over $F$ of dimension 3. Let $B$ denote the quaternion division algebra over $F$  and  $\bar{x}$  denote the main involution of $B$. Then the subspace of  elements $B^0=\{x \in B: x+\bar{x}=0\}$ is a $3$-dimensional quadratic space with quadratic form given by the reduced norm on $B^0$. More explicitly,
\begin{align}\label{eq: B^0}
    \left(B^0, \mathrm{Nrd}\right) \stackrel{\sim}{\rightarrow}\left(F^3,-u x_1^2-\pi x_2^2+u \pi x_3^2\right)
\end{align}
for any nonsquare $u \in O_F^{\times}$. Now the four possible $V_{\an}$ of dimension $3$ are the spaces $\left(B_0, d\cdot \mathrm{Nrd}\right)$ with $d \in F^{\times} /\left(F^{\times}\right)^2$. Again, we assume $\val(d)=0$ or $1$ without loss of generality. Then according to \eqref{eq: B^0}, when $\val(d)$ is even, a vertex lattice has moment matrix given by $\diag(-ud,-d\pi,ud\pi )$. When $\val(d)$ is odd, a vertex lattice   has moment matrix given by $\diag(-ud,-\pi^{-1}d,u\pi^{-1}d)$.
In particular, $V_{\an}$ can never contain self-dual lattices when $\dim(V_{\an})=3$.

If  $\operatorname{dim}\left(V_{\an}\right)=4$,   then $V\stackrel{\sim}{\to} B $ with quadratic form given by the reduced norm.  In this case, a vertex lattice has $\diag(u_1,u_2,u_3\pi,u_4\pi)$ as moment matrix where $u_i \in O_F^\times$ and $\chi(-u_1u_2)=\chi(-u_3u_4)=-1$.
  	
\subsubsection{}

Combining the above discussion, we summarize all the possible  vertex lattices of maximal type and minimal type up to isomorphism in the following proposition.  We set 
\begin{align*} 
    H_{n}^{\pm}\coloneqq \text{a self-dual lattice of rank $n$ with $\chi(H_{n}^\pm)=\pm 1$}.
\end{align*} 
We also use the same notation to denote the corresponding moment matrix. We set 
\begin{align*} 
    \Lambda_n^{\pm}\coloneqq \text{a $\pi$-modular lattice with moment matrix $\pi\cdot H_{n}^\pm$}.
\end{align*} 

\begin{proposition}\label{prop: class of V}
    Let $V$ be a quadratic space over $F$ of rank $n$.  Let $\Lambda_{\max}\subset V$  and  $\Lambda_{\min}\subset V$ be the  vertex lattices of maximal and minimal type. Then $\Lambda_{\max}\subset V$  and  $\Lambda_{\min}\subset V$ are of the following forms.
\begin{enumerate}
    \item If  $\operatorname{dim}\left(V_{\an}\right)=0$, then 
     \begin{align*}
         \Lambda_{\max}=\Lambda_{n}^+  \quad \text{ and } \quad \Lambda_{\min}=H_{n}^+.    
    \end{align*}
    \item If $\operatorname{dim}\left(V_{\an}\right)=1$, then
    \begin{align*}
       \begin{cases}
            \Lambda_{\max}= \Lambda_{n-1}^+\obot H_{1}^{\chi} \quad &\text{and}  \quad \Lambda_{\min}=H_{n}^\chi,\\
           \Lambda_{\max}=\Lambda_{n}^\chi \quad &\text{and}  \quad    \Lambda_{\min}=H_{n-1}^+\obot \Lambda_1^\chi.
        \end{cases} 
    \end{align*}
    \item If $\operatorname{dim}\left(V_{\an}\right)=2$, then 
      \begin{align*}
        \begin{cases}
          \Lambda_{\max}=\Lambda_{n-2}^+\obot H_2^-    \quad &\text{and}  \quad \text{   $\Lambda_{\min}=H_n^{-}$ },\\
           \Lambda_{\max}= \Lambda_{n}^-  \quad &\text{and}  \quad \text{ $\Lambda_{\min}=H_{n-2}^{+}\obot \Lambda_2^-$},\\
            \Lambda_{\max}= \Lambda_{n-1}^{\chi_1} \obot H_1^{\chi_2}  \quad &\text{and}  \quad  \text{ $\Lambda_{\min}=H_{n-1}^{\chi_2}\obot \Lambda_1^{\chi_1}$}.
        \end{cases} 
    \end{align*}
    \item If $\operatorname{dim}\left(V_{\an}\right)=3$, then 
    \begin{align*}
         \begin{cases}
           \Lambda_{\max}= \Lambda_{n-1}^-\obot H_1^{-\chi}    \quad &\text{ and } \quad \text{  $\Lambda_{\min}=H_{n-2}^{-\chi }\obot \Lambda_2^{-}$ },\\
            \Lambda_{\max}= \Lambda_{n-2}^{-\chi}\obot H_2^{-}    \quad &\text{ and } \quad \text{ $\Lambda_{\min}=H_{n-1}^{- }\obot \Lambda_1^{-\chi}$}.
        \end{cases} 
    \end{align*}
     \item If $\operatorname{dim}\left(V_{\an}\right)=4$, then 
     \begin{align*}
         \Lambda_{\max}=\Lambda_{n-2}^- \obot H_2^- \quad \text{ and } \quad \Lambda_{\min}=H_{n-2}^-\obot \Lambda_2^-. 
    \end{align*}
\end{enumerate}
Here, if $H_{n-t}^{\chi_0}\obot \Lambda_t^{\chi_1}\subset V$, then 
\begin{align}\notag
    \chi(V)&=\begin{cases}
       \chi_0\chi_1  & \text{ if $t$ is even},\\ 
      0   & \text{ if $t$ is odd},
    \end{cases}\\ \notag
    \chi'(V)&= \chi(-1)^{n-1}\cdot\begin{cases}
      0 & \text{ if $t$ is even},\\
     \chi_0\chi_1   & \text{ if $t$ is odd},
    \end{cases}\\ \label{eq: Hasse}
    \epsilon(V)&=   \begin{cases}
\chi_1& \text{ if $t$ is even},\\
  \chi((-1)^{\frac{(n-1)(n-2)}{2}})  \chi_0 & \text{ if $t$ is odd}.
    \end{cases}
\end{align}

\qed
\end{proposition}

\subsection{An explicit classification of quadratic spaces over $\breve{F}$}\label{sec: quadratic space over breveF}
\subsubsection{}Note that $O_{\brF}^\times=(O_{\brF}^\times)^2$. In particular, if $V=E$ with $q(x)=\mathrm{Nm}(x)$ where $E/F$ is an unramified quadratic extension, then $\brV=V\otimes_{F}\brF$ becomes isotropic.  Combining this with the discussion in the previous section, we have the following classification of quadratic spaces over $\brF$.

 Let     $\brV=\brV_0\obot\brV_{\an}$ be a quadratic space over $\brF$, where $$\brV_0=\langle e_1,\dotsc,e_n,f_1,\dotsc,f_n\rangle$$ with $(e_i,e_j)=(f_i,f_j)=0$, $(e_i,f_j)=\delta_{ij}$, and $\brV_{\an}$ is anisotropic. Then  $\brV_{\an}$ is of the following form.
\begin{enumerate}
	\item[(1)] $\brV_{\an}=0$.
	\item[(2a)] $\brV_{\an}=\brF u$, with $(u,u)=1$.
	\item[(2b)] $\brV_{\an}=\brF u$, with $(u,u)=\pi$.
	\item[(3)] $\brV_{\an}=\brF u\oplus \brF v$, with $(u,u)=\pi$, $(v,v)=1$.
 \end{enumerate}

	\section{Relative positions of vertex lattices}

 \subsection{The $\mu$-admissible set}
\subsubsection{}
 Let $G$ be a reductive group over $F$. We set  $\Gamma=\Gal(\bar{F} /F)$ and let $\sigma\in \Gamma$ denote a lift of the Frobenius of $k$ over $k_F$.   Let $S$ be a maximal $\brF$-split torus of $G$ and $T$ its centralizer (cf. \cite[1.10]{Tits} for the existence of $S$). By Steinberg's theorem, $G$ is quasi-split over $\brF$ and $T$ is a maximal torus of $G$. The relative Weyl group over $\brF$ and the Iwahori Weyl group  of $G$ are defined as 
 $$\tW_0=N(\brF)/T(\brF),\ \ \ \tW=N(\brF)/\calT_0(\calO_{\brF})$$
where $N$ is the normalizer of $T$ and $\calT_0$ is the connected N\'eron model for $T$. We have an exact sequence 
\begin{equation}\label{eqn: exact sequence IWG}\xymatrix{0\ar[r] &X_*(T)_I \ar[r]& \tW\ar[r]&\tW_0\ar[r] &0}\end{equation}
Here, $I = \mathrm{Gal}(\bar{F}/F^{\mathrm{ur}})$ denotes the inertia subgroup of the absolute Galois group of $F$, and $X_*(T)_I$ represents the group of coinvariants of the cocharacter lattice $X_*(T)$ under the action of $I$. For an element $\lambda\in X_*(T)_I$, we write $t^\lambda$ for the corresponding element in $\tW$; such elements are called translation elements.

Let $\calB(G,F)$ (resp. $\calB(G,\brF)$) be the (extended) Bruhat-Tits building for $G$ over $F$ (resp. $\brF$), and fix $\fka\subset \calB(G,F)$ a $\sigma$-stable alcove; this determines an Iwahori group scheme $\calI$ defined over $F$. Fix also a special vertex $\fks\in \calB(G,\brF)$ lying in the closure of $\fka$, which determines a splitting of the exact sequence \eqref{eqn: exact sequence IWG} and gives an identification 
 $$ A:=X_*(T)_I\otimes_{\bbZ}\bbR\cong \calA(G,S,\brF),$$ where $\calA(G,S,\brF)\subset \calB(G,\brF)$ is the apartment corresponding to $S$. This identification determines a chamber $C^+\subset X_*(T)_I\otimes_{\bbZ}\bbR$, namely the one  containing $\fka$. We write $B$ for the corresponding Borel subgroup defined over $\brF$.

Let $\bbS$ denote the set of simple reflections about the walls of $\fka$. We let $\tW_a$ denote the affine Weyl group; it is the subgroup of $\tW$ generated by the reflections in $\bbS$. Then we have an exact sequence 
\begin{equation}\label{eqn: exact sequence AWG}\xymatrix{0\ar[r] &X_*(T_{\mathrm{sc}})_I \ar[r]& \tW_a\ar[r]&\tW_0\ar[r] &0}
\end{equation}
where $T_{\mathrm{sc}}$ is the preimage of $T$ in the simply connected cover of the derived group of $G$. The group $(\tW_a,\bbS)$ has the structure of a Coxeter group, and hence we have a notion of length and Bruhat order. The Iwahori Weyl group and affine Weyl group are related by the following exact sequence

\[\xymatrix{0\ar[r]& \tW_a \ar[r] &\tW\ar[r] &\pi_1(G)_I\ar[r] &0.}\]
The choice of $\fka$ induces a splitting of this exact sequence and  $\pi_1(G)_I$ can be identified with the subgroup $\Omega\subset \tW$ consisting of elements which map $\fka$ to itself. The length function $\ell$ and Bruhat order extend to $\tW$ via the splitting.

By \cite{HR}, there is a reduced root system $\Sigma$ such that $$\tW_a\cong Q^\vee(\Sigma)\rtimes \tW(\Sigma),$$ where $Q^\vee(\Sigma) $ and $\tW(\Sigma)$  denote the coroot lattice and Weyl group of $\Sigma$ respectively. The root system $\Sigma$ is known as the \textit{\'echelonnage} root system.  

\subsubsection{} Let $X_*(T)_I^+$ denote the set of dominant elements with respect to $B$. For $\mu\in X_*(T)^+_I$, we have the  $\mu$-admissible set  which is defined as follows
$$
\Adm(\mu):=\left\{w \in \widetilde{W} \mid w \leq t^{\mu^{\prime}} \text { for some } \mu^{\prime} \in \tW_0 \mu\right\}.
$$
This set has a minimal element $\tau_\mu$, which is  the unique length zero element of $\Adm(\mu).$

Let $J\subset \bbS$ be a subset. We write $\tW_J\subset \tW$ for the subgroup generated by the reflections in $J$. If $\tW_J$ is finite, $J$ determines a facet $\fkf_J$ lying in the closure of $\fka$ and hence corresponds to a parahoric group scheme $\calG^\circ_J$. The corresponding parahoric $\mu$-admissible set  
$$\Adm(\mu)_J\subset \tW_J\backslash \tW/\tW_J$$ 
is defined to be the image of $\Adm(\mu)$ under the map  $\tW\rightarrow \tW_J\backslash \tW/\tW_J$. We also set $$\Adm(\mu)^J:=\tW_J\Adm(\mu)\tW_J\subset \tW.$$  
\subsection{Relative position of vertex lattices}
\subsubsection{}\label{subsec: VL cases}
 Let $V$ be a quadratic space over $F$. We now specialize to the case $G=\SO(V).$
  We let $\brV=V\otimes_{F}\brF$. Then we can write $\brV=\brV_0\obot\brV_{\an}$, where $$\brV_0=\langle e_1,\dotsc,e_n,f_1,\dotsc,f_n\rangle$$ with $(e_i,e_j)=(f_i,f_j)=0$, $(e_i,f_j)=\delta_{ij}$, and $\brV_{\an}$ is  of the following form:
\begin{enumerate}
	\item[(1)] $\brV_{\an}=0$.
	\item[(2a)] $\brV_{\an}=\brF u$, with $(u,u)=1$.
	\item[(2b)] $\brV_{\an}=\brF u$, with $(u,u)=\pi$.
	\item[(3)] $\brV_{\an}=\brF u\oplus \brF v$, with $(u,u)=\pi$, $(v,v)=1$.
	
\end{enumerate}
The local Dynkin diagram is of type $D_n$ in case (1), type $B_n$ in case (2a) and (2b), and of type $C$-$B_n$ in case (3); here we use the notation of Tits' table \cite[\S4]{Tits}.

Let $L_{\an}\subset \brV_{\an}$ denote the $\calO_{\brF}$-submodule $$L_{\an}=\begin{cases}
0 &\text{ in Case (1)},\\
\calO_{\brF}u& \text{ in Case (2)},\\
\calO_{\brF}u+\calO_{\brF}v& \text{ in Case (3)}.\end{cases}$$ For $s=0,\dotsc,n$, we let $L_0^{(s)}\subset\brV_0$ be the $\calO_{\brF}$-module given by $$L_0^{(s)}=\Span_{\calO_{\brF}}(\pi e_1,\dotsc,\pi e_s,e_{s+1},\dotsc,e_n,f_1,\dotsc,f_n)$$ and set $$L^{(s)}=L_{0}^{(s)}\obot L_{\an}\subset \brV.$$
Then $L^{(s)}$ is a vertex lattice of type $$t(s):=\begin{cases} 2s & \text{ in Case (1) and (2a)},\\
2s+1 &\text{ in Case (2b) and (3)}.
\end{cases}$$

We also set $$L'^{(0)}=\Span_{\calO_{\brF}}(\pi e_1,e_2,\dotsc, e_n,\pi^{-1}f_1,\dotsc,f_n)\obot L_{\an}$$
	$$L'^{(n)}=\Span_{\calO_{\brF}}(\pi e_1,\dotsc,\pi e_{n-1}, e_n,f_1,\dotsc,f_{n-1},\pi f_n)\obot L_{\an}$$which are vertex lattices of type $t(0)$ and $t(n)$ respectively.
\subsubsection{}\label{subsec: apartment basis}

We  take $S$ to be the maximal $\brF$-split torus of $\SO(\brV)$ given  by $\bbG_m^n\subset\SO(\brV)$  such that for $t=(t_1,\dotsc,t_n)\in \bbG_m^n$, we have $$t:\brV \mapsto \brV, \text{ is given by  } e_i\mapsto t_ie_i,\ f_i\mapsto t_i^{-1}f_i.$$ Let $\nu:\bbG_m\rightarrow \bbG_m^n$ be the cocharacter $t\mapsto (t,1,\dotsc,1)$, and let $\mu:\bbG_m\rightarrow \SO(\brV)$ denote the composition of $\nu$ and the above inclusion. Then $\mu$ is a minuscule cocharacter of $\SO(\brV)$. We then have the associated $\mu$-admissible set $\Adm(\mu)$ and its parahoric analogues. Our goal is to give a more explicit description of the set $\Adm(\mu)_J$, where $J$ corresponds to a vertex lattice in $\brV$. For this, we use the description of the building of $\SO(\brV)$ in \cite[\S20.3]{GarrettBook}; we translate the necessary results  to our setting.  We first describe the base alcove, and facets corresponding to vertex lattices.

In each case, we have $X_*(T)_I\otimes_{\bbZ}\bbR\cong X_*(S)\otimes_{\bbZ}\bbR$ which we identify with $\calA(G,S,\brF)$ by sending $0$ to the vertex $\bfx^{(0)}$ corresponding to the stabilizer of $L^{(0)}$.  We let $\epsilon_1,\dotsc,\epsilon_n$ denote the standard basis for $\bbZ^n$.

\textit{Case (1):} We have $S=T$ and $X_*(T)_I=X_*(S)$ which we identify with $\bbZ^n$ so that the roots are given by $\pm\epsilon_i\pm\epsilon_j$, $1\leq i< j\leq n$. We take the $n$ roots 
$$ \epsilon_i-\epsilon_{i+1}, 1\leq i\leq n-1, \text{ and }  \epsilon_{n-1}+\epsilon_n$$as simple roots.

The base alcove in $\bbR^n$ is given by $$\fka=\{(a_1,\dotsc,a_n)\in \bbR^n\mid a_1> a_2> \dotsc >\pm a_n, a_1+a_2< 1\}$$
Let $$\bfx^{(0)} =(0,\dotsc,0),\ \  \bfx'^{(0)}=(1,0\dotsc,0),$$
$$\bfx^{(s)}=\sum_{i=1}^s\frac{\epsilon_i}{2}, \text{ for } 1 \leq s \leq n-1,$$
$$\bfx^{(n)}=(\frac{1}{2},\dotsc,\frac{1}{2}),  \ \bfx'^{(n)}=(\frac{1}{2},\dotsc,\frac{1}{2},-\frac{1}{2})$$
which lie in the closure of $\fka$, with $\bfx^{(0)}$ a special vertex.
The points $\bfx^{(0)},\bfx'^{(0)}, \bfx^{(n)},\bfx'^{(n)}$ and $\bfx^{(s)}$ for $2\leq s\leq n-2$ are $0$-dimensional facets, while $\bfx^{(1)}$ and $\bfx^{(n-1)}$  lie on 1-dimensional facets given by the line segments between $\bfx^{(0)},\bfx'^{(0)}$ and $\bfx^{(n)},\bfx'^{(n)}$ respectively.
The parahoric corresponding to $\bfx^{(s)}, 0\leq s\leq  n$ is the connected stabilizer of $L^{(s)}$; similarly $\bfx'^{(0)}$ corresponds to $L'^{(0)}$ and $\bfx'^{(n)}$ corresponds to $L'^{(n)}$.

\textit{Case (2a) and (2b)}: 
 We have $S=T$ and $X_*(T)_I=X_*(S)$ which we identify  with $\bbZ^n$  so that the roots are given by $\pm\epsilon_i\pm\epsilon_j, 1\leq i<j\leq n$, $\pm\epsilon_i, 1\leq i \leq n$. We take the $n$ roots $$ \epsilon_i-\epsilon_{i+1}, 1\leq i\leq n-1, \text{ and }  \epsilon_n$$as simple roots.

We first assume we are in case (2a). The base alcove in $\bbR^n$ is given by $$\fka=\{(a_1,\dotsc,a_n)\in \bbR^n \mid a_1>a_2>\dotsc> a_n> 0, a_1+a_2< 1\}.$$
Let 
$$\bfx^{(0)}=(0,\dotsc,0),\ \  \bfx'^{(0)}=(1,0\dotsc,0),$$
$$\bfx^{(s)}=\sum_{i=1}^s\frac{\epsilon_i}{2}, \text{ for } 1 \leq s \leq n,$$ which lie in the closure of $\fka$, with $\bfx^{(0)}$ a special vertex.
The points $\bfx^{(0)},\bfx'^{(0)}$, $\bfx^{(s)},2\leq s\leq n$ are $0$-dimensional facets, while $\bfx^{(1)}$ lies on the 1-dimensional facet given by the line segment between $\bfx^{(0)}$ and $\bfx'^{(0)}$. The parahoric corresponding to $\bfx^{(s)}, 0\leq s\leq  n$ is the connected stabilizer of $L^{(s)}$, and similarly $\bfx'^{(0)}$ corresponds to $L'^{(0)}$.

Now assume we are in case (2b); in this case the point $\bfx^{(0)}={(0,\dotsc,0)}$ is \emph{not} a special vertex. The base alcove in $\bbR^n$ is given by $$\fka=\{(a_1,\dotsc,a_n)\in \bbR^n \mid \frac{1}{2}>a_1>a_2>\dotsc> a_n, a_{n-1}+a_n> 0\}.$$
Let 
$$\bfx^{(s)}=\sum_{i=1}^s\frac{\epsilon_i}{2}, \text{ for } 1 \leq s \leq n-1,$$ 
$$\bfx^{(n)}=(\frac{1}{2},\dotsc,\frac{1}{2},\frac{1}{2}),\ \  \bfx'^{(n)}=(\frac{1}{2},\dotsc,\frac{1}{2},-\frac{1}{2}),$$ which lie in the closure of $\fka$, with $\bfx^{(n)}$  a special vertex. The point $\bfx^{(n)}, \bfx'^{(n)}$, $\bfx^{(s)}, 1\leq s \leq n-2$  are $0$-dimensional facets while $\bfx^{(n-1)}$ lies on the $1$-dimensional facet given by the line segment between $\bfx^{(n)}$ and $\bfx'^{(n)}.$

\textit{Case (3):}
In this case we have  $T=S\times( \mathrm{Res}_{\brF'/\brF}\bbG_m)^{\mathrm{Nm}=1}$, where $\brF'/\brF$ is a degree 2 ramified extension, and hence $X_*(T)_I=X_*(S)\oplus \bbZ/2\bbZ$. The \'echelonnage root system is of type $C_n$ (cf. \cite[Example 1.16]{Tits}), so we fix an identification $X_*(S)\cong \bbZ^n$ so that the roots are given by $\pm\epsilon_i\pm \epsilon_j, 1 \leq i< j\leq n$, $\pm2\epsilon_i, i=1,\dotsc,n$. We take the $n$ roots $$ \epsilon_i-\epsilon_{i+1}, 1\leq i\leq n-1, \text{ and }  2\epsilon_n$$as simple roots.

 The base alcove in $\bbR^n$ is given by $$\fka=\{(a_1,\dotsc,a_n)\mid\frac{1}{2}>a_1>\dotsc>a_n>0\}.$$
Let $$\bfx^{(s)}=\sum_{i=1}^s\frac{\epsilon_i}{2}, \text{ for } 0 \leq s \leq n,$$ which lie in the closure of $\fka$; these are all $0$-dimensional facets. The parahoric corresponding to $\bfx^{(s)}$ is the connected stabilizer of $L^{(s)}$.
\begin{remark}\label{rem: automorphism case switch}\begin{enumerate}\item
From the description above, we see that the parahoric corresponding to the lattices $L^{(1)}, L^{(n-1)}$ in Case (1),  $L^{(1)}$ in Case (2a), and $L^{(n-2)}$ in Case (2b) are not maximal parahorics.
\item In case (2), the automorphism $$\theta(a_1,\dotsc,a_n)\mapsto (\frac{1}{2},\dotsc,\frac{1}{2})-(a_n,\dotsc,a_1)$$ of $\bbZ^n$ maps the alcove in case (2a) to case (2b), and is equivariant for the actions of the respective Iwahori Weyl groups. This  automorphism is induced by the isomorphism of group schemes $\SO(V)\cong \SO(V^{(\pi)})$, where $V^{(\pi)}$ is the quadratic space whose underlying vector space is $V$, but with quadratic form given by $\pi(\ ,\ ).$
\end{enumerate}
\end{remark}
\subsubsection{}
The relative Weyl group $W_0$ has the following description $$W_0\cong \begin{cases}((\bbZ/2\bbZ)^n\rtimes S_n)^\circ &\text{ in Case (1)},\\
(\bbZ/2\bbZ)^n\rtimes S_n &\text{ in Case (2) and (3)}.
\end{cases}$$
Here, $((\bbZ/2\bbZ)^n\rtimes S_n)^\circ $ denotes the kernel of the homomorphism $$(\bbZ/2\bbZ)^n\rtimes S_n\rightarrow \bbZ/2\bbZ$$ $$((z_1,\dotsc,z_n),\sigma)\mapsto \sum_{i=1}^n z_i.$$

In case (1), (2a) and (3), the groups $(\bbZ/2\bbZ)^n\cong \{\pm1\}^n$ and $S_n$ act on $\bbZ^n$ via factorwise multiplication and permutation of coordinates. In case (2b), they act via the same linear transformations, but centered at $\bfx^{(n)}$. Thus we  see in each case that the $\tW$-orbit of $\bfx^{(s)}$ in $\calA(G,S,\brF)\cong \bbR^n$ is given by \begin{equation}\label{eqn: xs orbit}\tW \bfx^{(s)}=\{(a_1,\dotsc, a_n)\in \bbZ[1/2]^n \mid \text{\ exactly $s$ of the $a_i$ are in $\frac{1}{2}+\bbZ$}\}.\end{equation}
For $y=(a_1,\dotsc,a_n)\in \tW \bfx^{(s)},$ define $$L_0(y)=\Span_{\calO_F}(\pi^{\lceil a_1\rceil} e_1,\dotsc,\pi^{\lceil a_n\rceil} e_n, \pi^{\lceil -a_1\rceil}f_1,\dotsc, \pi^{\lceil -a_n\rceil}f_n)\subset \brV_0$$
and set $L(y)=L_0(y)\obot L_{\an}$. Then in each case, the parahoric associated to $y$ is the connected stabilizer of $L(y)$, and the association $y\mapsto L(y)$ induces a $\tW$-equivariant bijection between the  $\tW$-orbit of $\bfx^{(s)}$ and the $\tW$-orbit of $L^{(s)}$. Here $\tW$ acts on vertex lattices of the form $L(y)$ via $w:L(y)\mapsto \dot{w}L(y)$, where $\dot{w}\in N(\brF)$ is a lift of $w$, noting that $\dot{w}L(y)$ does not depend on the choice of lift $\dot{w}$.

\subsubsection{} Under the identification $$ X_*(T)_I\cong \begin{cases} \bbZ^n & \text{ in Case (1) and (2)},\\
    \bbZ^n\times \bbZ/2\bbZ &\text{ in Case (3).}
\end{cases}$$ the element $\mu$ corresponds to  the element $(1,0,\dotsc,0)$. In particular, its image in $X_*(T)_I\otimes_{\bbZ}\bbR\cong\bbR^n$ is $\epsilon_1.$ 
\begin{lemma}\label{lem: dim=1 implies permissible}Let $s\in \{0,\dotsc,n\}$ and let $y\in \tW \bfx^{(s)}$ with $\dim_{k}(L(y)+L^{(s)})/L^{(s)}= 1$. \begin{enumerate} \item We have $$y=\bfx^{(s)}\pm \epsilon_i, i=1,\dotsc,n$$ or $$y=\bfx^{(s)}\pm\frac{1}{2}\epsilon_j\pm \frac{1}{2}\epsilon _i, 1\leq j\leq s, s+1\leq i\leq n.$$
	
\item Let $J^{(s)}\subset \bbS$ be the subset of simple reflections corresponding to the facet containing $\bfx^{(s)}$. Then there exists $w\in \Adm(\mu)^{J^{(s)}}$ such that
	$wx^{(s)}=y$.

	\end{enumerate}
\end{lemma}
\begin{proof}Using the automorphism $\theta$ from Remark \ref{rem: automorphism case switch} (2), case (2b) can be reduced to case (2a). So we may assume we are in case (1), (2a) or (3), and hence that $\bfx^{(0)}$ is a special point.

(i) 	Let $y=(a_1,\dotsc,a_n)$. Then using the description of $L(y)$ above, we see that exactly one of the integers $$\lceil a_1\rceil-1,\dotsc, \lceil a_s\rceil-1, \lceil a_{s+1}\rceil , \dotsc, \lceil a_n\rceil, \lceil -a_1\rceil,\dotsc, \lceil -a_n\rceil $$ is equal to $1$, and the others are all $\leq 0$.

Now let use write $\bfx_k^{(s)}$ for the $k^{\mathrm{th}}$-coordinate in $\bfx^{(s)}$.	 Suppose  first that $\lceil a_i\rceil-1 =1$ for some $i\in \{1\dotsc,s\}$, then $a_i=\frac{3}{2}$ or $a_i=2$. If $a_i=2$, then since $y\in \tW \bfx^{(s)}$, there exists $j\in \{s+1,\dotsc,n\}$ with $a_j\in \frac{1}{2}+\bbZ$. But then either $\lceil a_j\rceil>0$ or $\lceil -a_j\rceil>0$, which is a contradiction, and hence $a_i=\frac{3}{2}$. Suppose there exists $k\neq i$ with $a_k\neq \bfx^{(s)}_k$. If $k\leq s$, then we must have $a_k=0$ or $1$, otherwise $\lceil a_k\rceil>0$ or $\lceil -a_k\rceil>0$. Then as before, there exists $j\in \{s+1,\dotsc,n\}$ with $a_j\in \frac{1}{2}+\bbZ$. Thus upon replacing $k$ by $j$, we may assume $k>s+1$. In this case, we have $\lceil a_k\rceil>0$ or $\lceil -a_k\rceil <0$, which is a contradiction. Thus $y=\bfx^{(s)}+\epsilon_i$.
	 
	 Now suppose $\lceil a_i\rceil=1$ for some $i\in \{s+1,\dotsc,n\}$. Then $a_i=\frac{1}{2}$ or $a_i=1$. If $a_i=\frac{1}{2}$, then since $y\in \tW \bfx^{(s)}$ there exists $j\in \{1,\dotsc,s\}$, such that $a_j\in\bbZ$, and since $\lceil a_j\rceil-1,\lceil -a_j\rceil \leq 0$, we must have $a_j=0$ or $1$. The same argument as the previous paragraph shows that $a_k=\bfx_k^{(s)}$ for $k\neq i,j$, and hence $y=\bfx^{(s)}\pm\frac{1}{2}\epsilon_j+\frac{1}{2}\epsilon_i$. If $a_i=1$, then similarly, we have $a_k=\bfx^{(s)}_k$ for $k\neq i$, and hence $y=\bfx^{(s)}+\epsilon_i$.
	 
	 The cases when $\lceil -a_i\rceil=1$ for $i\in \{1,\dotsc,s\}$ and $i\in \{s+1,\dotsc,n\}$, can be handled analogously to the previous two cases: When $\lceil -a_i\rceil=1$ for $i\in \{1,\dotsc,s\}$, we have $y=\bfx^{(s)}-\epsilon_i$,  and when $\lceil -a_i\rceil=1$ for $i\in \{s+1,\dotsc,n\}$, we have $y=\bfx^{(s)}-\epsilon_i$ or $y=\bfx^{(s)}\pm\frac{1}{2}\epsilon_j-\frac{1}{2}\epsilon_i$.
	 
	 (ii) 
     It suffices to consider the two cases from (i).  Let $y=\bfx^{(s)}\pm\epsilon_i$. Then there is $u\in \tW_0$ with $u(\epsilon_1)=\pm \epsilon_i$. Set $\lambda=u(\mu)$. Then $t^\lambda \bfx^{(s)}=y$ and we have $t^\lambda\in \Adm(\mu)$ so the result follows in this case.
	 
	 Now let $y=\bfx^{(s)}\pm\frac{1}{2}\epsilon_j\pm \frac{1}{2}\epsilon _i, 1\leq j\leq s, s+1\leq i\leq n$; note that this case occurs only if $1\leq s\leq n-1$. Let $u\in \tW_{J^{(s)}}$ be the element which switches the $j^{\mathrm{th}}$ and $s^{\mathrm{th}}$ coordinates, and switches the $i^{\mathrm{th}}$ and $s+1^{\mathrm{st}}$ coordinates. Then upon replacing $y$ by $u(y)$, we may assume $y=\bfx^{(s)}\pm\frac{1}{2}\epsilon_s\pm \frac{1}{2}\epsilon _{s+1}.$
     
     Let $w^+\in S_n\subset \tW_0$ denote the reflection switching $\epsilon_i$ and $\epsilon_j$, and let $w^-\in \tW_0$ be the composition of $w^+$ with multiplication by $-1$ in the $s^{\mathrm{th}}$ and $s+1^{\mathrm{st}}$ factors; thus $w^-:\epsilon_s\mapsto -\epsilon_{s+1}, \epsilon_{s+1}\mapsto -\epsilon_s$. Let $\lambda_s,\lambda_{s+1}\in X_*(T)_I$ be the  $\tW_0$ conjugates of $\mu$ whose images in $\bbR^n$ are $\epsilon_s,\epsilon_{s+1}$ respectively.
	 
	 We define a pair $(u',\lambda)\in \tW_0\times X_*(T)_I$ as follows:
	 $$(u',\lambda)=\begin{cases} (w^+_{s,s+1},{\lambda_{s+1}})& \text{ if } y=\bfx^{(s)}+\frac{1}{2}\epsilon_s+\frac{1}{2}\epsilon _{s+1},\\
(w^-_{s,s+1},{-\lambda_{s+1}}) &\text{ if } y=\bfx^{(s)}+\frac{1}{2}\epsilon_s-\frac{1}{2}\epsilon _{s+1},\\
(w^-_{s,s+1},{-\lambda_{s}})&\text{ if }  y=\bfx^{(s)}-\frac{1}{2}\epsilon_s+\frac{1}{2}\epsilon _{s+1},\\
(w^+_{s,s+1},{-\lambda_{s}})	&\text{ if }  y=\bfx^{(s)}-\frac{1}{2}\epsilon_s-\frac{1}{2}\epsilon _{s+1}.
	 	 \end{cases}$$
and we set $w=u't^\lambda\in \tW$. Then we have $w(\bfx^{(s)})=y$. Moreover, we have $w\leq t^\lambda$ in the Bruhat order since $\fka$ and $t^\lambda\fka$ lie on different sides of the hyperplane corresponding to the reflection $u'$. Since $\lambda\in \tW_0(\mu)$, it follows that $t^\lambda$, and hence $w$, lie in $\Adm(\mu).$
\end{proof}

\begin{lemma}\label{lem: dim=0 implies admissible}  Let $s\in \{0,\dotsc,n\}$ and assume $t(s)\neq 0,\dim V$. Then the element $\tau_\mu\in\Adm(\mu)$ satisfies $\tau_\mu \bfx^{(s)}=\bfx^{(s)}$.
	\end{lemma}
\begin{proof}
Again, using Remark \ref{rem: automorphism case switch}, we may assume we are in case (1), (2a) or (3). We compute that $\tau_\mu:\bbZ^{n}\rightarrow \bbZ^n$ has the following form:
$$w:(a_1,a_2,\dotsc,a_{n-1},a_n)\mapsto \begin{cases} (1-a_1,a_2,\dotsc,a_{n-1},-a_n) &\text{ in case (1)},\\ (1-a_1,a_2,\dotsc,a_{n-1},a_n) & \text{ in case (2a)},\\
(a_1,a_2,\dotsc,a_{n-1},a_n) &\text{ in case (3)}.
\end{cases}$$  It is then clear that $\tau_
\mu$ fixes $\bfx^{(s)}$ if $t(s)\neq 0,\dim V$.
\end{proof}
\begin{remark} The assumption $t(s)\neq 0,\dim V$  is equivalent to the condition that $\bfx^{(s)}$ is not a hyperspecial vertex for $G_{\brQ_p}.$ 

\end{remark}

\subsubsection{}We now prove the main result of this subsection, which gives an explicit description of $\Adm(\mu)_{J^{(s)}}$ in terms of relative positions of lattices in $\brV$. 
\begin{lemma}\label{lem: transitive action on v lattices}
	Let $s\in \{0,1\dotsc,n-1\}$ and $L\subset \brV$ be a vertex lattice of type $t(s)$. Then there exists $g\in G(\brF)$ such that $gL^{(s)}=L$.
\end{lemma}
\begin{proof}
By the classification of quadratic lattices over local fields,
 every vertex lattice $L$ can be written as $L_1\obot L_2$ where $L_1$ is unimodular and $L_2^{\vee}=\pi^{-1}L_2$. Similarly, we can write $L^{(s)}$ as $L_1^{(s)}\obot L_2^{(s)}$. Since $L_1\cong L_1^{(s)}$ and $L_2\cong L_2^{(s)}$, we can find a $g$ as desired.
\end{proof}

\begin{proposition}\label{prop: relative position admissible set}Let $L,L'$ be vertex lattices of type $t=t(s)$ with $t(s)\neq 0,\dim V$. The following two conditions are equivalent:
	\begin{enumerate}\item $\dim_{k}(L+L')/L\leq1$.
		\item There exists $g\in G(\brF), w\in \Adm(\mu)_{J^{(s)}}$ such that $$gL^{(s)}=L,\ \ g\dot{w} L^{(s)}=L'$$
  where $\dot{w}\in N(\brF)$ is a lift of $w$.
		\end{enumerate}
	\end{proposition}
\begin{proof}(i) $\Rightarrow$ (ii): For simplicity, we write $\calG_s^\circ$ for the parahoric $\calG_{J^{(s)}}^\circ$. By Lemma \ref{lem: transitive action on v lattices}, there exists $g,g'\in G(\brF)$ such that $gL^{(s)}=L$ and $g'L^{(s)}=L'$. By the Bruhat decomposition, upon right multiplying $g$ and $g'$ by $\calG_s^\circ(\calO_{\brF})$, we may assume $g'=g\dot{w}'$ for some $w'\in \tW$. 
	
	If  $\dim (w' L^{(s)}+L^{(s)})/
	L^{(s)}\leq 1$, then by Lemmas \ref{lem: dim=1 implies permissible} and \ref{lem: dim=0 implies admissible}, there exists $w\in \Adm(\mu)^{J^{(s)}}$ such that $wL^{(s)}=w'L^{(s)}$ and the result follows since $\Adm(\mu)^{J^{(s)}}$ is the preimage of $\Adm(\mu)_{J^{(s)}}$ in $\tW$.

	(ii) $\Rightarrow$ (i): It suffices to show for $w\in \Adm(\mu)_{J^{(s)}}$, $\dim(L^{(s)}+wL^{(s)})/L^{(s)}\leq 1$. We first consider $w=t^{w_0(\mu)}$ for $w_0\in \tW_0$. Then $w_0(\mu)$ is of the form $\pm \epsilon_i$, and hence $L'$ is of the form $$L'=\langle e_1,\dotsc,e_{i-1}, \pi^{\pm 1}e_i,e_{i+1},\dotsc,e_n,f_1,\dotsc,f_{i-1},\pi^{\mp1}f_i,f_{i+1},\dotsc,f_n\rangle\obot L_{\an}$$ which satisfies $\dim_{k}(L+L')/L\leq1$.
	
	Let $\Fl_{\calG^\circ_s}$ denote the Witt vector affine partial flag variety over $k$ associated to the parahoric $\calG^\circ_s$ (see \cite{BhattScholze}, \cite{Zhu}).  For $g\in \Fl_{\calG^\circ_s}(k)=G(\brF)/\calG^\circ_s(\calO_{\brF})$, we associate to it the vertex lattice  $gL^{(s)}$ of type $t(s)$ (this association is not a bijection in case (3) as the parahoric is not equal to the stabilizer group in this case).
		The orbits of the positive loop group $L^+\calG^\circ_{s}$ on $\Fl_{\calG^\circ_s}\times \Fl_{\calG^\circ_s}$ are indexed by $\tW_{J^{(s)}}\backslash \tW/\tW_{J^{(s)}}$, and the closure relations are given by the Bruhat order.
	
	Now for $(g,g')\in \Fl_{\calG_s^\circ}\times \Fl_{\calG_s^\circ}(k)$, the condition that $$\dim (gL^{(s)}+g'L^{(s)})/gL^{(s)}\leq 1$$ defines a closed subscheme of $\Fl_{\calG_s^\circ}\times \Fl_{\calG_s^\circ}$. Indeed, it is the pre-image of the quasi-minuscule Schubert variety in the Witt vector affine Grassmannian $\Gr(L^{(s)})$ for $\GL(L^{(s)})$ under the natural map $\Fl_{\calG^\circ_s}\rightarrow \Gr(L^{(s)})$ induced by the locally closed immersion $\calG_s^\circ\rightarrow \GL(L^{(s)})$. Since the elements $t_{w_0(\mu)}, w_0\in \tW_0$ are the maximal elements of $\Adm(\mu)_{J^{(s)}}$, it  follows that for any $w\in \Adm(\mu)_{J^{(s)}}$, we have $\dim(L^{(s)}+wL^{(s)})/L^{(s)}\leq 1$.

\end{proof}

\section{Affine Deligne-Lusztig varieties and special lattices}

\subsection{Clifford algebra and GSpin local Shimura datum}\label{Subsec: GSpin Shimura datum} \subsubsection{}We now specialize to the case  $F=\bbQ_p$, so let $V$ be a quadratic space over $\bbQ_p$ of dimension $n$.
For a $\Q_p$-algebra $R$, the tensor product $V_R=V \otimes_{\Q_p} R$ is a   quadratic space over $R$, and let $C(V_R )=C(V) \otimes_{\Q_p} R$ denote the Clifford algebra of $V_R$. Then $C(V_R)$ is a semisimple algebra over $R$ of rank $2^{\dim_{\bbQ_p} V}$, and is equipped with a $\bbZ/2\bbZ$-grading $C(V_R)=C^+(V_R)\oplus C^-(V_R)$. We define $G=\operatorname{GSpin}(V)$ to be the group scheme over $\Q_p$ with $R$-points  
$$
G(R)=\{g \in C^{+}(V_R)^{\times}: g V_R g^{-1}=V_R,\,  g^* g \in R^{\times}\}.
$$
We use $g \bullet v=g v g^{-1}$ to denote the conjugation action of $G$ on $V$. Then there is a short exact sequence of group schemes
$$
1 \rightarrow \mathbb{G}_m \rightarrow G \stackrel{g \mapsto g \bullet}{\longrightarrow} \mathrm{SO}(V) \rightarrow 1
$$
over $\bbQ_p$. When $V$ is clear in the context, we simply denote   $C(V_R)$ as $C_R$.

\subsubsection{}\label{sssec: cocharcter mu}	Let $\{x_1, \ldots, x_n \}$ be a  basis of $V_{\breve{\Q}_p}$  with moment matrix
	$$
	 \left(\begin{array}{cccccc}
		0 & 1 & & & & \\
		1 & 0 & & & & \\
		& & * & & & \\
		& & &  \ddots & & \\
		& & & & *
	\end{array}\right). 
	$$
Given this choice of basis, we define a cocharacter $\mu: \mathbb{G}_m \rightarrow G$ by
	\begin{align}\label{eq: def of mu}
	    \mu(t)=t^{-1} x_1 x_2+x_2 x_1,
	\end{align}
	where we regard the right hand side of the above as an element in  $C(V_{\breve{\Q}_p})$. 

   We remark that the reflex field of $\mu$ is always $\Q_p$. Indeed, when $V$ has a split subspace of rank $2$, then we can choose $\{x_1,x_2\}$ to be the basis of this subspace and $\mu$ is even defined over $\Q_p$. Now we assume $V$ does not have a split subspace of rank $2$. Then $V$ contains a subspace generated by $\{v_1,v_2\}$ with moment matrix $\begin{pmatrix}
       1&0\\
       0&\nu
   \end{pmatrix}$ or $\begin{pmatrix}
       p&0\\
       0&p\nu
   \end{pmatrix}$ where $\chi(-\nu)=-1$. Assume  $\{v_1,v_2\}$ has moment matrix $\begin{pmatrix}
       1&0\\
       0&\nu
   \end{pmatrix}$. Let $\alpha\in \Q_{p^2}^\times$ such that $\alpha^2=-\nu$. Then we may choose $x_1=\alpha v_1+v_2$ and $x_2=\frac{-1}{2\nu} (\alpha v_1-v_2)$ so that the moment matrix of $\{x_1,x_2\}$ is $\begin{pmatrix}
       0&1\\
       1&0
   \end{pmatrix}.$ Therefore, in terms of $v_1,v_2$,  we have 
   \begin{align*}
       \mu(t)=t^{-1}(\alpha v_1+v_2)(\frac{-1}{2\nu} (\alpha v_1-v_2))+\frac{-1}{2\nu} (\alpha v_1-v_2)(\alpha v_1+v_2).
   \end{align*}
Let $\tau$ denote the nontrivial element in $\Gal(\Q_{p^2}/\Q_p)$ such that $\tau(\alpha)=-\alpha$. 

Although $\mu$ is only defined over $\Q_{p^2}$, we claim that
\begin{align*}
    \tau\cdot \mu(t)= t^{-1}(-\alpha v_1+v_2)(\frac{-1}{2\nu} (-\alpha v_1-v_2))+\frac{-1}{2\nu} (-\alpha v_1-v_2)(-\alpha v_1+v_2)
\end{align*}
lies in the $G(\bar{\Q}_p)$-conjugacy class of $\mu$. Indeed, note that the moment matrix of $\{-\alpha v_1+v_2, -\alpha v_1-v_2\}$ is still $\begin{pmatrix}
       0&1\\
       1&0
   \end{pmatrix}.$ Therefore we can find  $h\in\mathrm{O}(V)(\bar{\Q}_p)$ so that $\rho(\mu)=h\circ \rho(\tau\cdot \mu) \circ h^{-1}$ by Witt's theorem, where $\rho$ denotes the morphism from $G$ to $\mathrm{SO}(V)$. If $\det(h)=-1,$ we can consider $h'\in \mathrm{O}(V)(\bar{\Q}_p)$ such that $h'(x_3)=-x_3$  and fixes $x_i$ for $i\neq 3$. Then $h'h\in \mathrm{SO}(V)(\bar{\Q}_p)$ and $\rho(\mu)=(h'h)\circ \rho(\tau \cdot \mu) \circ (h'h)^{-1}$. Therefore we can assume $h$ lies in $\mathrm{SO}(V)(\bar{\Q}_p)$ without loss of generality. Now let $g\in G(\bar{\Q}_p)$ be a lifting of $h$. Since $\ker(\rho)=\mathbb{G}_m$ is central and $\rho(g\circ (\tau\cdot \mu)\circ g^{-1})=\rho(\mu)$, we have $g\circ (\tau\cdot \mu(t))\circ g^{-1}=\mu(t) \circ t^m$ for some $m\in \Z$.     Since
   $$(g\circ (\tau\cdot \mu(t))\circ g^{-1})(g\circ (\tau\cdot \mu(t))\circ g^{-1})^*=t^{-1}$$ 
   is equal to 
   $$(\mu(t) \circ t^m)(\mu(t) \circ t^m)^*=t^{2m-1},$$
   we have $m=0$, and our claim is proved.
   
   The case $\{v_1,v_2\}$ has moment matrix   $\begin{pmatrix}
       p&0\\
       0&p\nu
   \end{pmatrix}$ is the same, so we omit the details.

\subsubsection{}Recall that $\brQ_p$ denotes the completion of the maximal unramified extension of $\bbQ_p$, and we let $\sigma\in \mathrm{Aut}(\brQ_p/\bbQ_p)$ be the Frobenius element. For an element $b\in G(\breve{\Q}_p)$, we consider the following isocrystals:
\begin{align*}
    (V_{\breve{\Q}_p}, \Phi=b\circ \sigma) \quad \text{ and } \quad (C_{\breve{\Q}_p}, F=b\circ \sigma).
\end{align*}
Here $\Phi$ and $F$ are related by $\Phi x=F\circ x \circ F^{-1}$.
According to \cite{Kottwitz85},  we can associate to every $b\in G(\breve{\Q}_p)$, its Newton cocharacter 
$$
\nu_b: \mathbb{\bbD}  \rightarrow G 
$$ 
which is characterized by the following condition. For any representation $\rho: G  \rightarrow \mathrm{GL}(N)$ on a $\Q_p$-vector space $N$, we require that the decomposition
$$
N_{\brQ_p}=\bigoplus_{s / t \in \mathbb{Q}} N^{s / t}
$$
induced by $\rho\circ \nu_b$ is identified with the  slope decomposition of the isocrystal $(N_{\brQ_p},\rho(b)\circ \sigma)$. 
We call an element of $b\in G(\breve{\Q}_p)$ \emph{basic} if $\nu_b$ factors through the center of $G$.

Let $B(G)$ denote the set of $\sigma$-conjugacy classes of $G(\breve{\Q}_p)$. Given $\mu$,  Kottwitz introduced a subset   $B(G,\mu)\subset B(G)$ in \cite{kottwitz_1997}:
\begin{align*}
    B(G, \mu)=\left\{[b] \in B(G) ; \kappa([b])=\kappa(\mu), \bar{\nu}_b \leqslant \mu^{\diamond}\right\},
\end{align*}
where $\kappa([b])\in \pi_1(G)_{\Gamma}$ is the Kottwitz invariant of $b$, $\kappa(\mu)$ is the image of $\mu$ in $\pi_1(G)_{\Gamma}$, $\bar{\nu}(b)$ is the  dominant Newton cocharacter of $[b]$ and $\mu^{\diamond}$ is the Galois average of $\mu$. We refer the reader to \cite{kottwitz_1997} and \cite{He2016b} for details. From now on, $b$ will always denote the unique basic element of $B(G,\mu).$

Later on, we will need to consider  
    $\bV\coloneqq V_{\breve{\Q}_p}^{\Phi},$ which will be regarded as the space of special endomorphisms of the framing object of the GSpin RZ space in later sections (see \S \ref{sec: special cycles}). 
The following description of $\bV$ is a generalization of \cite[Proposition 4.2.5]{HowardPappas}. We also determine the vertex lattices of maximal and minimal types in $\bV$. The description of a general vertex lattice follows from  this easily. 

To simplify the notation,  we use $U_n^{+}$ (resp. $U_n^-$) to denote the split (resp. non-split) non-degenerate quadratic space over $\mathbb{F}_p$ of dimension $n$.
\begin{proposition}\label{prop: V^Phi}
Assume $V$ is a quadratic space over $\Q_p$ with Frobenius $\Phi$ as in \S \ref{sec: collection of background}. 
		Then the $\Q_p$-quadratic space
		$$
		\bV\coloneqq V_{\brQ_p}^{\Phi}=\{x \in V_{\brQ_p}: \Phi x=x\}
		$$
		has the same dimension and discriminant as $V$, and has Hasse invariant $\epsilon(\bV)=-\epsilon(V)$.
	
 More precisely, if we let $\Lambda_{\max}\subset \bV$ (resp. $\Lambda_{\min}$) denote the vertex lattice in $\bV$ of maximal (resp. minimal) type, then we have the following.
 \begin{enumerate}
    \item If  $\operatorname{dim}\left(V_{\an}\right)=0$, then   
     \begin{align*}
         \Lambda_{\max}=\Lambda_{n-2}^-\obot H_2^-  \quad \text{ and } \quad \Lambda_{\min}=H_{n-2}^-\obot \Lambda_2^-. 
    \end{align*}
 
    \item If $\operatorname{dim}\left(V_{\an}\right)=1$, then
    \begin{align*}
       \begin{cases}
            \Lambda_{\max}= \Lambda_{n-1}^-\obot H_1^{-\chi}, \quad   \quad \Lambda_{\min}=H_{n-2}^{-\chi}\obot \Lambda_2^{-} & \text{ if $V=(H_n^\chi)_{\Q_p}$},\\
           \Lambda_{\max}=\Lambda_{n-2}^{-\chi}\obot \Lambda_2^{-}, \quad    \quad    \Lambda_{\min}=H_{n-1}^-\obot \Lambda_1^{-\chi} &\text{ if $V=(H_{n-1}^+\obot \Lambda_1^\chi)_{\Q_p}$}.
        \end{cases} 
    \end{align*}
   
    \item If $\operatorname{dim}\left(V_{\an}\right)=2$, then 
      \begin{align*}
        \begin{cases}
          \Lambda_{\max}=\Lambda_{n}^-,  \quad  \quad    \quad\quad\,\, \, \, \, \, \, \,  \,  \text{ $\Lambda_{\min}=H_{n-2}^{+}\obot \Lambda_2^-$  } &\text{if $V=(H_n^{-})_{\Q_p}$},\\
           \Lambda_{\max}=\Lambda_{n-2}^+\obot H_2^- ,       \quad    \quad   \text{ $\Lambda_{\min}=H_n^{-}$ } &\text{if $V=(H_{n-2}^{+}\obot \Lambda_2^-)_{\Q_p}$},\\
            \Lambda_{\max}= \Lambda_{n-1}^{\chi_2} \obot H_1^{\chi_1}  \quad   \quad \  \text{ $\Lambda_{\min}=H_{n-1}^{\chi_1}\obot \Lambda_1^{\chi_2}$} &\text{if $V=(H_{n-1}^{\chi_2}\obot \Lambda_1^{\chi_1})_{\Q_p}$}.
        \end{cases} 
    \end{align*}
 
    \item If $\operatorname{dim}\left(V_{\an}\right)=3$, then 
    \begin{align*}
         \begin{cases}
           \Lambda_{\max}=\Lambda_{n-1}^+\obot H_1^{\chi},    \quad  \quad \text{  $\Lambda_{\min}=H_{n}^{\epsilon }$ } & \text{if $V=(H_{n-2}^{-\chi }\obot \Lambda_2^{-})_{\Q_p}$ },\\
            \Lambda_{\max}= \Lambda_{n-1}^+\obot \Lambda_1^{\chi},    \quad   \quad \text{ $\Lambda_{\min}=H_{n-1}^{+}\obot \Lambda_1^{\chi}$} & \text{if $V=(H_{n-1}^{- }\obot \Lambda_1^{-\chi})_{\Q_p}$}.
        \end{cases} 
    \end{align*}
     
     \item If $\operatorname{dim}\left(V_{\an}\right)=4$, then 
     \begin{align*}
         \Lambda_{\max}=\Lambda_{n}^+  \quad \text{ and } \quad \Lambda_{\min}=H_{n}^+. 
    \end{align*}
\end{enumerate}
\end{proposition}
 \begin{proof}
       If $V$ contains a self-dual lattice, then this is implied by \cite[Proposition 4.2.5]{HowardPappas}. 
      
      Now we assume $V$ does not contain a self-dual lattice. In particular, since $n\ge 3$, we may assume $V$ has a basis $\{x_1,\ldots,x_n\}$ with diagonal moment matrix such that $q(x_1)=u_1 p$ and $q(x_2)=u_2$ where $u_1, u_2\in \Z_p^\times$. Then we can choose $b=x_2x_1$ so that $b^2=-u_2u_1p\in p\Z_p^\times$.  Then one can check that for $\Phi=b\circ \sigma$, we have 
      \begin{align*}
          \Phi \bullet x_1=-x_1, \quad \Phi\bullet x_2=-x_2, \quad \Phi\bullet x_i=x_i \quad  \text{ for $3\le i\le n$.}
      \end{align*}
Choose $\eta \in \Z_{p^2}^\times\setminus \Z_p^\times$ such that $\sigma(\eta)=-\eta$ and $\eta^2=\delta$ where $\delta\in \Z_p^\times\setminus (\Z_p^{\times})^2$.  Set
\begin{align*}
    v_1=\eta x_1, \quad v_2=\eta x_2, \quad v_i =x_i \quad  \text{ for $3\le i\le n$.}
\end{align*}
Then one can easily check that $\Phi\bullet v_i=v_i$ for any $i\in\{1,\ldots,n\}$. Hence $\{v_1,\ldots,v_n\}$ is a set of basis vectors of $\bV$.  Moreover, we have 
\begin{align*}
    q(v_1)=\delta q(x_1) \quad q(v_2)=\delta q(x_2).
\end{align*}
Therefore we have $\chi_0(V)=-\chi_0(\bV)$ and $\chi_1(V)=-\chi_1(\bV)$. In particular, $V$ and $\bbV$ have the same discriminant.  Moreover, one can check that $\epsilon(V)=-\epsilon(\bV)$ by \eqref{eq: Hasse}.

Then applying Proposition \ref{prop: class of V},  the proposition follows.
 \end{proof}

\subsection{Affine Deligne--Lusztig varieties}
\subsubsection{} 
Let $G=\GSpin(V)$ and $\mu$ be as above.  We now fix a vertex lattice $L\subset V$ of type $h$. This determines a parahoric group scheme $\calG^\circ$ of $G.$ We fix $S$ a maximal $\brQ_p$-split torus of $G$ which is in good relative position with respect to $\calG^\circ$, so that $L$ corresponds to a facet $\fkf$ in the apartment $A(G,S,\brQ_p)$ corresponding to $S$, and we fix an alcove $\fka\subset A(G,S,\brQ_p)$ whose closure contains $\fkf$. Then $L$ corresponds to a subset $J\subset\bbS$ of the set of simple reflections in $\fka$.   We then have the associated admissible  set $\Adm(\mu)_J.$

Recall we have fixed an element $b\in G(\brQ_p)$ whose  image $[b]\in B(G)$ is the unique basic element of $B(G,\mu)$. We let $\K=\calG^\circ(\Z_p)$, and we set   
$$
X(\mu^\sigma,b)_{\K}(k):=\left\{g \cdot \K \in G(\breve{\Q}_p) / \K \mid g^{-1} b \sigma(g) \in \bigcup_{w \in \operatorname{Adm}(\mu)_J} \K \sigma(\dot{w}) \K\right\},
$$ 
where $k$ is the residue field of $\brZ_p$. Then $X(\mu^\sigma,b)_\K(k)$ arises as the $k$-points of a closed subscheme $X(\mu^\sigma,b)_\K$ of the Witt vector affine Grassmannian constructed in \cite{Zhu}, \cite{BhattScholze}, and $X(\mu^\sigma,b)_\K$ is locally of perfectly finite type. By \cite[Theorem A]{He2016b}, we have $X(\mu^\sigma,b)_\K(k)\neq \emptyset$.

\begin{definition}
   A special lattice $M\subset V_{\brQ_p}$ of type $h$ is a $\brZ_p$-lattice such that 
   \begin{align*}
        p M^\vee \subset M \stackrel{h}{\subset}M^\vee \quad 
   \end{align*}
and 
   \begin{align*}
	\begin{cases}
    [M+\Phi(M):M]=1 & \text{ if $h= 0,n$,}\\
 [M+\Phi(M):M]\le 1 & \text{ if $h\neq 0,n$.}
	\end{cases}		
		\end{align*}
        Here the notation $M \stackrel{h}\subset M^\vee$ means that $M^\vee/M$ is an $h$-dimensional vector space over the residue field $\bbF_p$.
\end{definition}

\subsubsection{}

For any $g\in X(\mu^\sigma,b)_\K(k)$, we define the following vertex lattices

$$
L_g\coloneqq \sigma^{-1}(b^{-1}g)\bullet L \quad \text { and } \quad L_g'\coloneqq g\bullet L.  
$$   
We have
\begin{align*}
    \Phi(L_g)=(bb^{-1}g)\bullet L   =L_g'.
\end{align*}
Also, note that $L_g$ and $L_g'$ determine each other. Indeed, if $L_g = L_{g'} $ for a different $g'$, then one can directly check $L_g'=L_{g'}'$ and vice versa.
 
\begin{proposition}\label{prop: ADLV and special lattice}
 The map   $g \mapsto L_g$ induces a bijection between $p ^{\mathbb{Z}} \backslash X(\mu^\sigma,b)_\K(k) $ and
 \begin{align*}
     \begin{cases}
     \text {\{special lattices of type $h$ in } V_{\brQ_p}\} & \text{ if $h= 0,n$},\\
         \text {two copies of \{special lattices of type $h$ in } V_{\brQ_p}\} & \text{ if $h\neq 0,n$.}
     \end{cases}
 \end{align*}
\end{proposition}
\begin{proof}
The case $h=0,n$ is \cite[Proposition 6.2.2]{HowardPappas}. So we assume $h\neq 0,n$ from now on.

   First, we check $L_g$ is indeed a special lattice. We have 
   \begin{align*}
       [L_g+\Phi(L_g):L_g]=[L_g+L_g':L_g]&=[g\bullet L+\sigma^{-1}(b^{-1}g)\bullet L:\sigma^{-1}(b^{-1}g)\bullet L]\\
       ~&=[(\sigma^{-1}(g^{-1}b)g)\bullet L+L:L]
   \end{align*}
   Note that the map $G\rightarrow \SO(V)$ induces a bijection  $$\Adm(\mu)_J\cong \Adm(\mu')_{J'},$$ where $\mu'$ is the cocharacter of $\SO(V)$ induced by $\mu$, and $\Adm(\mu')_{J'}$ is the corresponding admissible set of $\SO(V)$.
   Thus since $g^{-1}b\sigma(g) \in K\sigma(\dot{w})K$, $w\in \Adm(\mu)_J$, we have $[L_g+\Phi(L_g):L_g]\le 1$ by Proposition \ref{prop: relative position admissible set}. 
   
  Now we show the map $g\mapsto L_g$ is surjective.  Let $L'$ be any special lattice. Then by Proposition \ref{prop: relative position admissible set}, there exists $g\in G(\brQ_p)$ and $w\in \Adm(\mu)_{J}$ such that $$(\Phi(L'),L')=(g\bullet L,g\dot{w}^{-1}\bullet L).$$ It follows that $g^{-1}b\sigma(g)\sigma(\dot{w}^{-1})$ fixes $L$, and hence lies in $p^\bbZ K.$ Since $\kappa([b])=\kappa(\dot{w})=\kappa(\mu)\in \pi_1(G)_\Gamma$, it follows that $g^{-1}b\sigma(g)\sigma(\dot{w}^{-1})\in K$, and hence $g\in X(\mu^\sigma,b)_K(k).$ 

The stabilizer of $L$ in $G(\breve{\Q}_p)$ is $p^\Z \K \sqcup \delta p^\Z \K$, where the induced actions of $\delta\in G(\brF)$ on $L^\vee/L$ and $L/p L^\vee$ have determinant $-1$ (see \cite[\S 2.5]{PappasZachos}).  Hence the map $g\to L_g$ is a two-to-one map.
\end{proof}

\subsection{The crucial proposition}
\subsubsection{} 
 Let $M\subset V_{\brQ_p}$ denote a special lattice of type $h$.
Let 
\begin{align*}
    S_i(M)\coloneqq M+\Phi(M)+\cdots + \Phi^i(M).
\end{align*}
	From now on, we  always assume $c$ (resp. $d$) is the smallest non-negative integer such that $S_c(M)$ (resp. $S_d(M^\vee)$) is $\Phi$-invariant ($c$ and $d$ are finite by a standard argument). It is not hard to show that
	\begin{align}\label{eq: chain}
		M\stackrel{1}{\subset}S_1(M) \stackrel{1}{\subset} S_2(M) \cdots \stackrel{1}{\subset} S_c(M)=S_{c+1}(M).
	\end{align}
Also, $M$ is $\Phi$-invariant if and only if $M^\vee$ is $\Phi$-invariant 
 by the following lemma.
	\begin{lemma}\label{lem: same ind}
		$[M+\Phi (M):M]=[M^\vee +\Phi(M^\vee) : M^\vee]$.
	\end{lemma}
	\begin{proof}
		First of all, note that $[M^\vee +\Phi(M^\vee) : M^\vee]=[M: M\cap \Phi(M)]$. Since $S_c(M)$ is $\Phi$-invariant, we have $[S_c(M):M]=[S_c(M):\Phi(M)]$. Hence, we have $[M:M\cap \Phi(M)]=[\Phi(M):M\cap \Phi(M)]$.
		Since  $[\Phi(M)+M:M]=[\Phi(M):\Phi(M)\cap  M]$,  the proof is finished.
	\end{proof}
 
We also remark  that the equality $[S_c(M):M]=[S_c(M):\Phi(M)]$ combined with \ref{eq: chain} implies that  $[M+\Phi(M):M]=[M+\Phi(M):\Phi(M)]$, which further implies that 
\begin{align*}
    \Phi(M)\stackrel{1}{\subset} S_1(M) \stackrel{1}{\subset} S_2(M)\cdots \stackrel{1}{\subset} S_c(M).
\end{align*}

	\begin{lemma}\label{lem: inf inter}
		We  have
        \begin{enumerate}
            \item 
		 
			$S_c(M)\subset   \cap_{i\in \Z_{\ge 0}}\Phi^i(p^{-1} M)$ and $     S_d(M^\vee)\subset   \cap_{i\in \Z_{\ge 0}}\Phi^i(p^{-1} M^\vee). $
		 
        \item If $\Phi(M)\subset M^\vee$ and $c\le d$, then  
        \begin{align*}
            S_c(M)\subset   \cap_{i\in \Z_{\ge 0}}\Phi^i(M^\vee ).
        \end{align*}
         \item If $\Phi(p M^\vee)\subset M$ and $d\le c$, then  
        \begin{align*}
            S_d(p M^\vee)\subset   \cap_{i\in \Z_{\ge 0}}\Phi^i(M ).
        \end{align*}
         \end{enumerate}
	\end{lemma}
	\begin{proof}
	For (i),	we only prove the statement for $S_c(M)$ since the proof for $S_d(M^\vee)$ is literally  the same.

		First, note that
		\begin{align*}
			\Phi(M)\subset p^{-1} M.
		\end{align*}
		Hence
		\begin{align*}
			S_c(M)\subset p^{-1} S_{c-1}(M).
		\end{align*}
		Since $S_c(M)$ is $\Phi$-invariant, we have 
		\begin{align*}
			S_c(M)\subset \cap_{i\in \Z_{\ge 0}}p^{-1}\Phi^i(S_{c-1}(M)).  
		\end{align*}
		Note that
		\begin{align*}
			S_c(M)\subset p^{-1}  \Phi(S_{c-1}(M))\cap p^{-1} S_{c-1}(M)=p^{-1}\Phi(S_{c-2}(M)).
		\end{align*}
      Here, for the last equality,  we use the fact that  $$\Phi(S_{c-1}(M))\neq S_{c-1}(M),$$ $$\Phi(S_{c-2}(M))\subset \Phi(S_{c-1}(M))\cap S_{c-1}(M)$$ and $$[S_{c-1}(M):S_{c-2}(M)]=[\Phi(S_{c-1}(M)):\Phi(S_{c-2}(M))]=1.$$

  Since $S_c(M)$ is $\Phi$-invariant,
		\begin{align*}
			S_c(M)\subset p^{-1}S_{c-2}(M).
		\end{align*}
		Then by an induction argument, we have $S_c(M)\subset p^{-1} M$. Since $S_c(M)$ is $\Phi$-invariant, we have
		\begin{align*}
			S_c(M)\subset  p^{-1} \cap_{i\in \Z_{\ge 0}}\Phi^i(M)\subset  p^{-1} \cap_{0\le i\le f}\Phi^i(M)
		\end{align*}
		for any $f\in \Z_{\ge 0}$.

  (ii) and (iii) are essentially the same as (i) once we observe that the assumption $c\le d$ (resp. $d\le c$) implies that $\Phi(S_{c-1}(M^\vee))\cap S_{c-1}(M^\vee)=\Phi(S_{c-2}(M^\vee))$ (resp. $\Phi(S_{d-1}(M))\cap S_{c-1}(M)=\Phi(S_{c-2}(M))$).
	\end{proof}

The following is a first classification of special lattices.	
	 \begin{proposition}\label{prop: KR stratification}
	 Assume $M$ is a special lattice.    We have either $\Phi(M)\subset M^\vee$ or $\Phi(pM^\vee)\subset M$.
	 \end{proposition}
     \begin{proof}
         If  $\Phi(M)\not\subset M^\vee$, then $\Phi(M)+M^\vee=M^\vee+\Phi(M^\vee)$. Moreover, if $\Phi(pM^\vee)\not\subset M$, then $p^{-1}M+p^{-1}\Phi(M)=\Phi(M^\vee)+p^{-1}M=\Phi(M)+p^{-1}M=p^{-1}M$, which is a contradiction. 
     \end{proof}
 \begin{remark}\label{rem: KR}
     One can check that the conditions $\Phi(M)\subset M^\vee$ and $\Phi(pM^\vee)\subset M$ actually correspond to the conditions for cutting out KR-strata. See \S \ref{sec: ortho local model} and Lemma \ref{lem: KR strata} for more details. See also Remark \ref{rem: rem KR 2}.
 \end{remark}

The main result of this subsection is the following crucial proposition.	
	\begin{proposition}\label{prop: key prop}
	Assume $M$ is a special lattice. Then one of the following will happen: 
		\begin{enumerate}
			\item $p (S_c(M))^\vee \subset p M^\vee \subset M\subset   S_c(M) \subset (S_c(M))^\vee \subset M^\vee.$ In particular, $S_c(M)$ is a vertex lattice of type $t\le h$.
			\item $p M^\vee \subset p S_d(M^\vee)\subset  S_d(M^\vee)^\vee  \subset M\subset   M^\vee\subset S_d(M^\vee)$. In particular, $S_d(M^\vee)^\vee$ is a vertex lattice of type $t\ge h$.
		\end{enumerate}
More precisely,  we have
\begin{enumerate}[label=(\alph*)]
    \item   If $\Phi(p M^\vee)\not\subset   M$, then $M$ satisfies $(1)$.
    \item  If $\Phi(M)\not\subset M^\vee$, then $M$ satisfies $(2)$.
    \item If $\Phi(p M^\vee) \subset M$ and $\Phi(M) \subset M^\vee$, then $M$ satisfies $(1)$   if $c\le d$.

    \item If $\Phi(p M^\vee) \subset M$ and $\Phi(M) \subset M^\vee$, then $M$ satisfies $(2)$   if $d\le c$.
        
\end{enumerate}
	\end{proposition}
	\begin{proof}
$(a)$ Assume $\Phi(  M^\vee)\not\subset p^{-1} M$. Since $M\subset   S_c(M)$, it suffices to show  $  S_c(M)\subset S_c(M)^\vee$.

		Since $ M \stackrel{\le 1}{\subset}  M + \Phi(M)$, $\Phi(p M^\vee)\not\subset  M $ and $\Phi(p M^\vee)\subset  \Phi(M)$, we have  $M+\Phi(M)= M+\Phi(p M^\vee)$. In fact, an inductive argument shows that 
		\begin{align}\label{eq: 33}
	 S_c(M)=M+\Phi(p M^\vee)+\cdots +\Phi^c(p M^\vee).
		\end{align}
Equivalently,
		\begin{align}\label{eq: 66}
			S_c(M)^\vee=M^\vee\cap  (\cap_{1\le i\le c} \Phi^i(p^{-1} M)) .
		\end{align}
According to Lemma \ref{lem: inf inter}, we have
		\begin{align}\label{eq: 11}
        \begin{split}
			S_c(p M^\vee)&\subset     \cap_{0\le i\le c}\Phi^i(M^\vee) \subset   M^\vee\cap  (\cap_{1\le i\le c}\Phi^i(p^{-1}M)),\\
           M&\subset   M^\vee\cap (  \cap_{1\le i\le c}\Phi^i(p^{-1}M)).
        \end{split}
		\end{align}
Therefore,
\begin{align*}
    S_c(M)\stackrel{\eqref{eq: 33}}{\subset} M+S_c(p M^\vee) \stackrel{\eqref{eq: 11}}{\subset}M^\vee\cap (  \cap_{1\le i\le c}\Phi^i(p^{-1}M))\stackrel{\eqref{eq: 66}}{=}S_c(M)^\vee.
\end{align*}

$(b)$ Now we assume  $\Phi(M)\not\subset M^\vee$. By construction, we have $S_d(M^\vee)^\vee\subset M\subset   M^\vee\subset S_d(M^\vee)$. Therefore, we only need to show $p S_d(M^\vee)\subset S_d(M^\vee)^\vee$.

		Since $M^\vee \stackrel{\le 1}{\subset} M^\vee +\Phi(M^\vee)$ and $\Phi(M)\not\subset M^\vee$, we have  $M^\vee +\Phi(M^\vee)=M^\vee +\Phi(M)$. In fact, an inductive argument shows that 
		\begin{align}\label{eq: 3}
			S_d(M^\vee)=M^\vee+\Phi(M)+\cdots +\Phi^d(M).
		\end{align}
		Equivalently,
		\begin{align}\label{eq: 6}
			S_d(M^\vee)^\vee=M\cap (\cap_{1\le i\le d}\Phi^i(M^\vee)).
		\end{align}
According to Lemma \ref{lem: inf inter}, we have
		\begin{align}\label{eq: 1}
        \begin{split}
			S_d(p M)&\subset     \cap_{0\le i\le d}\Phi^i(M) \subset   M \cap  (\cap_{1\le i\le d}\Phi^i( M^\vee)),\\
           p M^\vee &\subset   M \cap (  \cap_{1\le i\le d}\Phi^i( M^\vee)).
        \end{split}
		\end{align}
Therefore,
\begin{align*}
   p  S_d(M^\vee)\stackrel{\eqref{eq: 3}}{\subset} p M^\vee+  S_d(p  M) \stackrel{\eqref{eq: 1}}{\subset}M \cap (  \cap_{1\le i\le d}\Phi^i( M^\vee))\stackrel{\eqref{eq: 6}}{=}S_d(M^\vee)^\vee.
\end{align*}

$(c)$ Assuming  $\Phi(M) \subset M^\vee$  and $c\le d$, we will show that $M$ satisfies $(1)$ (a stronger statement than $(c)$).  Note that $M+\Phi(M)\subset M^\vee$. Inductively, we see that  $S_c(M)\subset S_{c-1}(M^\vee)$. Since $S_c(M)$ is $\Phi$-invariant,  we have $S_c(M) \subset \cap_{0\le i\le c} \Phi^i(M^\vee) \subset (S_c(M))^\vee $ by   Lemma \ref{lem: inf inter}. Hence, we have
		\begin{align*}
			M\subset S_c(M) \subset (S_c(M))^\vee \subset M^\vee.
		\end{align*}
		Since $p M^\vee \subset M \subset M^\vee$, we have 
		\begin{align*}
		p (S_c(M))^\vee \subset p M^\vee \subset M\subset   S_c(M) \subset (S_c(M))^\vee \subset M^\vee.
		\end{align*}

 $(d)$  Assuming  $\Phi(p M^\vee) \subset M$  and $d\le c$, we will show that $M$ satisfies $(2)$.   
 Note that $p M^\vee+\Phi(p M^\vee)\subset M$. Inductively, we see that  $S_d(p M^\vee)\subset S_{d-1}(M)$. Since $S_d(p M^\vee)$ is $\Phi$-invariant,  we have $S_d(p M^\vee) \subset \cap_{0\le i\le d} \Phi^i(M) \subset p^{-1}(S_d(p M^\vee))^\vee $ by   Lemma \ref{lem: inf inter}. Hence, we have
		\begin{align*}
			p M^\vee \subset p S_d(M^\vee) \subset S_d(M^\vee)^\vee \subset M.
		\end{align*}
		Since $p M^\vee \subset M \subset M^\vee$, we have 
		\begin{align*}
		p M^\vee \subset p S_d(M^\vee)\subset  S_d(M^\vee)^\vee  \subset M\subset   M^\vee\subset S_d(M^\vee).
		\end{align*}\end{proof}

\section{Deligne--Lusztig varieties}

We let $V$ be a quadratic space over $\bbQ_p$ and let $G=\GSpin(V)$, $b$, $\mu$ be as in the last section. In particular, we have an associated quadratic space $\bbV$ given by Proposition \ref{prop: V^Phi}. We fix $L\subset V$ a vertex lattice of type $h$. Let $\cV^{\ge h}(\bV)$ (resp. $\cV^{\le h}(\bV)$) denote the set of vertex lattices in $\bV$ of type $t\ge h$ (resp. $t\le h$). If the quadratic space is clear in the context, we simply denote it as $\cV^{\ge h} $ (resp. $\cV^{\le h} $).  In this section, we attach to a vertex lattice $\Lambda \in \cV(\bbV)^{\ge h}$ (resp.  $\Lambda \in \cV(\bbV)^{\le h}$) a $k$-scheme   $S_{\Lambda}$  (resp. $R_{\Lambda}$ ) parameterizing certain special lattices.
Then we relate $S_{\Lambda}$ and  $R_\Lambda$ with certain Deligne-Lusztig varieties and study their geometry.

\subsection{The variety $S_{\Lambda}$ and  $R_\Lambda$}

\subsubsection{}
 First, we can associate to a vertex lattice certain quadratic spaces defined over $\F_p$ as follows. 
Let $\Lambda\subset \bbV$ be a vertex lattice of type $t_\Lambda$ and rank $n$ over $O_F$ with quadratic form $(\, ,\, )$. We set
\begin{align*}
   V_{\Lambda}\coloneqq \Lambda^\vee/ \Lambda \quad  \text{ and } \quad V_{\Lambda^\vee} \coloneqq \Lambda/p \Lambda^\vee  
\end{align*}
with quadratic form induced from $p(\, , \, )$ and $ (\, , \, )$ respectively. Note that  $\dim_{\F_p}V_{\Lambda}=t_\Lambda$  and $\dim_{\F_p}V_{\Lambda^\vee}=n-t_\Lambda$.   
   
   Set
	\begin{align*} 
	\Omega_{\Lambda}&\coloneqq V_{\Lambda} \otimes_{\mathbb{F}_p} k  \stackrel{\sim}{\rightarrow} \Lambda_{\breve{\Z}_p}^\vee / \Lambda_{\breve{\Z}_p},\\
    \Omega_{\Lambda^\vee}&\coloneqq V_{\Lambda^\vee} \otimes_{\mathbb{F}_p} k  \stackrel{\sim}{\rightarrow} \Lambda_{\breve{\Z}_p} /p  \Lambda_{\breve{\Z}_p}^{\vee}
  \end{align*} 
	with its Frobenius operator  $\mathrm{Id}\otimes \sigma=\Phi$.   
	
\subsubsection{}	For a quadratic space $\Omega$  over $k$,  we let $\mathrm{OGr}(i,\Omega)$ be the orthogonal Grassmannian that parametrizes $i$-dimensional isotropic subspaces of $\Omega$.
 For $\Lambda\in \cV^{\ge h}$, let  $S_{\Lambda}$ be the reduced closed subscheme of $\mathrm{OGr}(\frac{t_\Lambda-h}{2},\Omega_\Lambda)$ whose $k$-points are described as follows 
	\begin{align*}
		S_{\Lambda}(k) & = \{  \mathscr{V} \in \mathrm{OGr}(\frac{t_\Lambda-h}{2},\Omega_\Lambda)(k)\mid  \operatorname{dim}_k (\mathscr{V}+\Phi(\mathscr{V}))\le \operatorname{dim}_k (\mathscr{V})+1 \} \\
		& \stackrel{\sim}{\rightarrow} \{\text { special lattices } M \subset V_{\breve{\Q}_p}: \Lambda_{\brZ_p} \subset M \subset   M^\vee \subset \Lambda_{\breve{\Z}_p}^\vee\}.
	\end{align*}

  Similarly, for $\Lambda \in\cV^{\le h}$, let $R_{\Lambda}  $ be the reduced closed subscheme of $\mathrm{OGr}(\frac{h-t_\Lambda}{2},\Omega_{\Lambda^\vee})$  with $k$-points
	\begin{align*}
		R_{\Lambda}(k) & = \{   \mathscr{V}  \in \mathrm{OGr}(\frac{h-t_\Lambda}{2},\Omega_{\Lambda^\vee})(k) \mid  \operatorname{dim}_k (\mathscr{V}+\Phi(\mathscr{V}))\le \operatorname{dim}_k (\mathscr{V})+1  \} \\
		& \stackrel{\sim}{\rightarrow} \{\text { special lattices } M \subset V_{\breve{\Q}_p}: p\Lambda_{\breve{\Z}_p}^\vee \subset p M^\vee \subset M \subset \Lambda_{\breve{\Z}_p}\}.
	\end{align*}
Note that when $h=0$, $S_\Lambda$ has two connected components (see \cite[Proposition 5.3.2]{HowardPappas}). In this case, we write  $S_\Lambda$ as $S_\Lambda=S_\Lambda^+\cup S_\Lambda^-$.  Note that the definition of $S_\Lambda$ and $R_\Lambda$ also depend on the fixed $h$. We will denote it as  $S_\Lambda^{[h]}$ and $R_\Lambda^{[h]}$ if we want to emphasize the $h$ here. 

Moreover, for certain $\bV$ and $h$, $\cV^{\ge h}$ (resp.  $\cV^{\le h}$) might  only contain vertex lattices of type $h$ so that $S_\Lambda$ (resp. $R_\Lambda$) is simply a point. For example, when $h=2$ and $\epsilon(\bV)=-1$, then  $\cV^{\le h}$ is the set of vertex lattices of type $2$ since there is no self-dual lattice in $\bbV$ by  $\epsilon(\bV)=-1$.

For a quadratic lattice $\Lambda$, we use $\Lambda^{(p)}$ to denote the same lattice with quadratic form $q(\,)=p\cdot q(\,)_\Lambda$. We use $\Lambda^\sharp$ to denote $(\Lambda^\vee)^{(p)}$.
\begin{lemma}\label{lem: duality DL var}
    Assume    $\Lambda $ is a vertex lattice of type $ t_\Lambda\le  h$ in $\bV^\sharp$ so that $\Lambda^\sharp \subset \bV$ is a vertex lattice of type $t_{\Lambda^\sharp}=n-t_\Lambda\ge n- h$.  Then we have 
    \begin{align*}
       R_{\Lambda}^{[h]} \cong S_{\Lambda^\sharp}^{[n-h]}.
    \end{align*}
\end{lemma}
\begin{proof}
 Note that we have the following natural identification of quadratic spaces that preserves the quadratic forms: $$\Lambda/p\Lambda^\vee \cong (p^{-1}\Lambda )/\Lambda^\vee\cong (p^{-1}\Lambda )^{(p)}/(\Lambda^\vee)^{(p)}=  (\Lambda^\sharp)^\vee/\Lambda^\sharp.$$
 This identification induces an identification of \begin{align*} \mathrm{OGr}(\frac{h-t_\Lambda}{2},\Omega_{\Lambda^\vee})   \quad  \text{and} \quad  \mathrm{OGr}(\frac{t_{\Lambda^\sharp}-(n-h)}{2},\Omega_{\Lambda^\sharp}) ,
 \end{align*} which restricts to an identification between $R_\Lambda^{[h]}$ and $S_{\Lambda^\sharp}^{n-h}$.
\end{proof}
Because of the above lemma, we will focus on $S_\Lambda$ from now on. First, we consider a stratification of $S_\Lambda$. 
Assume $\Lambda\subset \Lambda'$ and $t(\Lambda')\ge h$, then we may regard   $S_{\Lambda'}$ as a closed subscheme of $S_{\Lambda}$ as follows. Let $\pi: \Lambda_{\brZ_p}^\vee/\Lambda_{\brZ_p}\to \Lambda_{\brZ_p}^\vee/\Lambda_{\brZ_p}'$ be the natural quotient map. Then for $\scrV'\in S_{\Lambda'}$, one can check that $\pi^{-1}(\scrV')\in S_{\Lambda}$. Now let 
\begin{align*}
    S_\Lambda^\circ \coloneqq S_\Lambda \setminus \cup_{\Lambda \subsetneq\Lambda'}S_{\Lambda'}.
\end{align*}
By construction, we have 
\begin{align*}
    S_\Lambda =   \cup_{\Lambda\subset\Lambda'}S_{\Lambda'}^\circ.
\end{align*}

\subsubsection{}
We need to consider the following natural decomposition of $S_{\Lambda}$ later:
\begin{align}\label{eq: decom of SLambda}
    S_{\Lambda}=S_{\Lambda}^{\smallheartsuit}\sqcup S_\Lambda^{\smalldagger}\sqcup S_{\Lambda}^{\Phi}
\end{align}
where   $S_\Lambda^{\smallheartsuit}$ is the  open subvariety of $S_\Lambda$ whose $k$-points are 
\begin{align*}
		S_{\Lambda}^{\smallheartsuit}(k) & = \{\text {isotropic subspace } \mathscr{V}\in S_{\Lambda}(k)\mid \Phi(\scrV)\not\subset \scrV^{\perp} \} \\
		 & = \{\text {isotropic subspace } \mathscr{V}\in S_{\Lambda}\mid \scrV+\Phi(\scrV) \text{ is not totally isotropic}\},
	\end{align*}
  $S_\Lambda^{\smalldagger}$ is the subvariety of $S_\Lambda$ such that  
\begin{align*}
    S_\Lambda^{\smalldagger}(k)=\{\scrV \in S_\Lambda(k) \mid \scrV+\Phi(\scrV) \text{ is totally isotropic and $\scrV\neq \Phi(\scrV)$}\},
\end{align*}
and 
$S_\Lambda^\Phi$ is the subvariety of $S_\Lambda$ such that  
\begin{align*}
    S_\Lambda^\Phi(k)=\{\scrV \in S_{\Lambda}(k) \mid \Phi(\scrV)=\scrV \}.
\end{align*}

Let $F_{\frac{t-h}{2}-i}^{\smalldagger}\coloneqq \cap_{\ell=0}^i\Phi^\ell(\scrV)$ and $F_{\frac{t-h}{2}+i}^{\smalldagger}\coloneqq \sum_{\ell=0}^i\Phi^\ell(\scrV)$ for $i\ge 0$. Let $S_\Lambda^0$ be the subvariety of $S_\Lambda^{\smalldagger}$ such that  
\begin{align*}
    S_\Lambda^0(k)=\{\scrV \in S_{\Lambda}(k) \mid F_i^{\smalldagger} \text{ is totally isotropic for all $i$} \}.
\end{align*}

Assume $\Lambda'$ is another vertex lattice that contains $\Lambda$ and set $\Omega_{\Lambda'/\Lambda}:=\Lambda'/\Lambda\otimes_{\bbF_p}k$.  
Let $S_{\Lambda',\Lambda}$ denote the subvariety of $S_{\Lambda}$ such that 
\begin{align*}
    S_{\Lambda',\Lambda}(k)=\{\scrV \in S_{\Lambda}\mid \scrV \subset \Omega_{\Lambda'/\Lambda}\}.
\end{align*}
The subvariety of the form $S_{\Lambda',\Lambda}$ will naturally show up later.
Note that if $\scrV \in S_{\Lambda',\Lambda}$, then $\Phi^i(\scrV)\subset \Omega_{\Lambda'/\Lambda}$ for any $i\ge 0$. In fact, since $\Omega_{\Lambda'/\Lambda}$ is already totally isotropic, any subspace of it is automatically totally isotropic. So we can simply define $S_{\Lambda',\Lambda}$ as a subvariety of the usual Grassmannian $\Gr(\frac{t_{\Lambda}-h}{2},\Omega_{\Lambda'/\Lambda})$ such that 
\begin{align*}
    S_{\Lambda',\Lambda}(k)=\{\scrV \in \Gr(\frac{t_{\Lambda}-h}{2},\Omega_{\Lambda'/\Lambda})(k) \mid  \operatorname{dim}_k (\mathscr{V}+\Phi(\mathscr{V}))\le \operatorname{dim}_k (\mathscr{V})+1 \}.
\end{align*}
Since for each $\scrV\in S_\Lambda^0(k)$, $F_\bullet^{\smalldagger}(\scrV)$ will eventually stabilize to the base change of some rational subspace of $\Omega_{\Lambda}$, we have the following decomposition:
\begin{align*}
    S_\Lambda^0=\cup_{\substack{\Lambda\subset \Lambda',\\ t(\Lambda')\le h}} S_{\Lambda',\Lambda}.
\end{align*}

\subsubsection{} We will need the following description of tangent spaces of $S_\Lambda^{\smallheartsuit}$ in \S \ref{sec: bruhat tits stratification}. 
\begin{proposition}\label{prop: tangent space of SLambda}
   Assume  $z=[\scrV]\in S_\Lambda^{\smallheartsuit}(k)$, and let $T_z{S_{\Lambda}}$ denote the tangent space of $S_\Lambda$ at $z$. Then we have 
   \begin{align*}
       T_{z}S_{\Lambda}\cong \Hom( \scrV+\Phi(\scrV)/\Phi(\scrV),\Omega_{\Lambda}/(\scrV+\Phi(\scrV))).
   \end{align*}
\end{proposition}
\begin{proof}
Let $\widetilde{\scrV}\in T_zS_\Lambda$, which we consider as an element of $S_{\Lambda}^{\smallheartsuit}(k[\epsilon])$ with $\widetilde{\scrV}\otimes_{k[\epsilon]} k = \scrV$; here $k[\epsilon]$ is the ring of dual numbers. Then $\widetilde{\scrV}+\Phi(\widetilde{\scrV})$ is finite free of rank $\frac{t_\Lambda-h}{2}+1$ over $k[\epsilon]$ and lifts $\scrV+\Phi(\scrV)$, hence gives us an element of the tangent space of the usual Grassmannian $\mathrm{Gr}(\frac{t_\Lambda-h}{2}+1,\Omega_\Lambda)$ at the point $\scrV+\Phi(\scrV).$ It follows that we obtain a map of vector spaces \begin{equation}\label{eqn: tangent map 1} T_z S_\Lambda\rightarrow\Hom_k(\scrV+\Phi(\scrV), \Omega_\Lambda/(\scrV+\Phi(\scrV))).\end{equation}
Note that since $\epsilon^p=0$, we have $\Phi(\widetilde{\scrV})=\Phi(\scrV)\otimes_kk[\epsilon].$ Thus \eqref{eqn: tangent map 1} factors through a map \begin{equation}\label{eqn: tangent space map 2}T_zS_\Lambda\rightarrow \Hom_k( \scrV+\Phi(\scrV)/\Phi(\scrV),\Omega_{\Lambda}/(\scrV+\Phi(\scrV))).
\end{equation}

We claim that this map is injective. Let $\widetilde{\scrV}$ be an element of $T_zS_\Lambda$, which we identify with a homomorphism $\lambda:\scrV\rightarrow \Omega_\Lambda/\scrV$ by considering $\widetilde{\scrV}$ as a $k[\epsilon]$-point of $\Gr(\frac{t_\Lambda-h}{2},\Omega_\Lambda).$ Then $\lambda$ factors through $\scrV/\scrV\cap\Phi(\scrV)$, and its image $\gamma$ under \eqref{eqn: tangent space map 2} is given by composing with the natural isomorphism $\scrV/\scrV\cap\Phi(\scrV)\cong \scrV+\Phi(\scrV)/\Phi(\scrV)$, and the projection $\Omega_\Lambda/\scrV\rightarrow\Omega_\Lambda/(\scrV+\Phi(\scrV)).$ We will show that $\lambda$ is determined by $\gamma.$

Let $v\in \scrV$ be a generator of the 1-dimensional vector space $\scrV/\scrV\cap\Phi(\scrV)\cong \scrV+\Phi(\scrV)/\Phi(\scrV)$ and fix a lifting $\tilde{\gamma}:\scrV+\Phi(\scrV)/\Phi(\scrV)\rightarrow \Omega_\Lambda/\scrV$ of $\gamma$. Then $\scrV+\Phi(\scrV)$ is generated by $\scrV$ and $\Phi(v)$, and for any $u\in \scrV$, we have $\lambda(u)=c_u\Phi(v)+\tilde{\gamma}(u)$ for some $c_u\in k$. By   assumption $\scrV+\Phi(\scrV)$ is not isotropic, and hence $(\Phi(v),v)\neq 0.$ Since $\widetilde{\scrV}$ is isotropic, we have $(u+\epsilon\lambda(u),v+\epsilon\lambda(v))=0$ for all $u\in \scrV$, or equivalently $$(u,c_v\Phi(v)+\tilde{\gamma}(v))+(c_u\Phi(v)+\tilde{\gamma}(u),v)=0.$$ Taking $u=v$, we find that $c_v=-\frac{(v,\tilde{\gamma}(v))}{(\Phi(v),v)}$, and for general $u\in \scrV$, we have
$$c_u=-\frac{(u,c_v\Phi(v)+\tilde{\gamma}(v))+(\tilde{\gamma}(u),v)}{(\Phi(v),v)}.$$

This shows that \eqref{eqn: tangent space map 2} is injective. By Theorem \ref{thm: decom pf SLambda}, we have $\dim_kS_\Lambda=\frac{t_\Lambda+h}{2}-1,$ and hence $$\dim_kT_zS_\Lambda\geq \dim_k\Hom_k(\scrV+\Phi(\scrV)/\Phi(\scrV),\Omega_\Lambda/(\scrV+\Phi(\scrV)))=\frac{t_\Lambda+h}{2}-1.$$ It follows that \eqref{eqn: tangent space map 2} is an isomorphism.
\end{proof}

\subsection{Deligne-Lusztig varieties}\label{sec: DL for SO}

We relate the $S_\Lambda$ and $R_\Lambda$ introduced in the last subsection with certain Deligne-Lusztig varieties in this subsection. It suffices to study $S_\Lambda$ by Lemma \ref{lem: duality DL var}. Recall that $S_\Lambda=S_{\Lambda}^{\smallheartsuit} \sqcup S_\Lambda^{\smalldagger}\sqcup S_\Lambda^\Phi$.
We will relate $S_{\Lambda}^{\smallheartsuit}$ and $S_\Lambda^{\smalldagger}$ with certain Deligne-Lusztig varieties.

\subsubsection{The even dimensional orthogonal case}
Let $\Lambda\subset \bV$ be a vertex lattice of even type $t=2d\ge h.$ In particular, according to Proposition \ref{prop: V^Phi}, $h$ is also even.
Recall that $\Omega_\Lambda=V_{\Lambda}\otimes_{\F_\p} k$ is a quadratic space over $k$, and we fix a basis $\left\{e_1, \ldots, e_d, f_1, \ldots, f_d\right\}$ of $\Omega_\Lambda$  such that $\operatorname{Span}_k\left\{e_1, \ldots, e_d\right\}$ and $\operatorname{Span}_k\left\{f_1, \ldots, f_d\right\}$ are totally isotropic, $\left(e_i, f_j\right)=\delta_{i, j}$. If $V_{\Lambda}$ is split, then we may choose this basis so that Frobenius $\Phi$ fixes $e_1, \ldots, e_{d}, f_1, \ldots, f_{d}$. If $V_{\Lambda}$ is non-split, then $\Phi$ fixes $e_1, \ldots, e_{d-1}, f_1, \ldots, f_{d-1}$ but interchanges $e_d \leftrightarrow f_d$.

Consider the isotropic flags $\mathscr{F}_{\bullet}^{+}$and $\mathscr{F}_{\bullet}^{-}$in $\Omega$ defined by
$$
\begin{aligned}
\mathscr{F}_i^{ \pm} & =\operatorname{Span}_k\left\{e_1, \ldots, e_i\right\} \text { for } 1 \leqslant i \leqslant d-1 \\
\mathscr{F}_d^{+} & =\operatorname{Span}_k\left\{e_1, \ldots, e_{d-1}, e_d\right\} \\
\mathscr{F}_d^{-} & =\operatorname{Span}_k\left\{e_1, \ldots, e_{d-1}, f_d\right\} .
\end{aligned}
$$
Hence when $V_{\Lambda}$ is split, we have $\Phi(\mathscr{F})_{\bullet}^{ \pm}= \mathscr{F}_{\bullet}^{\pm }$. When $V_{\Lambda}$ is non-split, we have $\Phi(\mathscr{F})_{\bullet}^{ \pm}= \mathscr{F}_{\bullet}^{\mp }$. 
The stabilizers of  $\Phi(\mathscr{F})_{\bullet}^{\pm}$ are the same, which we denote as $B \subset H=\mathrm{SO}(\Omega_\Lambda)$. Then $B$ is a $\Phi$-stable Borel subgroup containing $T$ where $T$ is a maximal $\Phi$-stable torus $T \subset \mathrm{SO}(\Omega_\Lambda)$. 

We use $\Delta^*=\left\{s_1, \ldots, s_{d-2}, t^{+}, t^{-}\right\}$ to denote the  set of corresponding simple reflections in the Weyl group $W=N(T) / T$ where
\begin{enumerate}
    \item $s_i$ interchanges $e_i \leftrightarrow e_{i+1}$ and $f_i \leftrightarrow f_{i+1}$, and fixes the other basis elements.
    \item $t^{+}$ interchanges $e_{d-1} \leftrightarrow e_d$ and $f_{d-1} \leftrightarrow f_d$, and fixes the other basis elements.
    \item $t^{-}$ interchanges $e_{d-1} \leftrightarrow f_d$ and $f_{d-1} \leftrightarrow e_d$, and fixes the other basis elements.
\end{enumerate}

For  $0\le r\le \frac{t-h}{2}-1$, let
\begin{align*}
    I_{r}\coloneqq
    \begin{cases}
       \{s_1,\ldots,s_{d-2-r}\}   & \text{ if $h\le 2$},\\
      \{s_1,\ldots,s_{\frac{t-h}{2}-1-r}, s_{\frac{t-h}{2}+1},\ldots,s_{d-2},t^+,t^-\}     & \text{ if $h>2$},
    \end{cases}
\end{align*}
and $P_{r}$ be the corresponding parabolic subgroup. Here, when  $r= \frac{t-h}{2}-1$ we mean $I_r$ is empty and  $P_{r}=B$.   

For $1\le r\le \frac{t-h}{2}$ ($r\le d-1$ when $h=0$), let
\begin{align}\label{eq: word 1}
     &\begin{cases} 
     w_r^{\pm}=  t^{\mp}  s_{d-2} \cdots s_{d-r} & \text{ if $h=0$},\\
     w_r=   t^{-}t^{+} s_{d-2} \cdots s_{\frac{t-h}{2}-r+1}  & \text{ if $h=2$},\\
      w_r=   s_{\frac{t-h}{2}}\cdots s_{d-2}t^{-}t^{+}s_{d-2} \cdots s_{\frac{t-h}{2}-r+1} & \text{ if $h>2$.}
    \end{cases}\\  \label{eq: word 2}
  &\begin{cases}
       w_r^{\prime\pm}=t^\pm s_{\frac{t-h}{2}-1}\cdots s_{\frac{t-h}{2}-r+1} & \text{ if $h=2$},\\
       w_r'=s_{\frac{t-h}{2}}\cdots s_{\frac{t-h}{2}-r+1} & \text{ if $h>2$}. 
  \end{cases} 
\end{align}  
Here we interpret $w_{1}^{\pm}$ as $t^\mp$ when $h=0$ and  $w_{1}$ as $t^{-}t^+=t^+t^-$ when $h=2$. Moreover, we define $w_0$ and $w_0^\pm$ as the identity. Also, note that $w_1'^{\pm}=t^{\pm}$ if $h=2$ and $w_1'=s_\frac{t-h}{2}$ when $h>2$.

\subsubsection{The odd dimensional orthogonal case}
In this subsection, we let $\Lambda\subset \bV$ be a vertex lattice of odd type $t=2d+1\ge h.$ In particular, according to Proposition \ref{prop: V^Phi}, $h$ is also odd.
Let $\Omega_\Lambda=V_{\Lambda}\otimes_{\F_\p} k$ be the associated quadratic space over $k$. Since $V_{\Lambda}$ is of odd dimension,  we can choose a basis $\left\{e_1, \ldots, e_d, f_1, \ldots, f_d, e_{2d+1}\right\}$ of $\Omega_\Lambda$  such that $\operatorname{Span}_k\left\{e_1, \ldots, e_d\right\}$ and $\operatorname{Span}_k\left\{f_1, \ldots, f_d\right\}$ are totally isotropic, $\left(e_i, f_j\right)=\delta_{i, j}$, $(e_{2d+1},e_{2d+1})=1$, $e_{2d+1}$ is orthogonal to all the other basis vectors and $\Phi$ fixes all the basis vectors.

Consider the isotropic flags $\mathscr{F}_{\bullet}$ in $\Omega_\Lambda$ defined by
$$
\begin{aligned}
\mathscr{F}_i & =\operatorname{Span}_k\left\{e_1, \ldots, e_i\right\} \text { for } 1 \leqslant i \leqslant d.
\end{aligned}
$$
The stabilizer of  $\mathscr{F}_{\bullet}$ is denote as $B \subset H=\mathrm{SO}(\Omega_\Lambda)$. Hence $B$ is a $\Phi$-stable Borel subgroup containing $T$ where $T$ is a maximal $\Phi$-stable torus $T \subset \mathrm{SO}(\Omega_\Lambda)$. 

We use $\Delta^*=\left\{s_1, \ldots, s_{d}\right\}$ to denote the  set of corresponding simple reflections in the Weyl group $W=N(T) / T$ where
\begin{enumerate}
    \item  for $i=1,\dotsc,d-1$, $s_i$ interchanges $e_i \leftrightarrow e_{i+1}$ and $f_i \leftrightarrow f_{i+1}$, and fixes the other basis elements.
    \item $s_d$ interchanges $e_{d} \leftrightarrow f_d$, sends $e_{2d+1}$ to $-e_{2d+1}$   and fixes the other basis elements.
\end{enumerate}

For  $0\le r  \le \frac{t-h}{2}-1$, let
\begin{align*}
    I_{r}\coloneqq
      \{s_1,\ldots,s_{\frac{t-h}{2}-1-r}, s_{\frac{t-h}{2}+1},\ldots,s_{d-2},s_{d-1},s_{d}\}.  
\end{align*}
and $P_{r}$ be the corresponding parabolic subgroup. Here, when   $r\ge d-1$, we mean $I_r$ is empty and  $P_{r}=B$.

For $1\le r\le \frac{t-h}{2}$, let
\begin{align} \label{eq: word 3}
    w_r&=\begin{cases}
     s_d    \cdots s_{d-r+1} & \text{ if $h=1$},\\
         s_{\frac{t-h}{2}}\cdots   s_{d-1}s_ds_{d-1} \cdots s_{\frac{t-h}{2}-r+1} & \text{ if $h>1$.}
    \end{cases}\\  \label{eq: word 4}
    w_r'&=s_{\frac{t-h}{2}}\cdots s_{\frac{t-h}{2}-r+1}  \text{ if $h>1$.} 
\end{align}
Here we interpret $w_{1}$ as $s_{d}$ when $h=1$ and  $w_{1} $ as $s_{d-1}s_{d}s_{d-1}$ when $h=3$. Moreover, we define $w_0$ as the identity. Also, note that  $w_1'=s\frac{t-h}{2}$ when $h>1$.

\subsubsection{Special linear case}

Assume $\Lambda\subset \Lambda'$ is a pair of vertex lattices in $\bV$ of types $t,t'$ respectively. Set $V_{\Lambda'/\Lambda}\coloneqq\Lambda'/\Lambda\subset \Lambda^\vee/\Lambda=V_{\Lambda}$ and $\Omega_{\Lambda'/\Lambda}\coloneqq V_{\Lambda'/\Lambda}\otimes_{\F_p}k $. Note that since $\Lambda'\subset \Lambda'^\vee$, $\Omega_{\Lambda'/\Lambda}$ is an isotropic subspace of $V_{\Lambda}$. Moreover, each rational isotropic subspace of $\Omega_{\Lambda}$ is realized in this way.

Let $\left\{e_1, \ldots, e_{\frac{t-t'}{2}}\right\}$ be a basis of $\Omega_{\Lambda'/\Lambda}$ such that $\Phi(e_i)=e_i$. 
Consider the standard flag  
\begin{align*}
\mathscr{F}_i& =\operatorname{Span}_k \{e_1, \ldots, e_i\} \text { for } 1 \leqslant i \leqslant \frac{t-t'}{2}  
\end{align*}
 
The stabilizer of  $\mathscr{F}_{\bullet}$ is  a $\Phi$-stable Borel subgroup $B \subset H=\mathrm{SL}(\Omega_{\Lambda'/\Lambda})$, which contains  a maximal $\Phi$-stable torus $T \subset \mathrm{SL}(\Omega_{\Lambda'/\Lambda})$.

We use $\{s_1, \ldots, s_{\frac{t-t'}{2}-1}\}$ to denote the set of corresponding simple reflections in the Weyl group $W=N(T)/T$ where $s_i$ exchanges $e_i$ and $e_{i+1}$.

For  $0\le r\le \frac{t-h}{2}$ and $0\le s\le \frac{h-t'}{2}$, let
\begin{align*}
    I_{r,s}\coloneqq \{s_1,\ldots,s_{\frac{t-h}{2}-1-r},s_{\frac{t-h}{2}+1+s},\ldots, s_{\frac{t-t'}{2}} \}
\end{align*}
and $P_{r,s}$ be the corresponding parabolic subgroup. 

For $0\le r\le \frac{t-h}{2}-1$ and $1\le s\le \frac{h-t'}{2}+1$, let
\begin{align}\label{eq: word 5}
    w_{r,s}\coloneqq \begin{cases}
      (s_{\frac{t-h}{2}}\cdots s_{\frac{t-h}{2}+s-1}) (s_{\frac{t-h}{2}-1}\cdots s_{\frac{t-h}{2}-r}) & \text{ if $r\neq 0$},\\
      s_{\frac{t-h}{2}}\cdots s_{\frac{t-h}{2}+s-1} & \text{ if $r=0$.}
      \end{cases} 
\end{align}
Then $w_{r,s}$ sends $e_{\frac{t-h}{2}-r}$ to $e_{\frac{t-h}{2}+1}$, $e_{\frac{t-h}{2}-i}$ to $e_{\frac{t-h}{2}-i-1}$ if $0\le i\le r-1$, $e_{\frac{t-h}{2}+i}$ to $e_{\frac{t-h}{2}+i+1}$ if $1\le i\le s-1$, $e_{\frac{t-h}{2}+s}$ to $e_{\frac{t-h}{2}}$ and fixes all the other $e_i$. 
Moreover, we set $w_{0,0}$ as the identity.

\subsubsection{Deligne-Lusztig varieties}\label{sec: DL varieties}

In general, for any subset $I\subset \Delta^*$, we let $W_I\subset W$ be the subgroup generated by simple reflections in $I$ and $P_I$ be the corresponding parabolic subgroup. The Bruhat decomposition implies that we have a bijection $$\inv: P_I(k)\backslash G(k)/P_{\Phi(I)}(k)\cong W_I\backslash W/W_{\Phi(I)}.$$
\begin{definition}
    For $w \in W_I \backslash W / W_{\Phi(I)}$, we define the Deligne-Lusztig variety $X_{P_I}(w)$ to be the locally closed reduced subscheme of $G / P_I$ whose $k$-points are
$$
X_{P_I}(w)=\left\{g P_I \in G / P_I: \operatorname{inv}(g, \Phi(g))=w\right\}.
$$
\end{definition}

First, we collect two theorems that will be used later.
\begin{theorem}\cite[Theorem 3]{ramanan1985projective}\label{thm: normality of Schuber var}
    Let $k$ be a field of arbitrary characteristic. Let $G$ be a reductive group 
over $k$ and $P$ a parabolic subgroup. Then any Schubert variety $X$ in $G/P $ is 
normal. Moreover the linear system on $X$ given by an ample line bundle on $G/P $ 
embeds $X$ as a projectively normal variety. 
\end{theorem}

\begin{theorem}\cite[Theorem 2]{BRirred}\label{thm: irr of DL}
    Let $I\subset S$ and let $w\in W$. Then $X_{P_I}(w)$ is irreducible if and only if $W_I w$ is not contained in a proper $\Phi$-stable standard parabolic subgroup of $W$.
\end{theorem}

\begin{theorem}\label{thm: decom pf SLambda}
 For a vertex lattice $\Lambda$ of type $t\ge h$,  we have
\begin{align*} 
    S_{\Lambda}=S_{\Lambda}^{\smallheartsuit}\sqcup S_\Lambda^{\smalldagger}\sqcup S_{\Lambda}^{\Phi},
\end{align*}
where 
\begin{align*}
    S_{\Lambda}^{\smallheartsuit} &\cong \begin{cases}
         (\sqcup_{i=0}^{\frac{t}{2}-1} X_{P_{i}}(w_i^+ ))\sqcup (\sqcup_{i=0}^{\frac{t}{2}-1} X_{P_{i}}(w_i^- ))& \text{ if $h=0$},\\
         X_{P_0}(w_1)  \cong \sqcup_{i=1}^{\frac{t-h}{2}} X_{P_{i}}(w_i)  & \text{ if $h>0$},
    \end{cases}\\
    S_\Lambda^{\smalldagger}&\cong  \begin{cases}
      \emptyset & \text{ if $h\le 1$},\\
       X_{P_0}(w_1'^+)\sqcup X_{P_0}(w_1'^-)  \cong (\sqcup_{i=1}^{\frac{t-h}{2}} X_{P_{i}}(w_i^{\prime +}))  \sqcup  (\sqcup_{i=1}^{\frac{t-h}{2}} X_{P_{i}}(w_i^{\prime -})) & \text{ if $h=2$},\\
        X_{P_0}(w_1') \cong \sqcup_{i=1}^{\frac{t-h}{2}} X_{P_{i}}(w_i') & \text{ if $h>2$},
      \end{cases}\\
      S_\Lambda^\Phi&\cong 
      \begin{cases}
      \emptyset & \text{ if $h=0$},\\
          X_{P_0}(1) & \text{ if $h\neq 0$}.
      \end{cases}
\end{align*}
Moreover,  if $t=h$, then $S_\Lambda$ is a point. If $t>h$, then it is an irreducible normal variety of dimension $(t+h)/2-1$.    For   a vertex lattice $\Lambda$ of type $t\le h$, similar results hold for $R_\Lambda$ by Lemma \ref{lem: duality DL var}.
\end{theorem}

\begin{proof}
By Lemma \ref{lem: duality DL var}, it suffices to study $S_\Lambda$. 
The case $h=0$ is  \cite[Lemma 3.7]{howardpappas2014supersingular}. So we assume $h>0$. To simplify notation, we let $m=\frac{t-h}{2}$.  

First, we prove 
\begin{align*}
    S_\Lambda\cong   \begin{cases}
      X_{P_0}(1)\cup X_{P_0}(w_1) & \text{ if $h=1$},\\
      X_{P_0}(1)\cup X_{P_0}(t^+)\cup X_{P_0}(t^-) \cup X_{P_0}(w_1) & \text{ if $h=2$},\\
      X_{P_0}(1)\cup X_{P_0}(s_{\frac{t-h}{2}}) \cup X_{P_0}(w_1) & \text{ if $h>2$}.
      \end{cases}
  \end{align*}
We assume $t$ is even first and let $d=\frac{t}{2}$.
Recall that we have
\begin{align*}
    S_{\Lambda}(k) & = \{\text {isotropic subspace}\, \mathscr{V} \subset \Omega_{\Lambda }: \operatorname{dim}_k (\mathscr{V} )=\frac{t-h}{2},\, \operatorname{dim}_k (\mathscr{V}+\Phi(\mathscr{V}))\le \frac{t-h}{2}+1 \}.
\end{align*}
In order to relate $S_{\Lambda}$ and Deligne-Lusztig varieties in the flag variety $H/P_0$, we consider the map $g\mapsto g \mathscr{F}_{\frac{t-h}{2}}$. This reduces the problem to a description of $[w]\in W_{I_0}\backslash W/ W_{\Phi(I_0)}$ such that
$$|\{w(e_1),\ldots,w(e_m)\}\cup \{e_1,\ldots,e_m\}|-|\{e_1,\ldots,e_m\}|\le 1.$$
It suffices to show such double coset is represented by   $w_1$, $w_1'$ or $1$.

Assume $h>2$ first.
If $\{w(e_1),\ldots,w(e_m)\}=\{e_1,\ldots,e_m\}$, then $[w]=1$ in $ W_{I_0}\backslash W/ W_{\Phi(I_0)}$ obviously. 

Now we assume $\{w(e_1),\ldots,w(e_m)\}\neq \{e_1,\ldots,e_m\}$. Assume $w(e_j)\not \in \{e_1,\ldots,e_m\}$ for some $1\le j\le m$. In particular, $w(e_i)\in \{e_1,\ldots,e_m\}$ for $i\neq j$. If $j< m$, then we can replace $w$ by another element as follows. Let $w_{jm}\in W_{I_0}$ be the element that switches $e_j$ and $e_m$. Then $w w_{jm}(e_m)=w(e_j)$ and $ww_{jm}(e_\ell)\in \{e_1,\ldots,e_m\}$ for $\ell<m$. Now we can find $w_1,w_2 \in W_{I_0}$ such that $w_1ww_{jm}w_2$ sends $e_m$ to $w(e_j)$ and fixes $e_{\ell}$ for $\ell <m$. Since $[w_1ww_{jm}w_2]=[w] \in W_{I_0}\backslash W/W_{I_0}$, we assume $w(e_m)\not \in \{e_1,\ldots,e_m\}$ and  $w(e_{\ell})=e_{\ell}$ for $\ell <m$ from now on.

If $w(e_m)=f_m$, then $w(f_m)=e_m$. In this case, we claim $[w]=[w_1]$. Note that  $w(e_\ell)$ and $w(f_\ell)$  lie in $\{e_{m+1},\ldots,e_d,f_{m+1},\ldots,f_d\}$ for $\ell >m$. So we can obviously find $w_1,w_2\in W_{I_0}$ such that $w_1ww_2\in [w_1]$ where $w_1$ switches $e_m$ with $f_m$ and $e_d$ with $f_d$ and fixes all the other elements. Therefore, we have $[w]=[w_1]$.

If $w(e_m)\neq f_m$, we claim that $[w]=[s_m]$. In this case, $w(e_m)=e_{\ell}$ or $f_\ell$ for $\ell >m$. We can find $w'\in W_{I_0}$ such that $w'w$ sends $e_m$ to $e_{m+1}$ and fixes $e_i$ for $1\le i\le m-1$.   Assume $w(e_k)=e_m$ or $w(f_k)=e_m$. Note that since $w(e_m)\neq f_m$, we have  $k>m$. Let $w''$ be the element in $W_{I_0}$ that switches $e_{m+1}$ and $e_k$ (resp. $f_k$) if $w(e_k)=e_m$ (resp. $w(f_k)=e_m)$.  
Then $w'ww''$ fixes $e_i$ for $1\le i\le m-1$ and switches $e_m$ and $e_{m+1}$. Now it is clear that we can find $w_1,w_2\in W_{I_0}$ such that $w_1w'ww'' w_2=s_m$. Since $[w_1w'ww'' w_2]=[w]$, we are done.  

When $h=2$, $I_0=\{s_1,\ldots,s_{d-2}\}$. In particular, one can check that $[t^+]\neq [t^-]$. If $[w]\neq 1$, we claim that $[w]=[t^+]$ or $[w]=[t^-]$ or $[w]=[t^-t^+]$. Since the proof is essentially the same as before, we leave the details to the reader.

Now we assume $t$ is odd so that $h$ is also odd. When $h=1$ and $[w]\neq 1$, a similar argument as before shows that $[w]=[t^-]$. When $h>1$, we have $t^\pm \in I_0$. Then a similar argument as before shows that for $[w]\neq 1$, we have $[w]=[s_m]$   or $[w]=[w_1]$.

Now we show $X_{P_0}(w_1)$ and $X_{P_0}(w_1')$ (resp. $X_{P_0}(w_1^{\prime +})\cup X_{P_0}(w_1^{\prime -})$ when $h=2$) parametrizes $\scrV \in S_{\Lambda}^{\smallheartsuit}$ and $S_\Lambda^{\smalldagger}$ respectively. Assume $g\in X_{P_0}(w_1)$, then $\Phi(g)=gp_0w_1p_0'$ where $p_0,p_0'\in P_0$.  Let $g'=gp_0$. Then $g'\scrF_{\frac{t-h}{2}}=g\scrF_{\frac{t-h}{2}}=\scrV$ and $\Phi(g')=g'w_1p_0'p_0$. So $\Phi(g')\scrF_{\frac{t-h}{2}}=g'w_1 \scrF_{\frac{t-h}{2}}$. Since $w_1 \scrF_{\frac{t-h}{2}}+\scrF_{\frac{t-h}{2}}$ is not totally isotropic, we know $\scrV+\Phi(\scrV)=g'\scrF_{\frac{t-h}{2}}+g'w_1\scrF_{\frac{t-h}{2}}$ is not totally isotropic. Hence we know $ X_{P_0}(w_1)\subset  S_{\Lambda}^{\smallheartsuit}$. Similarly, since 
$w_1' \scrF_{\frac{t-h}{2}}+\scrF_{\frac{t-h}{2}}$ (resp. $w_1^{\prime \pm} \scrF_{\frac{t-h}{2}}+\scrF_{\frac{t-h}{2}}$) is totally isotropic when $h>2$ (resp. when $h=2)$, the same argument as before shows that $ X_{P_0}(w_1')\subset  S_{\Lambda}^{\smalldagger}$ (resp. $ X_{P_0}(w_1^{\prime \pm})\subset  S_{\Lambda}^{\smalldagger}$).

The decomposition of $X_{P_0}(w_1)$ and $X_{P_0}(w_1')$ into finer DL varieties follows from the same proof of \cite[Proposition 5.5]{RTW} (see also \cite[Proposition 3.8]{howardpappas2014supersingular})  so we omit the details here. \footnote{In fact the statement of \cite[(5.4)]{RTW} needs to be rephrased a little bit. The claim $\pi^{-1}\left(X_{P_{j-1}}\left(w_j\right)\right)=X_{P_j}\left(w_j\right) \biguplus X_{P_j}\left(w_{j+1}\right)$ there should be changed to $\pi$ restricts to an isomorphism on $X_{P_j}\left(w_j\right) \biguplus X_{P_j}\left(w_{j+1}\right)$.}

Now we $S_\Lambda$ is irreducible of dimension $\frac{t+h}{2}-1$. When $t=h$, we can see that $S_\Lambda$ is a point by definition. So we assume $t>h$.
    Combining G\"ortz's local model diagram \cite[\S 5.2]{görtz_yu_2010} with Theorem \ref{thm: normality of Schuber var}, we obtain the normality of $S_\Lambda$. By Theorem \ref{thm:  irr of DL}  and  Proposition \ref{prop: dim of DL}, $X_{P_{\frac{t-h}{2}}}(w_{\frac{t-h}{2}})$ is irreducible of dimension $\frac{t+h}{2}-1$. Since $X_{P_{\frac{t-h}{2}}}(w_{\frac{t-h}{2}})$ is an open dense subscheme of $S_\Lambda$, we know $S_\Lambda$ is irreducible of dimension $\frac{t+h}{2}-1$.
\end{proof}
 
\begin{remark}\label{rem: rem KR 2}
Later on, we will identify a special lattice $M$ containing $\Lambda$ with a point $M/\Lambda \in S_\Lambda$. Under this identification,   $M/\Lambda \in S_\Lambda^{\smalldagger}$ exactly when  $M$ satisfies the condition $\Phi(M)\subset M^\vee$ in Proposition \ref{prop: KR stratification}. A comparison with \S \ref{sec: KR strata} shows that this implies that $M$ lies in two types of KR-strata. See \S 5.3 and Remark 7.12 of \cite{HLY} for a similar phenomenon in the unitary RZ space over ramified primes.
\end{remark}
\begin{remark}
  The decomposition of $S_\Lambda$ in Theorem \ref{thm: decom pf SLambda} is based on the behavior of  the filtration : $F_{\frac{t-h}{2}-i}\coloneqq \cap_{\ell=0}^i\Phi^\ell(\scrV)$  induced by  $\mathscr{V}$. By modifying \cite[Proposition 5.5]{RTW}, one can obtain a   decomposition of $S_\Lambda^{\smalldagger}$ into finer Deligne-Lusztig varieties by considering a longer filtration: $F_{\bullet}^{\smalldagger}$ where $F_{\frac{t-h}{2}-i}^{\smalldagger}\coloneqq \cap_{\ell=0}^i\Phi^\ell(\scrV)$  and $F_{\frac{t-h}{2}+i}^{\smalldagger}\coloneqq \sum_{\ell=0}^i\Phi^\ell(\scrV)$ for $i\ge 0$. Note that $\mathscr{V}+\Phi(\mathscr{V})$ is not totally isotropic for any $\mathscr{V}\in S_{\Lambda}^{\smallheartsuit}(k)$, so $F_{\bullet}^{\smalldagger}=F_{\bullet}$ in this case. See \cite[\S 5]{HLY} for such decomposition via a different approach with the one given here. 
\end{remark}

By a similar argument as above,  the map $\scrV \mapsto F_\bullet^{\smalldagger}$ induces the following isomorphism. 
\begin{proposition}\label{prop: decomp of SLambdaLambda'}
Assume $\Lambda $ is a vertex lattice with $ h\le t(\Lambda)$. Then we have
\begin{align*}
       S_\Lambda^0=\cup_{\substack{\Lambda\subset \Lambda',\\ t(\Lambda')\le h}} S_{\Lambda',\Lambda},
\end{align*}
where
     \begin{align*}
       S_{\Lambda',\Lambda} \cong
      \begin{cases}
          \emptyset & \text{ if $h\le 1$},\\
           \sqcup_{\substack{0\le r\le \frac{t-h}{2},\\
        0\le s\le \frac{h-t'}{2}}} X_{P_{r,s}}(w_{r,s}) & \text{ if $h> 1$}.
      \end{cases}
  \end{align*}
\end{proposition} \qed

\subsubsection{}\label{sec: fine decom of DL varieties}

In fact, we have a further decomposition of  $X_{P_{i-1}}(w_i)$ and  $X_{P_{r,s}}(w_{r,s})$.  When we want to emphasize the parabolic subgroup $P_i$ (resp. $w_i$) is associated to $\Omega_{\Lambda}$, we denote it as $P_{\Lambda,i}$ (resp. $w_{\Lambda,i}$). When the $i$ in $P_{\Lambda,i}$ (resp. $w_{\Lambda,i}$)   is the largest possible value, we simply denote it as $P_{\Lambda}$  (resp. $w_\Lambda$). We define $P_{\Lambda,\Lambda'}$ and $w_{\Lambda,\Lambda'}$ similarly.

\begin{proposition}\label{prop: further decom of DL}
Assume $\Lambda'$ is a vertex lattice such that $\Lambda\stackrel{i}{\subset }\Lambda'$. Let $P_{\Lambda',i}$ denote the parabolic subgroup of  $\mathrm{SO}(\Omega_{\Lambda'})$ defined as before and $w_{\Lambda'}$ denote the Weyl element of $\mathrm{SO}(\Omega_{\Lambda'})$ induced from $w_i$.   Then we have the following.
\begin{enumerate}
    \item   $X_{P_{\Lambda,i}}(w_{\Lambda,i})=\sqcup_{\Lambda\stackrel{i}{\subset} \Lambda'} X_{P_{\Lambda'}}(w_{\Lambda'})$. 
    \item $X_{P_{\Lambda,\Lambda',r,s}}(w_{r,s})\cong \sqcup_{\substack{\Lambda \stackrel{r}{\subset} \Lambda_1\subset \Lambda_2\stackrel{s}{\subset} \Lambda'\\
       t(\Lambda_2)\le h \le t(\Lambda_1)}} X_{P_{\Lambda_1,\Lambda_2}}(w_{\Lambda_1,\Lambda_2}).$
\end{enumerate}
\end{proposition}

\begin{proof}
For (i), note that    a $k$-point of $X_{P_{\Lambda,i}}(w_i)$ (resp. $X_{P_{\Lambda,i}}(w_i')$) corresponds to a flag of isotropic subspaces  $F_{\frac{t-h}{2}-i}\subset\cdots \subset F_{\frac{t-h}{2}}$ such that $F_{\frac{t-h}{2}-i}=\Phi(F_{\frac{t-h}{2}-i})$. In particular, $F_{\frac{t-h}{2}-i}$ is the base change of a rational subspace that corresponds to a vertex lattice $\Lambda'$ such that $\Lambda\stackrel{i}{\subset} \Lambda'$. Now we regard $F_{\frac{t-h}{2}-i}\subset\cdots \subset F_{\frac{t-h}{2}}$ as a flag in the subspace $\Omega_{\Lambda'}$. This gives the decomposition. 

(ii) is similar and we omit the details.
\end{proof}

\subsection{Dimension of Deligne-Lusztig varieties}
We compute the dimension of the relevant Deligne-Lusztig varieties in this subsection.
\begin{proposition}\cite[Lem. 2.1.3]{Hoeve}\label{prop: dim of DL var}
Assume $I=\Phi(I)$. We have $\operatorname{dim} X_{P_I}(w)=\ell_I(w)-\ell\left(w_I\right)$, where $w_I$ is the longest element in $W_I$, and $\ell_I(w)$ is the maximal length of an element in $W_I w W_I$.    

Moreover, if $w$ is the representative of the double coset $W_IwW_I$ with minimal length, then
\begin{align*}
   \operatorname{dim} X_{P_I}(w)=\ell_I(w)-\ell\left(w_I\right)= \ell(w)+\ell(W_{I})-\ell(W_{I\cap wIw^{-1}}).
\end{align*}
Here to make sense of $wIw^{-1}$, we identify $I$ with the corresponding set of simple reflections $\{s_i \mid i \in I\}$.
\end{proposition}

 \begin{proposition}\label{prop: dim of DL}
 Assume $\Lambda\subset \Lambda'$ are vertex lattices of types $t> h$ and $t'< h$ respectively. Then we have
    \begin{enumerate}
        \item $\dim X_{\Lambda}(w_{\Lambda})=\frac{t+h}{2}-1$.
        \item $\dim  X_{\Lambda}(w_{\Lambda}')=\frac{t+h}{2}-2$.
        \item $\dim X_{\Lambda,\Lambda'}(w_{\Lambda,\Lambda'})=\frac{t-t'}{2}$.
    \end{enumerate}
 \end{proposition}
\begin{proof}
We use Proposition \ref{prop: dim of DL var} to calculate the dimension. We only give the details for the case $h$ is even since the case $h$ is odd is completely analogous. First of all, $w_{\Lambda},w_{\Lambda}'$ and $w_{\Lambda,\Lambda'}$ are the minimal representatives of the corresponding double cosets by Lemma \ref{lem: minimal rep}.

We prove (i) first. Recall that \begin{align*}
    I_\Lambda=I_{\frac{t-h}{2}}=\{ s_{\frac{t-h}{2}+1},\ldots,s_{d-2},t^+,t^-\}.
\end{align*}
We claim that $w_\Lambda I_\Lambda w_\Lambda^{-1}=I_\Lambda$. Assuming this,  (i) follows from Proposition \ref{prop: dim of DL var} and the fact $\ell(w_\Lambda)=\frac{t-h}{2}$ by Lemma \ref{lem: length of w}. 

Now we show the claim. Let $w_\Lambda=w_1 w_2$ where $w_1=s_{\frac{t-h}{2}}\cdots s_{d-2}t^+t^-s_{d-2}\cdots s_{\frac{t-h}{2}}$ and $w_2=s_{\frac{t-h}{2}-1}\cdots s_{1}$. Since $w_2$ commutes with element in $W_I$, it suffices to show that $  w_1I_\Lambda w_1^{-1}=I_\Lambda$. One can check that $$w_1\cdot e_{\frac{t-h}{2}}=f_{\frac{t-h}{2}},\quad w_1\cdot e_d=f_d, \quad w_1\cdot e_i=e_i,\quad w_1\cdot f_i=f_i \quad \text{for }i\neq \frac{t-h}{2},d.$$  
For $s_i\in W_{I_\Lambda}$, we have $i\ge \frac{t-h}{2}+1$. Since $w'$ fixes $e_i$ and $f_i$ for  $ \frac{t-h}{2}+1\le i \le d-1$, we clearly have $s_iw_1=w_1s_i$ for $1\le i\le d-2$. Similarly, a direct calculation shows that $t^\pm w_1=w_1t^\mp$. The claim is proved.

To prove (ii), note that $w_\Lambda'I_\Lambda w_\Lambda'^{-1}\cap I_\Lambda=\{s_{\frac{t-h}{2}+2},\cdots,t^+,t^-\}$. Moreover, $w_\Lambda'$ is clearly reduced of $\ell(w_\Lambda')=\frac{t-h}{2}$. Then according to Proposition \ref{prop: dim of DL var} and Lemma \ref{lem: length of WI}, we have 
\begin{align*}
   \dim  X_{\Lambda}(w_{\Lambda}')= \ell(w_\Lambda')+\ell(W_{I_\Lambda})-\ell(W_{ w_{\Lambda}'I_{\Lambda}w_{\Lambda}'^{-1}\cap I_{\Lambda}})=\ell(w_\Lambda')+h-2=\frac{t+h}{2}-2.
\end{align*}

For (iii), the corresponding $I$ is empty. So $\dim X_{\Lambda,\Lambda'}(w_{\Lambda,\Lambda'})=\ell(w_{\Lambda,\Lambda'})=\frac{t-t'}{2}$.
\end{proof}

\subsubsection{}
We compute the length of the Weyl element by the combinatorial method as in \cite{bjorner2005combinatorics}, which we briefly recall here.  From now on, we set $e_{-i}\coloneqq f_i$ for $1\le i\le d$,   $s_{d-1}\coloneqq t^+$ and $s_{d}\coloneqq t^-$ for convenience. We use the bijections
\begin{align}\label{eq: bijection 1}
    \{e_{1},\ldots,e_{d}\}&\to \{1,\ldots,d\},\\ \notag
    e_i&\mapsto d+1-i,
\end{align}
and 
\begin{align}\label{eq: bijection 2}
    \{e_{-d},\ldots,e_{-1} \}&\to \{-d,\ldots,-1 \},\\ \notag
    e_{-i}&\mapsto -d-1+i
\end{align}
to identify $\{e_{-d},\ldots,e_{-1},e_{1},\ldots,e_{d}\}$ with $\{-d,\ldots,-1,1,\ldots,d\}.$ We denote the latter set as $[\pm d]$.

Under these bijections, we may and will  identify $W$ as a subgroup of $S_{[\pm d]}$. When $\Omega$ is of odd dimension, we denote the corresponding subgroup of $S_{2d}$ as $S_{d}^B$. When $\Omega$ is of even dimension, we denote the corresponding subgroup of $S_{2d}$ as $S_{d}^D$.  More precisely, 
\begin{align*}
    \text{$\sigma \in S_{d}^B\subset S_{[\pm d]} $\quad if and only if  \quad   $\sigma(-i)=-\sigma(i)$ for all $i\in [\pm d]$,}
\end{align*}
and
\begin{align*}
    S_d^D \stackrel{\text { def }}{=}\left\{\sigma \in S_d^B: \operatorname{neg}(\sigma(1), \ldots, \sigma(d)) \equiv 0(\bmod 2)\right\}.
\end{align*}
Here $\operatorname{neg}(\sigma(1), \ldots, \sigma(d)) $ is the number of negative entries.

Let $\sigma\in S_{d}^B$. Since  $\sigma(-i)=-\sigma(i)$ for all $i\in [\pm d]$, to determine $\sigma$, it suffices to describe the action of $\sigma \in S_{[d]}^B$ on $\{1,\ldots,d\}$. As a result,  we can write 
\begin{align}\label{eq: window notation}
    \sigma=[\sigma(1),\ldots,\sigma(d)],
\end{align}
which is called \emph{window notation} in \cite{bjorner2005combinatorics}.
For example, $$s_{d-i}=[1, \ldots, i-1, i+1, i, i+2, \ldots, d]$$ 
for $1\le i\le d-1$, and 
$$s_{d}=[-2,-1,3,\ldots,d]$$
when $s_d\in S_d^D.$

\subsubsection{}
In order to determine the length of  $\sigma\in W$,  we define the  set of  inversions of $\sigma$.
We define the invariant function on $  S_{d}^B$ to be
\begin{align*} 
\operatorname{Inv}_B(\sigma)\coloneqq& \Inv_B^+(\sigma)\cup \Inv_B^-(\sigma) \\
    \coloneqq  & \{(i, j) \in \{1,\ldots,d\}^2: i<j,\, \sigma(i)>\sigma(j)\} \\
    & \cup \{(i, j) \in \{1,\ldots,d\}^2: i \leq j,\, \sigma(-i)>\sigma(j)\}.
\end{align*}
We define the invariant function on   $S_{d}^D$ to be
\begin{align*}
    \Inv_D(\sigma)\coloneqq &\Inv_D^+(\sigma)\cup \Inv_D^-(\sigma) \\
    \coloneqq &\left\{(i, j)\in \{1,\ldots,d\}^2: i<j,\,  \sigma(i)>\sigma(j)   \right\}\\
    &\cup \left\{(i, j)\in \{1,\ldots,d\}^2: i<j,\,  \sigma(-i)>\sigma(j)   \right\}.
\end{align*}
When the underlying group is clear, we will drop the subscript and simply write it as $\Inv(\sigma)$.
Let $\inv(\sigma)=|\Inv(\sigma)|$.
Then $\inv(s_i)=1$ for $1\le i\le d$. Moreover, one can directly check that for any $\sigma \in W$, we have $\inv(s_i\sigma)=\inv(\sigma)\pm 1$.

\begin{proposition}\label{prop: length=inv}
     Given $\sigma\in W$, we have $\ell(w)=\inv(w)$.
\end{proposition}
\begin{proof}
 This is \cite[Proposition 8.1.1]{bjorner2005combinatorics} for $S_{d}^B$   and \cite[Proposition 8.2.1]{bjorner2005combinatorics} for $S_{d}^D.$
\end{proof}
\subsubsection{} Now we are ready to compute $\ell(w_r)$ by computing $\inv(w_r)$.
\begin{lemma}
\label{lem: length of w}
   We have $\ell(w_r)=r+h-1$.  In particular, for a vertex lattice $\Lambda_t$ of type $t=2d$ or $2d+1$, $\ell(w_{\Lambda_t})=\ell(w_{d-\lfloor\frac{h}{2}\rfloor-1})=(t+h)/2-1$. 
\end{lemma}
\begin{proof}
First, assume $h$ and hence $t$ are even.     Note that $w_r\cdot e_{\frac{t-h}{2}-r+1}=f_{\frac{t-h}{2}}=e_{-d+h/2}$, $w_r\cdot e_i=e_{i-1}$ for $\frac{t-h}{2}-r+1<i\le \frac{t-h}{2}$,  $w_r\cdot e_d=f_d=e_{-d}$ and $w_r$ fixes all the other $e_i$. By the identification \eqref{eq: bijection 1} and \eqref{eq: bijection 2}, 
    \begin{align}\label{eq: window notation of w_r}
        w_r=[-1,2,\ldots,h/2, h/2+2,\ldots, h/2+r,-(h/2+1),h/2+r+1,\ldots,d]
    \end{align}
under the notation \eqref{eq: window notation}. Here $[h/2+1]\mapsto [h/2+2]$, $[h/2+r-1]\mapsto[h/2+r]$, $[h/2+r]\mapsto [-(h/2+1)]$.
Then     
\begin{align*}
    \Inv^+(w)=\{(i,h/2+r): i \in \{1,\ldots,h/2+r-1\}\}, 
\end{align*}
and 
\begin{align*}
  \Inv^-(w)=\{(i,h/2+r): i \in \{1,\ldots,h/2\}\}.
\end{align*}
Hence $\ell(w)=\inv(w)=(\frac{h}{2}+r-1)+(\frac{h}{2})=r+h-1.$

Now we assume $h$ and hence $t$ are odd. For $h>1$, we have $w\cdot e_{\frac{t-h}{2}-r+1}=f_{\frac{t-h}{2}}=e_{-d+\frac{t-h}{2}}$,  $w\cdot e_i=e_{i-1}$ for $\frac{t-h}{2}-r+1<i\le \frac{t-h}{2}/2$ and $w_r$ fixes all the other $e_i$. By the identification \eqref{eq: bijection 1} and \eqref{eq: bijection 2}, 
 \begin{align*}
        w=[1,\ldots,\frac{h-1}{2}, \frac{h-1}{2}+2,\ldots, \frac{h-1}{2}+r,-\frac{h+1}{2},\frac{h+1}{2}+r,\ldots,d].
    \end{align*}
Then     
\begin{align*}
    \Inv^+(w)=\{(i,(h-1)/2+r): i \in \{1,\ldots,(h-1)/2+r-1\}\}, 
\end{align*}
and 
\begin{align*}
  \Inv^-(w)=\{(i,\frac{h-1}{2}+r): i \in \{1,\ldots,\frac{h+1}{2}\}\cup \{\frac{h-1}{2}+r,\frac{h-1}{2}+r)) \}.
\end{align*}
Hence $\ell(w)=\inv(w)=(\frac{h-1}{2}+r-1)+(\frac{h-1}{2}+1)=r+h-1$.
\end{proof}

\begin{lemma}\label{lem: minimal rep}\hfill
\begin{enumerate}
    \item 
  $w_r$  and $w_r'$ are the minimal representative elements in the double cosets $W_{I_r}w_r W_{I_r}$ and $W_{I_r}w_r' W_{I_r}$
  \item $w_{r,s}$ is the  the minimal representative element in the double coset $W_{I_{r,s}}w_{r,s} W_{I_{r,s}}$
  \end{enumerate}
\end{lemma}
\begin{proof}
   First, the expressions of $w_r'$ and $w_{r,s}$ given in  \eqref{eq: word 2}, \eqref{eq: word 4}, and \eqref{eq: word 5} are clearly reduced. One can  check the expressions of  $w_r$ given in \eqref{eq: word 1}, and \eqref{eq: word 3},  are also reduced by comparing the length of $w_r$ given in Lemma \ref{lem: length of w} and the number of letters in the word expression of $w_r$.

 To show that an element $w$ is a minimal representative element in $W_IwW_I$, it suffices to show that $\ell(s_iw)>\ell(w)$ and $\ell(ws_i)>\ell(w)$ for any $i\in I$. This is clearly the case for $w_r'$ in $W_{I_r}w_r' W_{I_r}$ and $w_{r,s}$ in $W_{I_{r,s}}w_{r,s} W_{I_{r,s}}$. For $w_r$, one can show this by checking multiplying with $s_i$ for $i\in I_r$ will increase the number of inversions of \eqref{eq: window notation of w_r}.  
\end{proof}

\subsubsection{}
The following is a well-known fact. We sketch a proof for completeness.
\begin{lemma}\label{lem: length of WI}
Let $\Lambda$ be a vertex lattice of type $t\ge h$. Let $I_\Lambda=I_{\frac{t-h}{2}}$. We have    \begin{align*}
    \ell(w_{I_\Lambda})=\begin{cases}
       \frac{h}{2}(\frac{h}{2}-1)  & \text{ if $h$ is even},\\
        (\frac{(h-1)}{2})^2 &  \text{ if $h$ is odd}.
    \end{cases}
\end{align*}
\end{lemma}
\begin{proof}

First, we assume $h$ and $t$ are even.
Using Proposition \ref{prop: length=inv}, one can check that
   \begin{align*}
       w_I= 
           [(-1)^{h/2-1},-2,\ldots,-h/2,h/2+1,\ldots,d]
   \end{align*}
is the longest element in $W_I$ with $\ell(w_I)=\frac{h}{2}(\frac{h}{2}-1).$ More precisely, we have  
\begin{align*}
  \Inv(w_I)&=\Inv^+(w_I)\cup \Inv^-(w_I)
  \end{align*}
and 
\begin{align*}
   \Inv^+(w_I)=\Inv^-(w_I)=  \{(i,j) \in \{1,\ldots,d\}^2: i<j, j\le h/2\}.
\end{align*}

Now we assume $h$ and $t$ are odd.
Using Proposition \ref{prop: length=inv} again, one can check that
 \begin{align*}
       w_I=[-1,\ldots,-(h-1)/2,(h-1)/2+1,\ldots,d]
   \end{align*}
is the longest element in $W_I$ with $\ell(w_I)=(\frac{(h-1)}{2})^2.$  More precisely, we have  
\begin{align*}
   \Inv^+(w_I)&=   \{(i,j) \in \{1,\ldots,d\}^2: i<j, j\le (h-1)/2\},\\
   \Inv^-(w_I)&=   \{(i,j) \in \{1,\ldots,d\}^2: i\le j, j\le (h-1)/2\}.
\end{align*}\end{proof}

\section{Rapoport--Zink spaces for spinor similitudes}
\subsection{Integral models of Shimura varieties}

\subsubsection{}\label{6.1.1}
Let $V$ be a quadratic space over $\bbQ_p$ of dimension $n$ and  fix $L\subset V$ a vertex lattice of type $h$. Then $L$ corresponds to a point $\bfx'\in \calB(\SO(V),\bbQ_p).$ We set $G=\GSpin(V)$. We let $\bfx\in \calB(G,\bbQ_p)$ be a point lifting $\bfx'$ and let $\calG$ be the Bruhat--Tits stabilizer group scheme of $G$ corresponding to $\bfx$. We  obtain a local model triple $(G,\mu,\calG)$, where $\mu$ is the cocharacter defined in \S\ref{sssec: cocharcter mu}. We let $\Mloc_{\calG,\mu}$ denote the associated local model. By definition, $\Mloc_{\cG,\mu}$ is the unique, up to unique isomorphism, flat $\bbZ_p$-scheme with generic fiber $G/P_\mu$ and reduced special fiber which represents the $v$-sheaf $\rmM^v_{\calG,\mu}$ over $\mathrm{Spd}(\bbZ_p)$ defined in \cite{SW20}.  
By \cite{HPR}, this is isomorphic to the local model constructed by Pappas--Zhu \cite{PZ}.

\begin{lemma}\hfill
    \begin{enumerate}
    
        \item $\pi_1(G)=\bbZ$ and $I$ acts trivially on $\bbZ$.
        \item Let $E=\bbQ_{p^2}(\sqrt{p})$. Then $G$ splits over $E$.
    \end{enumerate}
\end{lemma}
\begin{proof}
    We have an exact sequence \[\xymatrix{1\ar[r] & \pi_1(G^{\der}) \ar[r] & \pi_1(G)\ar[r]& X_*(G^{\ab})\ar[r]& 1}.\]
    Now $G^{\der}$ is simply connected,  and $G^{\ab}\cong \bbG_m$. It follows that $\pi_1(G)\cong \bbZ$ with trivial $I$ action and hence we obtain (i).

We can choose a basis of $V$ such that the form $(\ ,\ ) $ is given by $\diag(u_1,\dotsc,u_n)$. The extension $E$ is the unique $\bbZ/2\bbZ\times \bbZ/2\bbZ$ extension of $\bbQ_p$ and hence contains $\sqrt{u_i}$ for each $i$. It follows that there is a basis of $V_E$ such that the form $(\ ,\ )_E$ is given by $\diag(1,\dotsc,1)$. Thus $\SO(V)$, and hence $G$ splits over $E$.
\end{proof}
Part (i) implies that  $\calG$  has connected special fiber, and hence is equal to the parahoric group scheme $\calG^\circ$ associated to $x$; see \cite{HainesRapoport}.

\subsubsection{}\label{6.1.2}

 We fix a choice of basis $\{e_1,\dotsc,e_n\}$ for $L$   so that the form $(\ ,\ )$ is given by $\diag(u_1,\dotsc,u_n)$ with $1= v_p(u_1)=\dotsc= v_p(u_h) $ and $0=v_p(u_{h+1})=\dotsc= v_p(u_n)$. Let $E=\bbQ_{p^2}(\sqrt{p})$ and  $\Gamma=\Gal(E/\bbQ_p)$. Then the image of $\bfx$ in $\calB(\SO(V),E)$ is a hyperspecial vertex (cf. the proofs of \cite[Prop. 2.7 and Prop. 2.8]{PRtame}), and hence corresponds to  a self-dual lattice $\widetilde{L}$ in $V_E$. In the basis $\{e_1,\dotsc,e_n\}$ above, $\widetilde{L}$ is given by $$\widetilde{L}=\mathrm{Span}_{O_E}\langle \sqrt{u_1}^{-1}e_1,\dotsc,\sqrt{u_n}^{-1}e_n\rangle.$$

\begin{proposition}\label{prop: very good embedding} \begin{enumerate}\hfill
\item $(G,\mu)\rightarrow (\GL(C(V_E),\mu_d))$ extends to a good integral local Hodge embedding $$\iota:(\calG,\mu)\rightarrow (\GL(C(\tL)),\mu_d);$$ see \cite[Definition 3.4.4]{KPZ}. Explicitly, we have closed immersions $$\calG\rightarrow \GL(C(\tL)), \ \ \ \ \Mloc_{\calG,\mu}\rightarrow \Gr(C(\tL),d)$$ which extends $G\rightarrow \GL(V)$ and the natural map $G/P_\mu\rightarrow \Gr(C(\tL),d)_E$ respectively.
\item The integral local Hodge embedding $\iota:(\calG,\mu)\rightarrow (\GL(C(\tL)),\mu_d)$ is very good at all points of $\Mloc_{\calG,\mu}(k)$ in the sense of \cite[Definition 5.2.5]{KPZ}.
\end{enumerate}
\end{proposition}
    \begin{proof} This  can be deduced from the proof \cite[Theorem 6.1.1]{KPZ}; we recall the argument for the convenience of the reader. Let $\calG'$ be hyperspecial parahoric of $G_E$ corresponding to the image of $x$ in $\calB(G,E)$, and set $\widetilde{\calG}=\Res_{\calO_E/\bbZ_p}\calG'$. 
 Then we obtain morphisms $$\calG\rightarrow\widetilde{\calG}\rightarrow \GL(C(\tL));$$ the first is a closed immersion by \cite[Proposition 2.4.10]{KZ}, and the second is a closed immersion by \cite[1.3]{PrY} and the fact that Weil restriction preserves closed immersions.

    Similarly, we have closed immersions $$ \Mloc_{\calG,\mu}\rightarrow \Mloc_{\widetilde{\calG},\widetilde{\mu}}\rightarrow \mathrm{Gr}(C(\tL)),d);$$ where $\widetilde{\mu}$ is the conjugacy class of cocharacters of $\widetilde{G}$ induced by $\mu$, and $\Mloc_{\widetilde{\calG},\widetilde{\mu}}$ is the local model associated to the local model triple $(\widetilde{G},\widetilde{\mu},\widetilde{\calG})$. The fact that these are closed immersions follows from \cite[Proposition 3.3.7]{KPZ} and \cite[Proposition 3.3.14]{KPZ}; we thus obtain (1).

    For (ii), note that $\calG$ can be cut out inside $\widetilde{G}\hookrightarrow \GL(C(\tL))$ by a set of endomorphisms. It follows that 
   $(\calG,\mu)\rightarrow (\GL(C(\tL)),\mu_d)$   is very good at all points of $\Mloc_{\calG,\mu}(k) $  by \cite[Corollary 5.3.4]{KPZ} and \cite[Proposition 5.3.10]{KPZ}; cf. \cite[Proof of Theorem 6.1.1]{KPZ}. \end{proof}

\subsubsection{} Now let $V_\Q$ be a rational quadratic space of signature $(n-2,2)$ with $n\geq 3$, and we let $Q$ denote the associated quadratic form. As in the local setting, we associate to $V_{\bbQ}$ the Clifford algebra $C(V_{\bbQ})$ which is endowed with a $\Z/2\Z$-grading $C(V_{\bbQ})=C(V_{\bbQ
})^+\oplus C(V_{\bbQ})^-$ and a canonical involution $c\mapsto c^*$. We let $\rmG=\GSpin(V_{\bbQ})$ be the group of spinor similitudes whose points valued in a $\bbQ$-algebra $R$ are given by $$\rmG(R)=\{g\in C^+(V_{R})^\times\mid gV_Rg^{-1}=V_R, g^*g\in R^\times\}.$$
We let $X$ be the space of oriented negative 2-planes in $V_\Q \otimes_{\Q} \R$. As explained in \cite[\S7.1.2]{HowardPappas} we can realize $X$ as a $\rmG(\bbR)$-conjugacy class of homomorphisms $h:\bbS\rightarrow \rmG_{\bbR}$, where $\bbS$ is the Deligne torus $\Res_{\bbC/\bbR}\bbG_m$, and  the pair $(\rmG,X)$  forms a Shimura datum with reflex field  $\Q$.     We let $\mu=\mu_h^{-1}$ denote the inverse of the Hodge cocharacter  $\mu_h$ associated to $h$. Then the conjugacy class of $\mu$ coincides with that of the cocharacter constructed in \S\ref{sssec: cocharcter mu}.

Let $\mathrm{F}$ be a totally real field with degree $r=[E,\Q_p]$   and $\mathfrak{p}$ be a place of $\rmF$ above $p$ such that $\rmF_\mathfrak{p}=E$ where $E$ is the local field  we chose in \S \ref{6.1.1}.   Let $V_\rmF$ be the quadratic space over $\rmF$ induced from $V_\Q$. Let $\widetilde{\rmG}=\mathrm{Res}_{\rmF/\Q}\mathrm{GSpin}(V_\rmF)$. We have a natural embedding $\iota_\rmG: \rmG \hookrightarrow \widetilde{\rmG}$.

Fix $L\subset V:=V_{\bbQ_p}$ a vertex lattice of type $h$ and let $\calG$ be the corresponding parahoric of $G$; we set $\rmK_p=\calG(\bbZ_p)$. We let $\tL\subset V_E$ be the self-dual $\calO_E$-lattice constructed in \S\ref{6.1.2} and $C(\tL)\subset C(V_E)$ the associated integral Clifford algebra. The canonical involution $c\mapsto c^*$ preserves the integral structure $C(\widetilde{L})$, and hence we obtain a decomposition $$C(\tL)\cong C(\tL)^+\oplus C(\tL)^-.$$

\subsubsection{}\label{sec: alternating form}Fix an element $c\in C(\tL)^\times\cap C(V_\rmF)$ with $c=-c^*$; for example we may choose $e,f$ elements of an orthogonal  $\calO_{\rmF_{(\fkp)}}$-basis for $\tL\cap V_{\rmF}$ with $Q(e),Q(f)\in \calO_{K_{(\fkp)}}$ and set $c=ef$. We define a $\bbQ$-linear alternating form $$\psi: C(V_\rmF)\times C(V_\rmF)\rightarrow \bbQ$$ by the formula $$\psi(x,y)=\tr_{\rmF/\Q}(\delta_{\rmF/\Q}^{-1}\mathrm{Trd}(xc y^*)),$$ where $\mathrm{Trd}:V_\rmF\rightarrow \rmF$ is the reduced trace and $\delta_{\rmF/\Q}$ is the different of the extension $\rmF/\Q$. Then $C(\tL)$ is a self-dual lattice in $C(V_E)$ for the base change of this form to $\bbQ_p$. We then obtain a Hodge embedding $$(\rmG,X)\rightarrow (\mathrm{GSp}(C(V_\rmF),\psi),S^\pm)$$ where $S^\pm$ is the Siegel double space. 
\subsubsection{}

 Let $\bbA_f^p$ be the ring of prime-to-$p$ adeles and fix $\rmK^p\subset \rmG(\bbA_f^p)$ a compact open subgroup. Setting $\rmK=\rmK_p\rmK^p\subset \rmG(\bbA_f)$, we obtain the associated Shimura variety $\Sh_{\rmK}(\rmG,X)$ over $\bbQ$ whose complex points are given by $$\Sh_{\rmK}(\rmG,X)(\bbC)=\rmG(\bbQ)\backslash X\times \rmG(\bbA_f)/\rmK.$$

 Let $\rmK'_p\subset \mathrm{GSp}(C(V_{\rmF}))(\bbQ_p)$   be the  stabilizer of $C(\tL)$. Then for $\rmK'^p$ a sufficiently small compact open subgroup of $\mathrm{GSp}(C(V_\rmF))(\bbA_f^p)$ and setting $\rmK'=\rmK'_p\rmK'^p$, we obtain a closed immersion of Shimura varieties $$\Sh_{\rmK}(\rmG,X)\rightarrow \Sh_{\rmK'}(\mathrm{GSp}(C(V_\rmF)),S^\pm).$$

The $\bbZ_p$-lattice $C(\tL)\subset C(V_{E})$ gives rise to a moduli interpretation of the Shimura variety $\Sh_{\rmK'}(\mathrm{GSp}(C(V_\rmF)),S^\pm)$ as a moduli space of  polarized abelian varieties up to prime-to-$p$ isogeny, and hence to an integral model $\scrS_{\rmK'}(\mathrm{GSp}(C(V_\rmF)),S^\pm)$ over $\bbZ_{(p)}$; see \cite[\S 4]{KisinPappas}.  Note that since $C(\tL)$ is self-dual, the polarization has degree coprime to $p$. We define $\scrS_{\rmK}(\rmG,X)$ to be the normalization of the  Zariski closure of $\Sh_{\rmK}(\rmG,X)$ inside $\scrS_{\rmK'}(\mathrm{GSp}(C(V_\rmF)),S^\pm)$ (by the main result of \cite{Xu}, the normalization is not necessary). For notational simplicity we will simply denote $ \mathscr{S}_{\mathrm{K}}(\rmG, \rmX)$ and $ \mathscr{S}_{\mathrm{K}^{\prime}}(\mathrm{GSp}(C(V_\rmF)), S^{ \pm})$ as $ \mathscr{S}_{\rmK}$ and $ \mathscr{S}_{\rmK'}$.

\subsubsection{}\label{para: tensors}For a module $M$ over a ring $R$, we let $M^\otimes$ denote the direct sum of all $R$-modules obtained from $M$ by taking duals, tensor products, symmetric and exterior products.
We set $C(\tL)_{\bbZ_{(p)}}:=C(\tL)\cap C(V_{\rmF})$ which we regard as  a $\bbZ_{(p)}$-module, and we fix a set of tensors $t_\alpha\in C(\tL)_{\bbZ_{(p)}}^\otimes$ whose stabilizer is $\calG_{\bbZ_{(p)}}$; here $\calG_{\bbZ_{(p)}}$ is the Zariski closure of $\rmG$ in $\GL(C(\tL)_{\bbZ_{(p)}})$. We let $\calA\rightarrow\scrS_{\rmK}$ be the pullback of the universal abelian variety on $\scrS_{\rmK'}$.

 Let $x_0\in \scrS_{\rmK}(k)$.  We let $\calA_{x_0}$ be the corresponding abelian variety over $k$ and $\scrG_{x_0}$ its $p$-divisible group. Then as explained in \cite[\S6.6]{Zhoumod}, we obtain Frobenius-invariant tensors $t_{\alpha,0,{x_0}}\in\bbD^\otimes$, where $\bbD:=\bbD(\scrG_{x_0})$ is the contravariant Dieudonn\'e module associated to $\scrG_{x_0}$. The $t_{\alpha,0,{x_0}}$ are obtained via the $p$-adic comparison isomorphism from corresponding tensors in $t_{\alpha,p,\tilde{x}}\in T_p\calA_{\tilde{x}}^{\vee,\otimes}$, where $\tilde{x}\in \scrS_{\rmK}(K)$ is a lift of ${x_0}$ for some $K/\brQ_p$ finite; the $t_{\alpha,0,{x_0}}$ are independent of the choice of $\tilde{x}$. Moreover, we have a $\brZ_p$-linear isomorphism 
$$D_{\bbZ_{(p)}}\otimes_{\bbZ_{(p)}}\brZ_p\cong \bbD$$
which takes $t_{\alpha}$ to $t_{\alpha,0,x_0}$; here $D_{\bbZ_{(p)}}=\Hom_{\Z_{(p)}}(C(\tL)_{\Z_{(p)}},\Z_{(p)})$. Under this identification, the Frobenius on $\bbD$ is given by $\varphi= b_{x_0}\sigma$
 for some $b_{x_0}\in G(\brQ_p)$ which is $\mu$-admissible (i.e. $b_{x_0}\in \calG(\brZ_p)\sigma(\dot{w})\calG(\brZ_p)$ for some $w\in \Adm(\mu)_{\calG}$). We obtain a map $$\delta:\scrS_{\rmK}(k)\rightarrow B(G,\mu)$$ sending $x$ to $[b_x]\in B(G,\mu)$ which induces the Newton stratification on $\scrS_{\rmK}(\rmG,X)$. For $[b]\in B(G,\mu)$, $\delta^{-1}([b])$ is the set of $k$-points of a locally closed subscheme of $\scrS_{\rmK}\otimes k$ which we denote by $\scrS^{[b]}_{\rmK}$. By \cite[Theorem 1]{KMPS}, if $[b]$ is the unique basic element of $B(G,\mu)$, then $\scrS^{[b]}_{\rmK}$ is non-empty and closed.

\subsection{Rapoport--Zink spaces for spinor similitudes}
\subsubsection{}\label{para: RZ space GSp}
 
We keep the notations of the previous subsection so that $V_{\bbQ}$ is a rational quadratic space of signature $(n-2,2)$,  and $L\subset V:=V_{\bbQ_p}$ a vertex lattice with corresponding parahoric $\calG$. We have a Shimura datum $(\rmG,X)$ associated to the group $\rmG=\GSpin(V_{\bbQ})$ and we let $\mu=\mu_h^{-1}$ denote the inverse of the Hodge cocharacter. We  obtain a local Shimura pair $(G,\mu)$ with $\mu$ of the form given  in \S\ref{Subsec: GSpin Shimura datum}, and we let $b\in G(\brQ_p)$ be a representative of the unique basic conjugacy class in $B(G,\mu).$ Let $x_0\in \scrS^{[b]}_{\rmK}(k)$ and $b_{x_0}\in G(\brQ_p)$ the element giving the Frobenius on $\bbD(\scrG_{x_0})$ as in \S\ref{para: tensors} so that $b_{x_0}\in [b].$ Upon replacing $b$ by a $\sigma$-conjugate, we may assume $b=b_{x_0}.$

We recall the construction of Rapoport-Zink spaces associated with the local Shimura datum  $(G,b,\mu)$ in \cite{HK19}, \cite{PRshtukas}. First, we define a symplectic Rapoport-Zink space associated with $\mathscr{S}_{\mathrm{K}^{\prime}} (\mathrm{GSp}(C(V_{\rmF})), S^{ \pm} )$.  We set $\bbX=\scrG_{x_0}$. Then we have an isomorphism $D_{\brZ_p}\cong \bbD(\bbX)$ taking $t_\alpha$ to $t_{\bbX,\alpha}:=t_{\alpha,0,x_0}$ such that the Frobenius on $\bbD(\bbX)$ is given by $b\sigma$, and $\bbX$ is equipped with a principal polarization $\lambda_{\bbX}:\bbX\rightarrow \bbX^\vee$.

We let $\Nilp_{\brZ_p}$ denote the category of $\brZ_p$-schemes $S$ such that  $p$ is Zariski locally nilpotent in $\calO_S$. The symplectic Rapoport-Zink space $\cN_{C(\tL)}$ is the following functor: for each $S$ in $\operatorname{Nilp}_{\brZ_p}$, $\cN_{C(\widetilde{L})}(S)$ is the set of isomorphism classes of triples $(X, \lambda, \rho)$ where
\begin{enumerate}
    \item $X$ is a $p$-divisible group over $S$,
    \item $\lambda: X \rightarrow X^{\vee}$ is a principal polarization,
    \item $\rho: \bX \times_k S/pS \rightarrow X \times_S S/pS$ is a quasi-isogeny such that Zariski locally on $S/pS$, we have
$$
\rho^{\vee} \circ \lambda \circ \rho=c(\rho) \cdot \lambda_{\bX},
$$
for some $c(\rho) \in \mathbb{Q}_p^{\times}$.
\end{enumerate}

\subsubsection{}Let $\widehat{\scrS_{\rmK}}$, $\widehat{\scrS_{\rmK'}}$ denote the $p$-adic completion of $\scrS_{\rmK}\otimes_{\bbZ_p}\brZ_p$, $\scrS_{\rmK'}\otimes_{\bbZ_p}\brZ_p$ respectively. As in \cite[Theorem 6.21]{RZ}, there is a uniformization morphism
$$
\Theta: \cN_{C(\tL)}\rightarrow \widehat{\mathscr{S}_{\rmK'}}.
$$
 Then we define $\cN_L^\diamond$ to be the fiber product   
 \begin{equation}\label{eq: diagram}
 \begin{tikzcd}
\cN_L^\diamond \arrow{r}{} \arrow[swap]{d}{\Theta^\diamond} & \cN_{C(\tL)}\arrow{d}{\Theta} \\
\widehat{\mathscr{S}}_{\rmK} \arrow{r}{\iota}&  \widehat{\mathscr{S}_{\rmK'}}.
\end{tikzcd} 
\end{equation} 
The Dieudonn\'e isocrystal of the universal $p$-divisible group over $(\calN_{L}^\diamond\otimes_{\brZ_p}k)^{\mathrm{perf}}$ is equipped with two families of Frobenius-invariant tensors. One of them is the family of tensors induced from $(t_{\bX,\alpha})$ via the universal quasi-isogeny $\rho$, which we denote as $t_\alpha^\diamond$. Another  family of tensors is the one induced by pullback along $\Theta^\diamond$, which we denote as $u_{\alpha, y}^{\diamond}:=\Theta^{\diamond}{ }^*\left(t_{\alpha, 0,\Theta^{\diamond}(y)}\right)$ for any geometric point $y \in \cN_L^\diamond(k)$.

According to \cite[Lemma 5.8]{HK19}, the condition $t_{\alpha, y}^{\diamond}=u_{\alpha, y}^{\diamond}$ holds over a union of connected components of $\cN_L^\diamond$.  
\begin{definition}
    We define $\cN_L$ as the union of connected components of $\cN_L^\diamond$ where  $t_{\alpha, y}^{\diamond}=u_{\alpha, y}^{\diamond}$.
\end{definition}

The following proposition is \cite[Proposition 5.10]{HK19}. Note that Axiom A of \emph{loc. cit.} is satisfied by \cite[Proposition 7.8]{Zhoumod} and \cite[Proposition 6.5]{Zhoumod} since $b$ is basic.
\begin{proposition}\cite[Proposition 5.10] {HK19}\label{prop: points of N and ADLV}
    We have a natural bijection
    \begin{align*}
        \varphi: \cN_{L}(k) \xrightarrow{\sim} X(\mu^\sigma,b)_\K(k).
    \end{align*}
\end{proposition}\qed

\subsubsection{}Let $y\in \calN_L(k)$ with corresponding triple $(X_y,\lambda_y,\rho_y)$. Via the framing quasi-isogeny $\rho_y$,  $y$ determines a $\brZ_p$-lattice
$$
M_y:=\mathbb{D}(X_y)(\brZ_p) \subset D_{\brQ_p}
$$
and a $\brZ_p$-submodule $M_{y,1}=VM_y=F^{-1}(p M_y)$. Then $M_y=gD_{\brZ_p}$ for $g=\varphi(y)\in X(\mu^\sigma,b)_\K$ (recall we have fixed a tensor preserving isomorphism $D_{\brZ_p}\cong \bbD(\bbX)$). For $S$ an object of $\Nilp_{\brZ_p}$ and $X$ a $p$-divisible group over $S$,  we let $\bbD(X)$ denote the contravariant Dieudonn\'e crystal associated to $X$; see \cite{Messing}, \cite{DeJong} for the construction and properties of this functor. Let $\widehat{\calN}_{L,y}$ denote the formal neighbourhood of $y\in \calN_L(k)$ and $\widehat{\scrX}$ the universal $p$-divisible group on $\widehat{\calN}_{L,y}$. We recall the description of $\widehat{\calN}_L$ and $\widehat{\scrX}$ given in \cite{KisinPappas}, \cite{KPZ}. 

Upon fixing a representative of $g$ in $G(\brQ_p)$, multiplication by $g$ induces a tensor preserving identification $D_{\brZ_p}\cong M_y$, and hence allows us to identify $\Mloc_{\calG,\mu,\brZ_p}$ as a subscheme of $\Gr(M_y,d)$. The filtration on $\overline{M}_{y}:=M_y/pM_y$ is then given by the image of $M_{y,1}$ and corresponds to a point $x\in \Mloc_{\calG,\mu}(k)$.

For any complete local ring $S$ with maximal ideal $\fkn$ and residue field $k$, let $ \widehat{W}(S)\subset W(S)$ denote Zink's subring $\widehat{W}(S)=W(k)\oplus \bbW(\fkn)$, where $\bbW(\fkn)\subset W(S)$ consists of Witt vectors $(w_i)_{i\geq 1}$ such that $w_i$ tends to $0$ $\fkn$-adically.  Let $R$ be the complete local ring of $\Mloc_{\calG,\mu,\brZ_p}$ at $x$ and $\fkm$ its maximal ideal. We set $N:=M_y\otimes_{\brZ_p}\widehat{W}(R)$ and we let $N_1\subset N$ denote the preimage of the universal filtration on $M_y\otimes_{\brZ_p}R$ coming from the embedding $\Mloc_{\calG,\mu,\brZ_p}\rightarrow \Gr(M_y,d)$. The following proposition summarizes the main results from \cite[\S3,4]{KisinPappas}  that we will need in the next section.

\begin{proposition}\begin{enumerate}\label{prop: KP deformation theory}
\hfill
		\item There is an identification \begin{equation}
\widehat{\calN}_{L,y}\cong \Spf R
		\end{equation}such that if we consider $\widehat{\scrX}$ as a $p$-divisible group over $\Spf R$, there is a $\widehat{W}(R)$-linear bijection $$\gamma:\bbD(\widehat{\scrX})(\widehat{W}(R))\cong N$$ which identifies $\Fil^1\bbD(\widehat{\scrX})(\widehat{W}(R))$ with $N_1$. 
		
		\item Let $\fka=(\fkm^2,p)$. Then the identification $$\bbD(\widehat{\scrX})(\widehat{W}(R/\fka))\cong M_y\otimes_{\brZ_p}\widehat{W}(R/\fka)$$
		induced by $\gamma$ is equal to the canonical identification arising from the crystal structure on $\bbD(\widehat{\scrX})$, where we consider $R/\fka\rightarrow k$ as being equipped with the trivial divided power structure.
	\end{enumerate}
\end{proposition}
\begin{proof}
	Part (i) follows from \cite[\S3,4]{KisinPappas}  (cf. also \cite{KPZ}) and the global construction of the RZ-space $\calN_L$. More precisely, the closed immersion of formal schemes $\calN_L\rightarrow \widehat{\scrS}_{\rmK}$ induces an identification of formal neighbourhoods $\widehat{\calN}_{L,y'}\cong \widehat{\scrS}_{\rmK,y'}$ where $y'\in \widehat{\scrS}_{\rmK}(k)$ is the image of $y$, and $\widehat{\scrX}$ corresponds to the pullback of the universal $p$-divisible group $\widehat{\scrX}'$ on $\widehat{\scrS}_{\rmK,y'}$. 
	
	In \cite[\S3]{KisinPappas},  it is shown that we can upgrade the pair $(N,N_1)$ to a versal display over $\widehat{W}(R)$, and hence gives rise to a $p$-divisible group over $R$, such that the induced closed immersion to the universal deformation space of $\scrG_{y'}$ identifies $\Spf R$ with $ \widehat{\scrS}_{\rmK,y'}$.  This proves (1). Note that the constructions in \cite{KisinPappas}  can be applied  by Proposition \ref{prop: very good embedding} (see \cite{KPZ}).
	
	For part (ii), let $\beta:\bbD(\widehat{\scrX})(\widehat{W}(R/\fka))\cong M_y\otimes_{\brZ_p}\widehat{W}(R/\fka)$ denote the identification arising from the crystal structure. The composition $$\gamma \circ\beta^{-1}:M_y\otimes _{\brZ_p}\widehat{W}(R/\fka)\rightarrow M_y\otimes _{\brZ_p}\widehat{W}(R/\fka)$$   lifts  the identity on $M_y$. We have two different display structures on these modules; on the left, we have the trivial display structure coming from base change $k\rightarrow R/\fka$, and on the right, the display structure arises from the base change $R\rightarrow R/\fka$. We write $\Phi_{1}, \Phi_{1,0}$ for the corresponding `inverse Verschiebung' morphisms for these display structures. Let $\hat{M}_{y,1}\subset M_y\otimes_{\brZ_p}\widehat{W}(R/\fka)$ denote the preimage of $M_{y,1}$ under the projection $M_y\otimes_{\brZ_p}\widehat{W}(R/\fka)\rightarrow M_y$. 
  Then $\Phi_1$ and $\Phi_{1,0}$ extend to $\varphi$-semilinear morphisms $$\hat{\Phi}_1,\hat{\Phi}_{1,0}:\hat{M}_{y,1} \rightarrow M_y\otimes_{\brZ_p}\widehat{W}(R/\fka),$$ and we have $\hat{\Phi}_1=\hat{\Phi}_{1,0}$ by \cite[Lemma 3.1.12]{KisinPappas}. It follows from the uniqueness part of \cite[Theorem 3]{Zink} that $\gamma\circ\beta^{-1}$ is the identity.
\end{proof}
\subsubsection{} 
We conclude this section by recalling the Rapoport--Zink uniformization of the basic locus. Let $V_{\bbQ}'$ denote the unique positive definite quadratic space over $\bbQ$
with $$\epsilon(V'_{\bbQ_\ell})=\begin{cases}
    \epsilon(V_{\bbQ_\ell}) & \text{if $\ell\neq p$.}\\
  -  \epsilon(V_{\bbQ_\ell}) & \text{if $\ell= p$.}
\end{cases}$$
We set $\rmI:=\Gspin(V_{\bbQ}')$. Recall we have fixed $b\in G(\brQ_p)$ corresponding to the basic element of $B(G,\mu)$ and $\scrS_{\rmK}^{[b]}$ denotes the basic locus of $\scrS_{\rmK}\otimes k$. We let $(\widehat{\scrS_{\rmK}})_{\scrS_{\rmK}^{[b]}}$ denote the formal completion of $\scrS_{\rmK}$ along the closed subscheme $\scrS^{[b]}_{\rmK}.$
\begin{proposition}\label{prop: RZ uniformization}
    There exists an isomorphism of formal $\brZ_p$-schemes:
    $$\rmI(\bbQ)\backslash \calN_L\times \rmG(\bbA_f^p)/\rmK^p\xrightarrow{\sim}(\widehat{\scrS_{\rmK}})_{\scrS_{\rmK}^{[b]}}$$
\end{proposition}
\begin{proof}This follows from the argument of \cite[Theorem 3.3.2]{HowardPappas} and \cite[Theorem 7.2.4]{HowardPappas}; cf. also \cite[Theorem 6.1]{OkiNotes}, \cite[\S5]{HZZ}.
\end{proof}

\section{The Bruhat--Tits stratification}\label{sec: bruhat tits stratification}
\subsection{The $k$-points of $\cN_L$}
 \subsubsection{}
We keep the notation of the previous subsection so that $V_{\bbQ}$ is a quadratic space over $\bbQ$ of signature $(n-2,2)$ and $L\subset V= V_{\bbQ_p}$ is a vertex lattice of type $h$. Then we have the Rapoport-Zink space $\calN_L$ associated to $(G,b,\mu)$ and $L$ where $b=b_{x_0}$ corresponds to the Frobenius on the framing object $\bbX=\scrG_{x_0}.$

Recall that $\tL$ is a self-dual lattice in $V_E$ corresponding to the image of $\bfx$ in $\calB(G,E)$, which we regard  as a $\Z_p$-lattice of rank $4n$. We also let $\widetilde{V}=\tL\otimes_{\Z_p}\breve{\Q}_p$.
Then we have the corresponding Clifford algebra $C(\tL)$ and  GSpin group $\tilde{G}=\mathrm{Res}_{E/\Q_p}\mathrm{GSpin}(V_E)$. Note that we have a natural embedding  $\iota_G: G\hookrightarrow \tilde{G}$ which induces a basic element $\tilde{b}=\iota(b)$ of $\tilde{G}$.  It also induces a Hodge cocharacter $\tilde{\mu}=\iota\circ \mu$ of $\tilde{G}$.

\subsubsection{}	
First, we consider a decomposition of $\cN_L$ following \cite{HowardPappas}. Since we construct $\cN_L$ as a closed formal subscheme of the Siegel RZ-space, the universal $p$-divisible group $\scrX$ has a polarization $\lambda: \scrX \to \scrX^\vee$. The universal quasi-isogeny preserves the polarizations $\lambda$ and $\lambda_{\bbX}$ up to scaling. More precisely, Zariski locally on $\cN_L$, we have $\rho^{\vee} \circ \lambda \circ \rho=c(\rho) \cdot \lambda_{\bX}$  for   some $c(\rho) \in \mathbb{Q}_p^{\times}$ (see \S \ref{para: RZ space GSp}). For each $\ell \in \mathbb{Z}$, let $\cN_L^{(\ell)} \subset \cN_L$ be the open and closed formal subscheme on which $\operatorname{ord}_p(c(\rho))=\ell$ so that
\begin{align}\label{eq: decomposition}
\cN_L=\bigsqcup_{\ell \in \mathbb{Z}} \cN_L^{(\ell)}.
\end{align}
We let $J_b$  denote the $\sigma$-centralizer group scheme over $\mathbb{Q}_p$, i.e.  for any $\mathbb{Q}_p$-algebra $R$, we have
$$
J_b(R)=\left\{g \in G\left(R \otimes_{\mathbb{Q}_p} \brQ_p\right): g^{-1} b \sigma(g)=b\right\}.
$$
Then  $J_b\cong \GSpin(\bbV)$, and we have an acion of
$J_b\left(\mathbb{Q}_p\right)$ on $\cN_L$ (see \cite[\S 2.3.7]{HowardPappas}). 
Moreover, each $g\in J_b(\Q_p)$ induces an isomorphism 
$$g: \cN_L^{(\ell)} \rightarrow \cN_L^{\left(\ell+\operatorname{val}_p\left(\eta_b(g)\right)\right)},$$
where $\eta_b$ is the spinor similitude $\eta_b: J_b\left(\mathbb{Q}_p\right) \rightarrow \mathbb{Q}_p^{\times}$.
In particular, $p^\Z$ acts on $\cN_L$ and 
\begin{align*}
    \p^\Z \backslash \cN_L \cong \cN_L^{(0)} \sqcup  \cN_L^{(1)}.
\end{align*}
Since $\eta_b$ is surjective, we know $\cN_L^{(\ell)}$ are isomorphic for each $\ell$.

\subsubsection{}\label{para: associated special lattice}
Recall by Proposition \ref{prop: points of N and ADLV}, there is a natural bijection 
\begin{align}\label{eq: N and X(b,mu)}
    \varphi:\calN_L(k)\xrightarrow{\sim} X(\mu^\sigma,b)_\K(k),
\end{align}
 and we have 
 $\brZ_p$-submodules 
 $$M_y=gD_{\brZ_p}\subset D_{\brQ_p},\ \ \ M_{y,1}=F^{-1}(pM_y)=p \sigma^{-1}(b^{-1}g)D_{\brZ_p}\subset D_{\brQ_p}$$
where $g=\varphi(y).$  
Using the inclusion $L_{\brQ_p} \subset \operatorname{End}_{\brQ_p}(D_{\brQ_p})$, we define $\brZ_p$-lattices
$$
\begin{aligned}
L_y & =\{x \in V_{\brQ_p}: x M_{y,1} \subset M_{y,1}\} \\
L_y' & =\{x \in V_{\brQ_p}: x M_y \subset M_y\}.
\end{aligned}
$$
Note that the action of $p^{\mathbb{Z}}$ on $\cN_{L}(k)$ rescales the lattices $M_y$ and $M_{y,1}$, and therefore the three lattices defined above depend only on the image of $y$ in $p^{\mathbb{Z}} \backslash \cN_{L}(k)$.

\begin{lemma}\label{lem: V=B}  
     \begin{align*}
       L_{\breve{\Z}_p}&=\{x\in V_{\brQ_p}\mid x  C(\tL_{\brZ_p})\subset C(\tL_{\brZ_p})\},\\
       L_{\breve{\Z}_p}^\vee&=\{x\in V_{\brQ_p}\mid  \widetilde{\pi} x  C(\tL_{\brZ_p})\subset C(\tL_{\brZ_p})\}.
     \end{align*} 
Equivalently, we have
 \begin{align}
       L_{\breve{\Z}_p}&=\{x\in V_{\brQ_p}\mid x D(\tL_{\brZ_p})\subset D(\tL_{\brZ_p})\},\label{eq: L}\\
       L_{\breve{\Z}_p}^\vee&=\{x\in V_{\brQ_p}\mid \widetilde{\pi} x  D(\tL_{\brZ_p})\subset D(\tL_{\brZ_p})\}. \label{eq: L^vee}
     \end{align}
 \end{lemma}
 \begin{proof}
We fix a choice of basis $\{e_1,\dotsc,e_n\}$ for $L_{\brZ_p}$   so that the symmetric bilinear form $(\ ,\ )$ is given by $\diag(u_1,\dotsc,u_n)$ with $1= v_p(u_1)=\dotsc= v_p(u_h) $ and $0=v_p(u_{h+1})=\dotsc= v_p(u_n)$. Then $\{ \tilde{e}_1,\ldots,\tilde{e}_n\}=\{\widetilde{\pi}^{-1}e_1,\ldots,\widetilde{\pi}^{-1}e_h,e_{h+1},\ldots,e_n\}$ is an $\calO_E\otimes_{\bbZ_p}\brZ_p$-basis of $\tL_{\brZ_p}$.

We clearly have $L_{\brZ_p} \subset\{x\in V_{\brQ_p}\mid x  C(\tL_{\brZ_p})\subset C(\tL_{\brZ_p})\}$. Now we assume $x\in V_{\brZ_p}$ such that $x C(\tL_{\brZ_p})\subset C(\tL_{\brZ_p})$. In particular, we have $x \cdot 1=x \in C(\tL_{\brZ_p})$. Write $x=\sum_{\ell=1}^n a_\ell e_\ell$.  Since $e_1  ,\ldots,e_n  $ are linearly independent, we have $x  \in C(\tL_{\brZ_p})$ if and only if $a_\ell \in \brZ_p$. Therefore, $x\in L_{\brZ_p}$.

To prove the second claim, we choose $\{p^{-1}e_1,\ldots,p^{-1}e_h,e_{h+1},\ldots,e_n\}$ as a basis of $L_{\brZ_p}^\vee$. Then $\widetilde{\pi}L_{\brZ_p}^\vee=\Span_{\brZ_p}\{ \widetilde{\pi}^{-1}e_1,\ldots,\widetilde{\pi}^{-1}e_h,\widetilde{\pi} e_{h+1},\ldots,\widetilde{\pi} e_n\}\subset \tL_{\brZ_p}$. So $ L_{\breve{\Z}_p}^\vee\subset \{x\in V_{\brQ_p}\mid  \widetilde{\pi} x  C(\tL_{\brZ_p})\subset C(\tL_{\brZ_p})\}$. Now we assume $x=\sum_{\ell=1}^n a_\ell e_\ell $ and $ \widetilde{\pi}  x  C(\tL_{\brZ_p})\subset C(\tL_{\brZ_p})$. In particular, $\widetilde{\pi}x\cdot 1=\widetilde{\pi}x \in C(\tL_{\brZ_p})$. Then $\widetilde{\pi} a_\ell \in \widetilde{\pi}^{-1}\brZ_p$ for $\ell \le h$ and  $\widetilde{\pi} a_\ell \in \brZ_p$ for $\ell > h$. Hence, we have $a_\ell \in p^{-1}\brZ_p$ for $\ell \le h$ and $a_\ell \in \brZ_p$ for $\ell >h$. Therefore, $x\in L_{\brZ_p}^\vee$.
 \end{proof}

\subsubsection{}

\begin{proposition}\label{prop: k points of N}
    For every $y \in \cN_{L}(k)$ the lattice $L_y$ is special, and satisfies $\Phi(L_y)=L_y'$. Moreover, $y \mapsto L_y$ establishes the following bijection
\begin{align*}
\begin{cases}
  \cN_L^{(0)}(k)\sqcup \cN_L^{(1)}(k)   \stackrel{\sim}{\rightarrow} \{\text{special lattices } M \subset V_{\brQ_p}: M\stackrel{h}{\subset}M^\vee\} & \text{ if $h=0$ or $n$},\\
    \cN_L^{(0)}(k) \stackrel{\sim}{\rightarrow} \{\text{special lattices } M \subset V_{\brQ_p}: M\stackrel{h}{\subset}M^\vee\} & \text{ otherwise.}
\end{cases}
\end{align*}
\end{proposition}
\begin{proof}
    The case $h=0$ or $n$ is covered in \cite{HowardPappas} so we assume $h\neq 0$ and $n$. 
  Let $y \in \cN_{L}(k)$.
Recall that we have a bijection $\varphi$ as in \eqref{eq: N and X(b,mu)} that sends a point $y \in \cN_L(k)$ to the unique $\varphi(y)\in X(\mu^\sigma,b)
)_K(k)$ such that $M_y=g\cdot D_{\brZ_p}$. Moreover, $VD_{\brZ_p}=F^{-1}p g\cdot D_{\brZ_p}=p \sigma^{-1}(b^{-1}g)\cdot D_{\brZ_p}.$ Now  $p ^{\mathbb{Z}} \backslash X(\mu^\sigma,b)_\K(k) $ is identified with two copies of special lattices in $V_{\brQ_p}$ by Proposition \ref{prop: ADLV and special lattice}. 

We claim that $\cN_L^{(0)}$ is identified with one of the copies under this identification. Write the stabilizer of $L$ in $G(\brQ_p)$ as $K\sqcup \delta K$ as in the proof of Proposition \ref{prop: ADLV and special lattice}. Then it suffices to find an element $j$ in $J_b(\Q_p)\cap \delta p^\Z K$ since $j \backslash (\cN_L^{(0)}\sqcup \cN_L^{(1)})$ will be identified with the set of special lattices and $\delta$ also induces an isomorphism between 
$\cN_L^{(0)}$ and $\cN_L^{(1)}$.  To find such $j$, note that we have a surjection of $J_b(\Q_p)$ to $\pi_1(\mathrm{SO}(V))_I\cong \Z/2\Z$, which we regard as the stabilizer of the base alcove of $\mathrm{SO}(V)$. Then an inverse image of the nontrivial element of $\Z/2\Z$ will work.

\end{proof}

\subsection{Special cycles on $\cN_L$}\label{sec: special cycles}

\subsubsection{}\label{para: Z and Y cycles} 

Recall that $\bV=V_{\brQ_p}^{\Phi}\subset \widetilde{\bV}\coloneqq \widetilde{V}_{\brQ_p}^{\Phi}$ where $\Phi=b\sigma$. We let $\tilde{\pi}=\sqrt{p}\in E$, and we consider $\widetilde{\pi} \bV \subset  \widetilde{\bV}$. Then, as in \cite[\S4.3]{HowardPappas}, we may consider $\widetilde{\bV}$ as a subspace of $\operatorname{End}\left(\bX\right)_{\mathbb{Q}}$; this is the space of special quasi-endomorphisms.  Let $\scrX$ denote the universal $p$-divisible group over $\cN_{L}$ and we let $\overline{\cN}_L$ denote the special fiber of $\calN_L$. 
We have the universal quasi-isogeny
$$
\rho: \bX \times_{\mathrm{Spf}(k)} \overline{\cN}_L \rightarrow \scrX \times_{\cN_L} \overline{\cN}_L.
$$

\begin{definition}
Let $\bbL\subset \bV$ be a non-degenerate lattice.
\begin{enumerate} \item   We define $\cZ_L(\bbL)$ to be the closed formal subscheme of $\cN_{L}$  where $\rho \circ x\circ \rho^{-1}\in \End(\overline{\scrX})_{\Q}$ lifts to an integral endomorphism of $\scrX$, for all $x\in \bbL$.
\item We define $\cY_L(\bbL)$ to be the closed formal subscheme of $\cN_{L}$ where $\rho \circ \widetilde{\pi} \circ x\circ \rho^{-1}\in \End(\overline{\scrX})_{\Q}$ lifts to an element of $\End(\scrX)$ for all $x\in \bbL$. 
\end{enumerate}
\end{definition}
Note that both  the framing object $\mathbb{X}$ and the universal $p$-divisible group $\mathscr{X}$ carry actions of $O_E$, and that the universal quasi-isogeny $\rho$ is $\mathcal{O}_E$-linear. In particular, we have   $\mathcal{Z}_L(\mathbb{L}) \subset \mathcal{Y}_L(\mathbb{L})$, which is analogous to the unitary case \cite{cho2022special}.

\subsubsection{}\label{para: duality iso} We now introduce a duality isomorphism which will allow us to reduce the study of $\calY$-cycles to that of $\calZ$-cycles.

\begin{definition}\label{def: L^sharp} For a quadratic lattice $L$, we use $L^{(p)}$ to denote the lattice $L$ with quadratic form $p\cdot q(\,)_L$ where $q(\,)_L$ is the quadratic form of $L$. Now for a vertex lattice $L\subset V$ of type $h$, we let $L^\sharp\coloneqq (L^\vee)^{(p)}$.  Moreover, we set $V^\sharp=L^\sharp\otimes_{\Z}\Q$.
In particular, $\widetilde{\pi}V\subset \widetilde{V}$ (resp. $\widetilde{\pi}  L^\vee\subset \widetilde{V}$) is a specific realization of $V^\sharp$ (resp. $L^\sharp$) in $\widetilde{V}$.
\end{definition}

\begin{proposition}\label{prop: duality for Z(Lambda)}
There is a natural isomorphism $$\Upsilon_{L,L^\sharp}:\calN_L\xrightarrow{\sim}\calN_{L^\sharp}$$
such that   for any non-degenerate lattice $\bL\subset \bV_L$, $\Upsilon_{L,L^\sharp}$ induces an isomorphism
 \begin{align*}
        \cY_L(\bbL)= \cZ_{L^\sharp}(\widetilde{\pi}   \bL).
    \end{align*} 
\end{proposition}
\begin{proof}
     We use $\calG_L$ to denote the parahoric group  scheme corresponding to $L$. Since we have a perfect pairing between $L$ and $L^\vee$,  
 $\calG_L$ is isomorphic to the stabilizer $\calG_{L^\vee}$ of $L^\vee$.   Moreover, we have $\calG_{L^\vee}\cong \calG_{L^\sharp}$ since the quadratic form of $L^\sharp$ only differs with the quadratic form of $L^\vee$ by a scalar $p$. In particular, we have $G^\sharp \coloneqq \mathrm{GSpin}(V^\sharp)\cong \mathrm{GSpin}(V)= G$.  

In fact, we can choose an isomorphism $\phi: G\xrightarrow{\sim} G^\sharp$ and choose $(G^\sharp,\phi(b),\phi\circ \mu)$ as a local Shimura datum of $\cN_{L^\sharp}$.  Moreover, we choose a basis $\{e_1,\ldots,e_n\}$ of $L$ as in \S \ref{6.1.2} so that $\{p^{-1}e_1,\ldots,p^{-1}e_{h},e_{h+1},\ldots,e_n\}$ is basis of $L^\vee$. Then we can identify $\widetilde{\pi} L^\vee$ with $L^\sharp$ which induces an identification 
\begin{align*}
   \tL=\widetilde{L^\sharp}  \quad \text{ and } \quad C(\tL)\stackrel{\varphi}{=}C(\widetilde{L^\sharp}).
\end{align*}
Under this identification,  $\widetilde{\pi} L^\vee=L^\sharp$ can be regarded as the same endomorphisms of $C(\tL)=C(\widetilde{L}^\sharp)$ via Lemma \ref{lem: V=B}. Moreover, we may define alternating forms $\psi$ and $\psi^{\sharp}$ on $C(\widetilde{L})$ and $C(\widetilde{L^\sharp})$ as in \S \ref{sec: alternating form} such that $\psi$ and $\psi^{\sharp}$ are   same  under this identification. For example, if we choose $c$ as in \S \ref{sec: alternating form} to define $\psi$, then  we can choose $\phi(c)$ to define $\psi^\sharp.$

Using these identifications, we have $\bX = \bX^\sharp\in \cN_{C(\tL)}(k)$, where $\bX^\sharp$ is the framing object of $\cN_{L^\sharp}$.  In particular, we can choose the same $\iota$ and $\Theta$ as in \eqref{eq: diagram} to construct $\cN_{L^\sharp}$ so that $\cN_{L}\stackrel{\Upsilon_{L,L^\sharp}}{\cong}\cN_{L^\sharp}$ in $\cN_{C(\tL)}$. Moreover,   by the identification $\widetilde{\pi} L^\vee=(L^\vee)^{(p)}$,  we regard $\widetilde{\pi} \bL$ and $\bL^{(p)}$   as the same special endomorphisms of $\bX=\bX^\sharp$ for any $\bL\subset \bV$ and $\bL^{(p)} \subset \bV^{(p)}=\bV^\sharp$. Then  the proposition follows from the definition of $\calZ$ and $\calY$-cycles.
\end{proof}

\subsection{$k$-points of $\calZ_L(\Lambda)$ and $\calY_L(\Lambda)$.}\subsubsection{} We now study the $k$-points of special cycles on $\calN_L$. In what follows, for notational simplicity, we will often omit the subscript $L$ in $\calZ_L(\Lambda)$ and $\calY_L(\Lambda)$. For a lattice $\bL\subset \bV$, we use $\cZ^{(0)}(\bbL)$  and $\cY^{(0)}(\bbL)$ to denote  the special cycles on $\cN_L^{(0)}$ (resp. $\cN_L^{(0)}\sqcup \cN_L^{(1)}$) if $h\neq 0,n$ (resp. $h=0,n$).

Recall that $\cV^{\ge h}(\bV)$ (resp. $\cV^{\le h}(\bV)$) denote the set of vertex lattice in $\bV$ of type $t\ge h$ (resp. $t\le h$). If the quadratic space is clear in the context, we simply denote it as $\cV^{\ge h} $ (resp. $\cV^{\le h} $).
 \begin{proposition}\label{prop: k points}\hfill
\begin{enumerate}
    \item For $\Lambda \in \cV^{\ge h}$, we have
\begin{align*}
\cZ^{(0)}(\Lambda)(k) &  \cong \{\text{special lattices } M \subset V_{\brQ_p}: \Lambda_{\brZ_p} \subset M \subset  M^\vee \subset \Lambda_{\brZ_p}^\vee \}\cong S_\Lambda(k).
\end{align*}
\item For $\Lambda \in \cV^{\ge h}$, we have
\begin{align*}
\cY^{(0)}(\Lambda^\vee)(k) &  \cong \{\text{special lattices } M \subset V_{\brQ_p}: p\Lambda_{\brZ_p}^\vee \subset  p M^\vee \subset  M \subset \Lambda_{\brZ_p} \}\cong R_\Lambda(k).  
\end{align*}
\end{enumerate}
 \end{proposition}
\begin{proof}

Recall that previously we have defined
$$
L_g\coloneqq \sigma^{-1}(\tilde{b}^{-1}g)\bullet L_{\brZ_p} \quad \text { and } \quad     L_g'\coloneqq g \bullet L_{\brZ_p}.
$$
We claim that 
\begin{align*}
    (L_y,L_y')=(L_g,L_g').
\end{align*}
First, when $g=1$ and $y$ is the point corresponding to $g=1$, the claim follows from \eqref{eq: L} and the definition of $L_y'$. Applying $\sigma^{-1}(\tilde{b}^{-1}g)\bullet$ to both sides, we have $L_y=L_g$. Similarly, applying $g\bullet $ to both sides, we obtain $L_y' =L_g'$.

Now for $y\in \cN_L(k)$, we have  $y \in \cZ^{(0)}(\Lambda)(k) $  if and only if  $\Lambda$ preserves both lattices $M_{y,1} \subset M_y$, which is equivalent to $\Lambda\subset L_y \cap L_y'$. However, since $\Lambda$ is $\Phi$-invariant, $\Lambda\subset L_y \cap L_y'$ is equivalent to $\Lambda\subset L_y$. Therefore,
$$
\begin{aligned}
\cZ^{(0)}(\Lambda)(k) & \cong  \left\{ \text { special lattices } L_y \subset V_{\brF}: \Lambda  \subset L_y\right\} \\
& =\left\{ \text { special lattices } L_y \subset V_{\brF}: \Lambda  \subset L_y \subset L_y^\vee \subset \Lambda^\vee \right\}.
\end{aligned}
$$

The description for $  \cY^{(0)}(\Lambda^\vee)(k)$ follows from the same argument once we replace \eqref{eq: L} by \eqref{eq: L^vee}.

Finally, we can easily verify the following identifications by definition of $S_\Lambda$ and $R_\Lambda$:
\begin{align*}
    \{\text{special lattices } M \subset V_{\brQ_p}: \Lambda_{\brZ_p} \subset M \subset  M^\vee \subset \Lambda_{\brZ_p}^\vee \}&\stackrel{\sim}{\to} S_\Lambda(k),\\
    M&\mapsto M/\Lambda_{\brZ_p}.\\
     \{\text{special lattices } M \subset V_{\brQ_p}: p\Lambda_{\brZ_p}^\vee \subset p M^\vee \subset  M \subset \Lambda_{\brZ_p}  \}&\stackrel{\sim}{\to} R_\Lambda(k),\\
    p M^\vee&\mapsto pM^\vee/p\Lambda_{\brZ_p}^\vee.
\end{align*}
\end{proof}

\subsubsection{}The following two propositions explain the intersection behaviour between special cycles associated to different choices of vertex lattice $\Lambda.$

\begin{proposition}\label{prop: inter of Z(Lambda)}
    Assume $\Lambda$ and $\Lambda'$ are both vertex lattices in $\bbV$. 
    \begin{enumerate}
        \item If $\Lambda\subset \Lambda'$, then $\cZ^{(0)}(\Lambda')(k)\subset \cZ^{(0)}(\Lambda)(k)$.
        \item We have 
        \begin{align*}
            \cZ^{(0)}(\Lambda)(k) \cap \cZ^{(0)}(\Lambda^{\prime})(k)= \begin{cases}\cZ^{(0)}(\Lambda + \Lambda^{\prime})(k) & \text { if } \Lambda+ \Lambda^{\prime} \in \cV^{\ge h} , \\ \emptyset & \text { otherwise. }\end{cases}
        \end{align*}
    \end{enumerate}
\end{proposition}
\begin{proof}
    (i) directly follows from Proposition \ref{prop: k points}. According to Proposition \ref{prop: k points}, we have 
    \begin{align*}
         &\cZ^{(0)}(\Lambda)(k) \cap \cZ^{(0)}(\Lambda^{\prime})(k)\\
         &=\{\text{ special lattices } L_y \subset V_{\brF}: \Lambda+\Lambda' \subset L_y \subset L_y^\vee \subset (\Lambda+\Lambda')^\vee\}.
    \end{align*}
In particular, if $\Lambda+ \Lambda^{\prime}$ is not a vertex lattice, then we obtain an empty set.
\end{proof}

\begin{proposition}\label{prop: inter of Y(Lambda)}
    Assume $\Lambda$ and $\Lambda'$ are both vertex lattices in $\bbV$. 
    \begin{enumerate}
        \item If $\Lambda\subset \Lambda'$, then $\cY^{(0)}((\Lambda')^\vee)(k)\subset 
 \cY^{(0)}(\Lambda^\vee)(k)$.
        \item We have 
        \begin{align*}
            \cY^{(0)}(\Lambda^\vee)(k) \cap \cY^{(0)}((\Lambda^{\prime})^\vee)(k)= \begin{cases}\cY^{(0)}(\Lambda^\vee + (\Lambda^{\prime})^\vee)(k) & \text { if } \Lambda+ \Lambda^{\prime} \in \cV^{\le h}, \\ \emptyset & \text { otherwise. }\end{cases}
        \end{align*} 
    \end{enumerate}
\end{proposition}

 \subsubsection{} Let $\cN_{L,\red}^{(0)}$ denote the reduced $k$-scheme underlying $\cN_L^{(0)}$. Similarly, let $\cZ_\red^{(0)}(\Lambda)$ and $\cY_\red^{(0)}(\Lambda^\vee)$ be the reduced $k$-scheme underlying $\cZ^{(0)}(\Lambda)$ and $\cY^{(0)}(\Lambda^\vee)$. 

 The following is one of our main results, which states that $\cN_{L,\red}^{(0)}$ is a union of two types of stratas given by $\cZ_\red^{(0)}(\Lambda)$  and $\cY_\red^{(0)}(\Lambda^\vee)$ respectively.

\begin{theorem}\label{thm: N=ZunionY}
Assume $h\neq 0,n$. 
  We have  
    \begin{align*}
        \cN_{L,\red}^{(0)} = \Big (\cup_{\Lambda \in \cV^{\ge h}} \cZ_{\red}^{(0)}(\Lambda) \Big ) \cup \Big(\cup_{\Lambda \in \cV^{\le h}} \cY_{\red}^{(0)}(\Lambda^\vee) \Big).
    \end{align*}
\end{theorem}
\begin{proof}
    Proposition \ref{prop: k points of N} identifies the $\cN_{L,\red}^{(0)}(k)$ with special lattices $M\subset V_{\brQ_p}$ such that $M\stackrel{h}{\subset}M^\vee$. Proposition  \ref{prop: key prop} combined with Proposition \ref{prop: k points} implies that a point $z\in \cN_{L,\red}^{(0)}(k)$ lies either in  $\cZ_{\red}^{(0)}(\Lambda)(k)$ for a $\Lambda \in \cV^{\ge h}$ or in $\cY_{\red}^{(0)}(\Lambda^\vee)(k)$ for  a $\Lambda \in \cV^{\le h}$.  
\end{proof}

The following shows that $\cN_{L,\red}^{(0)}$ is connected as in many previously studied cases.
\begin{proposition}
 $\cN_{L,\red}^{(0)}$ is connected.
  \end{proposition}
\begin{proof}
The case $h=0,n$ is  \cite[Proposition 5.1.5]{HowardPappas}. So we assume $h\neq 0,n$.
We say two vertex lattices $\Lambda,\Lambda'\subset \bbV$ are adjacent if $\Lambda \subset \Lambda'$ or $\Lambda'\subset \Lambda$. We write $\Lambda\sim \Lambda'$ if $\Lambda$ and $\Lambda'$ are adjacent. The proof of \cite[Proposition 5.1.5]{HowardPappas} shows that any $\Lambda$ with type $t_\Lambda\ge h$ (resp. $t_\Lambda\le h$) can be connected with a fixed vertex lattice $\Lambda_h$ with type $h$ via adjacent vertex lattices with type at least $h$ (resp. at most $h$). In  particular, $\cup_{\Lambda_t \in \cV^{\ge h}} \cZ_{\red}^{(0)}(\Lambda_t)$ (resp. $\cup_{\Lambda_t \in \cV^{\le h}} \cY_{\red}^{(0)}(\Lambda_t^\vee)$) is connected by Proposition \ref{prop: inter of Z(Lambda)} (resp. Proposition \ref{prop: inter of Y(Lambda)}).  Moreover,  $\cZ_\red^{(0)}(\Lambda_h)$ is the worst point associated to $\Lambda_h$ and $\cZ_\red^{(0)}(\Lambda_h)\subset \cY_{\red}^{(0)}(\Lambda^\vee)$ if $p\Lambda^\vee \subset \Lambda_h$. In particular,
$$\cZ_\red^{(0)}(\Lambda_h)\subset (\cup_{\Lambda_t \in \cV^{\ge h}} \cZ_{\red}^{(0)}(\Lambda_t))\cap (\cup_{\Lambda_t \in \cV^{\ge h}} \cY_{\red}^{(0)}(\Lambda_t^\vee)).$$ 
So $\cN_{L,\red}^{(0)}$ is connected.
\end{proof}

\subsubsection{}

Given Theorem \ref{thm: N=ZunionY}, to understand $\cN_{L,\red}^{(0)}$, it suffices to describe $\cZ_\red^{(0)}(\Lambda)$ and  $\cY_\red^{(0)}(\Lambda^\vee)$.  It turns out that  $\cZ_\red^{(0)}(\Lambda)$ and $\cY_\red^{(0)}(\Lambda^\vee)$ is isomorphic to $S_\Lambda$ and $R_\Lambda$ introduced in \S \ref{sec: DL varieties} respectively. The rest of the paper is devoted to the proof of this statement.

First, we will construct a proper morphism from    $\cZ_\red^{(0)}(\Lambda)$ and $\cY_\red^{(0)}(\Lambda^\vee)$   to $S_\Lambda$ and $R_\Lambda$ respectively. We will need the following proposition, whose proof is the same as the proof of \cite[Proposition 6.1.2]{HowardPappas}.
\begin{proposition}\label{prop: proj of Z(Lambda)}
Assume $\Lambda$ is a vertex lattice of type $t\ge h$ (resp. $t\le h$). Then $\cZ_\red^{(0)}(\Lambda)$ (resp. $\cY_\red^{(0)}(\Lambda^\vee)$) is projective.
\end{proposition}\qed

\begin{proposition}\label{prop: morphism BT strata to DLV}
  There are natural proper morphisms between $k$-schemes
$$
\Psi_{\calZ,\Lambda}: \cZ_\red^{(0)}(\Lambda) \rightarrow S_{\Lambda},
 $$
 and
 $$
\Psi_{\calY,\Lambda}: \cY_\red^{(0)}(\Lambda^\vee)  \rightarrow R_{\Lambda}
 $$
inducing the bijections from Proposition \ref{prop: k points} on $k$-points.
\begin{proof} We consider the case of $\calZ(\Lambda)$; the case of $\calY(\Lambda)$ is completely analogous.

    Let $R$ be a reduced $k$-algebra of finite type and $y\in \cZ_\red^{(0)}(\Lambda)(R)$ corresponding to a triple $\left(X_y,\lambda_y, \rho_y \right)$ as in \S\ref{para: RZ space GSp}. Let $\calR$ be a $p$-torsion free, $p$-adically complete and separated $\brZ_p$-algebra which is equipped with an isomorphism $\calR/p\calR\cong R$ and a lift (also denoted as $\sigma$) of the Frobenius $\sigma$ on $R$.  We set $\calM_y:=\bbD(X_y)(\calR)$ and let $\calM_{y,1}\subset \calM_{y}$ denote the preimage of the filtration on $\bbD(X_y)(R)$. Using the quasi-isogeny $\rho_y:\bbX\times_k R\rightarrow X_y$, we may consider  $\calM_y$ as an $\calR$-submodule of $D_{\brZ_p}\otimes_{\brZ_p} \calR[1/p].$ We define $$ \calL_y=\{x\in V\otimes_{\bbQ_p}\calR[1/p]\mid x\calM_{y,1}\subset \calM_{y,1}\}$$ which is an $\calR$-submodule of $V\otimes_{\bbQ_p}\calR[1/p]$. Then we claim that $$\Lambda\otimes_{\bbZ_p}\calR \subset\calL_y\subset \Lambda^\vee\otimes_{\bbZ_p}\calR, $$ and the image $\overline{\calL}_y$ of $\calL_y$ in $\Lambda^\vee\otimes_{\bbZ_p}\calR/\Lambda\otimes_{\bbZ_p}\calR=\Lambda^\vee/\Lambda \otimes_{\bbF_p}R$  is a locally free direct summand of rank $\frac{t_\Lambda-h}{2}$. Indeed since $R$ is reduced, it suffices by \cite[Exercise X.16]{Lang}  to check this on $k$-points. For $R\rightarrow k$ corresponding to $x\in \calZ(\Lambda)(k)$, it follows from Lemma \ref{lem: V=B} that the specialization of $\calL_y$ is the sublattice $L_x\subset \breve V$, and hence its image in $\Lambda^\vee/\Lambda\otimes_{\bbZ_p} k$ lies in $S(\Lambda)$.  Thus $\overline{\calL}_y$ is an element of $S_\Lambda(R)$, and it is easy to see that it does not depend on the choice of $\calR$. Indeed, since $R$ is reduced, $\overline{\calL}_y$ is determined by its restriction along all $R\rightarrow k$. We obtain a morphism $\cZ_\red^{(0)}(\Lambda)  \rightarrow S_{\Lambda}$  as desired. 
    
    Finally, $\cZ_\red^{(0)}(\Lambda)$ is projective over $k$ by Proposition \ref{prop: proj of Z(Lambda)} and $S_\Lambda$ is projective over $k$ as a generalized Deligne-Lusztig variety. The properness of the morphism follows.
\end{proof}

\end{proposition}

\subsection{Orthogonal local models}\label{sec: ortho local model} Our goal in the rest of the paper is to show the morphisms in Proposition \ref{prop: morphism BT strata to DLV} are isomorphisms. In order to do this, we will give a description of the  tangent space of $\calZ_{\red}^{(0)}(\Lambda)$ at certain $k$-points generalizing \cite[Corollary 4.2.7]{LiZhu2018}. We begin with some preparations concerning orthogonal local models.

We keep the notation of the previous subsection so that $L\subset V$  is a vertex lattice of type $h$ and $\cG$ is the parahoric of $G=\GSpin(V)$ corresponding to $L$.  
We recall the some properties of the local models $\Mloc_{\calG,\mu}$ from \cite{AGLR} that we will need. 

We let $\mathcal{FL}_{\calG}$ denote the Witt vector partial affine flag variety for $\calG$ constructed in \cite{BhattScholze}, \cite{Zhu}. We write $\calM$ for the geometric special fiber of $\Mloc_{\calG,\mu}$.
Then by \cite[Theorem 7.23]{AGLR}, $\calM$ can be identified with the canonical deperfection $\calA_{\calG,\mu}^{\mathrm{can}}$ of the $\mu$-admissible locus inside $\mathcal{FL}_{\calG}$. In particular, $\calM$ has a stratification indexed by $w\in \Adm(\mu)_K$  whose closure relations are given by the Bruhat order on $\tW$. We write $\calM^w$ for the strata corresponding to $w\in \Adm(\mu)_K$. 

On the level of $k$-points, we may identify $\calM(k)$ with a subset of $G(\brQ_p)/\calG(\brZ_p)$ corresponding to the elements of $h\dot{w}$, where $h\in \calG(\brZ_p)$ and $w\in \Adm(\mu)_K.$ Let $\iota^{\mathrm{loc}}:\Mloc_{\calG,\mu}\rightarrow \Gr(C(\tL),d)$ denote the map of local models constructed in Proposition \ref{prop: very good embedding}. Then for $g\in G(\brQ_p)/\calG(\brZ_p)$ lying in the image of $\calM(k)$, the filtration on $C(\tL)$ corresponding to its image in $ \Gr(C(\tL),d)(k)$ is induced by the reduction mod $p$ of $gC(\tL)\subset C(\tL)$.

\subsubsection{}\label{sec: KR strata}Let $\Adm(\mu)_K^{\max}\subset \Adm(\mu)_K$ denote the subset of maximal elements for the Bruhat order; explicitly these are given by ${^K\tW}\cap t^{W_0(\mu)}$. In what follows we give an explicit description of $\Adm(\mu)^{\max}_K$, together with a distinguished subset $Z_K\subset\Adm(\mu)_K$; this subset corresponds to the KR-strata whose closure contains  the $\calZ$-cycles; see \S\ref{para: LMD}.  We identify $X_*(T^{\mathrm{ad}})$ with $\bbZ^n$ with standard basis $\epsilon_1,\dotsc,\epsilon_n$. Then $\mu=\epsilon_1$, and the Weyl orbit $W_0(\mu)$ is the set $\{\pm\epsilon_i\mid i=1,\dotsc,n\}$. We write $$h=\begin{cases} 2s & \text{ in Case (1) and (2a)},\\
2s+1 &\text{ in Case (2b) and (3)}.
\end{cases}$$
Here we label the cases as in \ref{subsec: VL cases}.

Case (1) \emph{(Type $D_n$)}. 

$$\Adm(\mu)_K^{\max}=\begin{cases}
\{t^{\epsilon_1}\} &\text{if $s=0$}\\
\{t^{\pm\epsilon_1},t^{\epsilon_2}\} &\text{if $s=1$},\\
\{t^{-\epsilon_{s}},t^{\epsilon_{s+1}}\} &\text{$1<s<n-1$},\\
\{t^{-\epsilon_{n-1}},t^{\pm\epsilon_n}\} &\text{if $s=n-1$},\\
\{t^{-\epsilon_n}\} &\text{if $s=n$}.
\end{cases}$$

We set $$Z_K=\begin{cases}
\{t^{\epsilon_{s+1}}\} &\text{if $0\leq s<n-1$},\\
\{t^{\pm\epsilon_n}\} &\text{if $s=n-1$},\\
\emptyset&\text{if $s=n$}.
\end{cases}$$

Case (2a)  \emph{(Type $B_n$)}. 
$$\Adm(\mu)_K^{\max}=\begin{cases}
	\{t^{\epsilon_1}\} &\text{if $s=0$},\\
	\{t^{\pm\epsilon_1},t^{\epsilon_2}\} &\text{if $s=1$},\\
	\{t^{-\epsilon_{s}},t^{\epsilon_{s+1}}\} &\text{if $1<s\leq n-1$},\\
	\{t^{-\epsilon_n}\} &\text{if $s=n$}.
\end{cases}$$

	We set $$Z_K=\begin{cases}
	\{t^{\epsilon_{s+1}}\} &\text{if $0\leq s<n$}\\
	\emptyset&\text{if $s=n$}
	\end{cases}$$

Case (2b)  \emph{(Type $B_n$)}. 
$$\Adm(\mu)_K^{\max}=\begin{cases}
	\{t^{\epsilon_1}\} &\text{if $s=0$},\\
	\{t^{-\epsilon_{s}},t^{\epsilon_{s+1}}\} &\text{if $1\leq s< n-1$},\\
    	\{t^{\pm\epsilon_{n}},t^{-\epsilon_{n-1}}\} &\text{if $s=n-1$},\\
	\{t^{-\epsilon_n}\} &\text{if $s=n$}.
\end{cases}$$

   We set $$Z_K=\begin{cases}
	\{t^{\epsilon_{s+1}}\} &\text{if $0\leq s<n-1$},\\
    \{t^{\pm\epsilon_{n}}\} &\text{if $s=n-1$},\\
	\emptyset&\text{if $s=n$}.
	\end{cases}$$

	Case (3) \emph{(Type $C-B_n$)}. 
	
	$$\Adm(\mu)_K^{\max}=\begin{cases}
		\{t^{\epsilon_1}\} &\text{if $s=0$},\\
		\{t^{-\epsilon_{s}},t^{\epsilon_{s+1}}\} &\text{if $1\leq s\leq n-1$},\\
		\{t^{-\epsilon_n}\} &\text{if $s=n$}.
	\end{cases}$$
    
	We set $$Z_K=\begin{cases}
\{t^{\epsilon_{s+1}}\} &\text{if $0\leq s<n$},\\
\emptyset&\text{if $s=n$}.
\end{cases}$$

\subsubsection{}We set $\calM_Z=\bigcup_{w\in Z_K}\calM^w$. Let $\calM^{\sm}$ denote the smooth locus of $\calM$. 

\begin{lemma}\label{lem: Z-cycle smooth locus}\hfill
\begin{enumerate}\item Each $\calM^w$ is an orbit for the action of $\calG_k$ on $\calM$.

\item 	We have $\bigcup_{w\in Z_K}\calM^w\subset \calM^{\sm}$, where $\calM^{\sm}$ is the smooth locus of $\calM$.
	
	\end{enumerate}
\end{lemma}

\begin{proof} (i) By \cite[Lemma 3.15]{AGLR}, $\calM$ may be identified with a certain union of Schubert varieties in the partial affine flag variety $\mathcal{FL}_{\underline{\calG}}$ for an equicharacteristic group $\underline{\calG}$ over $k[[t]]$ which is equipped with an identification $\underline{\calG}_k\cong \calG_k$. The action of the positive loop group $L^+\underline{\calG}$ on $\calM$ factors through $\underline{\calG}_k$ (cf. \cite[Theorem 3.16]{AGLR}), and the corresponding action of $\underline{\calG}_k$ can be identified with the action of $\calG_k$ under the isomorphism $\underline{\calG}_k\cong \calG_k$ above.

(ii) The smooth locus of $\calM$ is open dense, and is a union of $\calG_{k}$-orbits. Thus $\calM^{\sm}$ contains all the strata $\calM^w$ corresponding to $w\in \Adm(\mu)_K^{\max}$.

\end{proof}

\subsubsection{} For notational simplicity, we write $\breve V$ for $V_{\brQ_p}$, $\breve L$ for $L_{\brZ_p}$ and $\brD$ for $D_{\brZ_p}$.
Let $x\in \calM^w(k)$ for some $w\in Z_K$. This corresponds to a subspace $\overline{D}_{1}\subset  \overline{D}:=\brD/p \brD$, and we let $\breve D_1\subset \breve D$ denote the preimage of $\overline{D}_1$. 
By Lemma \ref{lem: V=B}, we have $$\brL=\{v\in \brV\mid vD\subset D\}.$$
We equip $\overline L:=\brL/p \brL$ with the structure of a (possibly degenerate) quadratic space over $k$ and set 
$
\brL_1:=\{v\in \brV\mid v\brD_1\subset \brD_1\}$. This is equipped with an action of $\calG_k$ via the $\bullet$ action which preserves the quadratic form.
\begin{lemma}\label{lem: L-filtration}\hfill
	\begin{enumerate}
		\item  $\dim(\brL+\brL_1)/\brL=1$.

    \item  The $1$-dimensional isotropic line $\overline{L}_1$ in $\overline{L}$ spanned by the image of $p\brL_1$ does not lie in the radical of $\overline{L}$.

    	\end{enumerate}
\end{lemma}
\begin{proof}We write $\breve V=\breve V_0\oplus \breve V_{\an}$  and choose a basis $e_1\dotsc e_n,f_1,\dotsc,f_n$ for $\breve V_{0}$ as in \ref{sec: quadratic space over breveF}. Then  $$\brL=\Span(pe_1,\dotsc,pe_s,e_{s+1},\dotsc,e_n,f_1,\dotsc,f_n)\oplus \brL_{\an}$$ with $s$ as above.
		
Let $g\in G(\brQ)/\calG(\brZ_p)$ be the image of $x$ under the map $\calM(k)\rightarrow G(\brQ)/\calG(\brZ_p)$. Then $ \brD_1=g \brD$, and hence $\brL_1=g\bullet \brL$. 

We write $g= \dot t^{v(\mu)}k$ with $k\in \calG(\brZ_p)$ and $t^{v(\mu)}\in Z_K$. Then $k$ induces an automorphism of $\overline{L}$ which preserves the associated quadratic form. Thus we may replace $\brL_1$ by $k^{-1}\cdot \brL_1$ and assume $g=\dot t^{v(\mu)}$. Then (i) is clear.

For (ii), we have the following description of $\mathrm{rad}(\overline{L})$. We let  $\overline L'$ be the radical of the quadratic space $\brL_{\an}/p\brL_{\an}$. Then we have
 $$\mathrm{rad}(\overline{L})=\Span(\overline{pe}_1,\dotsc,\overline{pe}_s,\overline{f}_1,\dotsc,\overline{f}_s)\oplus \overline{L}'$$
where for $v\in \brL$, we wrote $\overline{v}$ for its image in $\overline{L}$. 

From the description of $Z_K$ above, we see that the image of $p\brL_1$ in $\overline{L}$ is generated by either $\overline{e}_{s+1}$ or $\overline{f}_{s+1}$. The result follows.
\end{proof}

\subsubsection{}\label{para: subspace local model}The group $\calG_k$ acts on $\overline{L}$, and so we have a map of Lie algebras $$\Lie\calG_k\rightarrow \End_k(\overline{L}),\ h\mapsto \overline h.$$ We write $k[\epsilon] $ for the ring of dual numbers $k[t]/t^2$ over $k$. Then we can identify $\Lie \calG_k$ with elements of $\calG_k(k[\epsilon])$ lifting the identity. Thus for $h \in \Lie \calG_k$, we obtain an automorphism of $\overline{L}_{k[\epsilon]}$ given by $1+\overline{h}\epsilon$. We write $\scrD$ for the set of isotropic rank one $k[\epsilon]$-submodules of  $\overline{L}_{k[\epsilon]}$ lifting $\overline{L}_1$, i.e. $\scrD$ is the tangent space of $\OGr(\overline{L},1)$ at $\overline{L}_1$. Then $\Lie\calG_k$  acts on $\scrD$   and we have a canonical isomorphism
$$\phi:\scrD\rightarrow \Hom_k(\overline{L}_1,\overline{L}_1^\perp/\overline{L}_1),$$ where $\overline{L}_1^\perp$ denotes the orthogonal complement of $\overline{L}_1$ in $\overline{L}$. Here, $\phi$ is defined as follows. For $F\in \scrD$ and $m\in \overline{L}_1$, we write $F=\langle m+v\epsilon\rangle$ for some $v\in \overline{L}$ well-defined up to addition by an element of $\overline{L}_1$. Since $F$ is isotropic,  we have $(m,v)=0$, and we define $\phi(F)(m)=v$.
\begin{lemma}
The map $$\Theta:\Lie\calG_k\rightarrow \Hom_k(\overline{L}_1,\overline{L}_1^\perp/\overline{L}_1)$$ given by $h\mapsto \phi((1+\overline h\epsilon)\overline{L}_{1,k[\epsilon]})$ is a surjective homomorphism of vector spaces.

\end{lemma}
\begin{proof}
Note that $\Theta$ is induced by the map $h\mapsto (m\mapsto \overline{h}m\mod \overline{L}_1)$ for $h\in \Lie\calG_k$. Here as above, $\overline{h}\in \End(\overline{L})$ is the image of $h$; it is  an  endomorphism of $\overline{L}$ which preserves $\overline{L}_1^\perp$. Since the map $\Lie\calG_k\rightarrow \End_k(\overline{L}),\ h\mapsto \overline{h},$ is a morphism of vector spaces, it follows that $\Theta$ is a morphism of vector spaces. 

To see that $\Theta$ is surjective, it suffices to  show that $\Im(\Theta)$ contains a basis for $\Hom_k(\overline{L}_1,\overline{L}_1^\perp/\overline{L}_1)$. As in Lemma \ref{lem: L-filtration},  we may assume $\overline{L}_1=\langle \overline{e}_{s+1}\rangle$ or $\langle \overline{f}_{s+1}\rangle$. We consider the case $\overline{L}_1=\langle \overline{e}_{s+1}\rangle$; the other case is analogous. A basis of the vector space $\Hom_k(\overline{L}_1,\overline{L}_1^\perp/\overline{L}_1)$ is then given by the maps which send $\overline{e}_{s+1}$ to the images of $$\overline{pe}_1,\dotsc,\overline{pe}_s,\overline{e}_{s+2},\dotsc,\overline{e}_n,\overline{f}_1,\dotsc,\overline{f}_s, \overline{f}_{s+2},\dotsc,\overline{f}_n$$
in $\overline{L}_1^\perp$.   

For $i<j$, let $U_{\pm i,\pm j}$ denote the root subgroup of $G$ corresponding to the root $\pm \epsilon_i\pm\epsilon_j$. We fix an isomorphism $u_{\pm i,\pm j}:\bbG_a\rightarrow U_{\pm i,\pm j}$ so that \begin{align*}u_{i, -j}(x) e_\ell=e_\ell+x\delta_{i\ell}e_j \ \ \ \ u_{i,- j}(x) f_\ell=f_\ell+x\delta_{i\ell}f_j\\ u_{-i, j}(x) e_\ell=e_\ell+x\delta_{j\ell}e_i\ \ \ \ u_{-i, j}(x) f_\ell=f_\ell+x\delta_{j\ell}f_i \\
u_{i, j}(x) e_\ell=e_\ell+x\delta_{i\ell}f_j\ \ \ \ u_{i, j}(x) f_\ell=f_\ell+x\delta_{i\ell}e_j\\
u_{-i,- j}(x) e_\ell=e_\ell+x\delta_{j\ell}f_i \ \ \ \ u_{-i,- j}(x) f_\ell=f_\ell+x\delta_{j\ell}e_i
\end{align*}
for $1\leq \ell\leq n$.

Now let $\ell>s+1$. The image of $u_{s+1,-\ell}(\calO_{\breve F})$ (resp. $u_{s+1,\ell}(\calO_{\breve F})$) modulo $p$ gives a unipotent subgroup of $\calG_k$ isomorphic to $\bbG_a$ over $k$ and the corresponding element  of $\Lie\calG_k$ maps $\overline{e}_{s+1}$ to $\overline{e}_\ell$ (resp. $\overline{f}_\ell$). 

Now let $\ell<s+1$. In  this case the image of $u_{s+1,-\ell}(p\calO_F)$ (resp. $u_{s+1,\ell}(\calO_F)$) modulo $p$ gives a unipotent subgroup of $\calG_k$ isomorphic to $\bbG_a$ over $k$ and the corresponding element  of $\Lie\calG_k$ maps $\overline{e}_{s+1}$ to $\overline{pe}_\ell$ (resp. $\overline{f}_\ell$).  

Thus the image of $\Theta$ contains a basis for $\Hom_k(\overline{L}_1,\overline{L}_1^\perp/\overline{L}_1)$ and hence $\Theta$ is surjective.
\end{proof}
\subsubsection{}

We continue to assume $x\in \calM^w(k)$ for some $w\in Z_K$, and let $v\in T_x\calM$; thus $v$ corresponds to an element $v\in \calM(k[\epsilon])$ extending $x$. We write $m_x\in \overline{L}$ for a generator of the 1-dimensional line given by the image of $\brL_1$.  Via the embedding $\iota^{\mathrm{loc}}$, we obtain a submodule $\overline{D}_{1,v}\subset \overline{D}_{k[\epsilon]}$. Note that  $\overline{D}_{k[\epsilon]}$ is equipped with an action of $
\overline{L}_{k[\epsilon]}$, and we set $$\overline{L}_{1,v}=\{m\in \overline{L}_{k[\epsilon]}\mid m\overline{D}_{1,v}\subset \overline{D}_{1,v}\}.$$

\begin{proposition}\label{prop: tangent space LM}\hfill
    \begin{enumerate}
        \item $\overline{L}_{1,v}$ is an isotropic rank one $k[\epsilon]$-submodule in $\overline{L}_{k[\epsilon]}$ lifting $\overline{L}_1$, i.e. $\overline{L}_{1,v}\in \scrD$.
        \item The association $v\mapsto \phi(\overline{L}_{1,v})$ defines an isomorphism $$T_x\calM\rightarrow \Hom_k(\overline{L}_1,\overline{L}_1^\perp/\overline{L}_1).$$
    \end{enumerate}
\end{proposition}
\begin{proof}   Note that since $w\in Z_K\subset \Adm(\mu)_K^{\max}$, we have $T_x\calM=T_x\calM^w$. We consider $x$ as an element of $\calM_x(k[\epsilon])$ via $\epsilon\mapsto0 $. Then we have $\overline{L}_{x,1}= \overline{L}_{1,k}\otimes_k k[\epsilon]$ which is isotropic and  free of rank $1$ over $k[\epsilon]/\epsilon^2$. By Lemma \ref{lem: Z-cycle smooth locus} (2), $\calM^w$ is an orbit of the action of $\calG_k$. Therefore there exists $g\in \calG_k(k[\epsilon])$ such that $$g\equiv 1\mod v\text{ and } gx=v.$$  Then we have $\overline{L}_{1,v}=g\bullet \overline{L}_{x,1}$ which is therefore isotropic and free of rank 1 over $k[\epsilon]/\epsilon^2$, and hence we obtain (i).
	
	For (ii), note that the association $T_x\calM\rightarrow \scrD$ defined in (1) is equivariant for the action of $\Lie\calG_k$. We define a map $\Lie\calG_k\rightarrow T_x\calM$ by considering the action of $\Lie\calG_k$ on $x$. Then we have a commutative diagram \[\xymatrix{\Lie\calG_k\ar[rd]\ar[d]& \\
	T_x\calM\ar[r]& \scrD}\]
where the diagonal map is given by $h\mapsto (1+\overline h\epsilon)\overline{L}_{1,k[\epsilon]}$. It follows that the composition $$\Lie\calG_k\rightarrow T_x \calM\rightarrow \scrD\xrightarrow{\phi}\Hom_k(\overline{L}_1,\overline{L}_1^\perp/\overline{L}_1)$$ is equal to the map $\Theta$, and hence  $T_x\calM\rightarrow \overline{L}_1^\perp/\overline{L}_1$ is surjective. 

By Lemma \ref{lem: Z-cycle smooth locus}, we have \begin{align*}\dim T_x\calM&=\dim \Mloc_{\calG}\\&=\langle \mu,2\rho\rangle\\
&=\begin{cases} 2n-2 & \text{ in Case (1)}\\
2n-1 & \text{ in Case (2a) and (2b)}\\
2n &\text{ in Case (3)}
\end{cases}\\
&=\dim\overline{L}_1^\perp/\overline{L}_1.\end{align*}
Here, $\rho$ is the half-sum of positive roots, and the last equality follows from Lemma \ref{lem: L-filtration}. Thus $T_x\calM\rightarrow \overline{L}_1^\perp/\overline{L}_1$ is an isomorphism by comparing dimensions.
	    \end{proof}

		\subsection{Description of  $\calZ_L(\Lambda)$ and  $\calY_L(\Lambda)$}\subsubsection{}\label{para: LMD} We now apply the above to compare  $\calZ^{(0)}(\Lambda)$ to the Deligne--Lusztig variety $S_\Lambda$. This will immediately imply the corresponding description of $\calY^{(0)}(\Lambda)$ via the duality isomorphism in \S\ref{para: duality iso}.

  By construction there is a morphism $\calN_L\rightarrow \widehat{\scrS}_{\rmK}$ where we recall that $\widehat{\scrS}_{\rmK}$ denotes the completion of $\scrS_{\rmK}\otimes\brZ_p$ along its special fiber. Pulling back the local model diagram from $\widehat{\scrS}_{\rmK}$ (see \cite[Theorem 0.4]{KisinPappas}, \cite[Theorem 1.1.2]{KPZ}), we obtain  a local model diagram for $\calN_L$:
	\[\xymatrix{ & \widetilde{\calN_L}\ar[rd]^q\ar[ld]_\pi& \\
	\calN_L& &\widehat{\bbM}^{\mathrm{loc}}_{\calG,\mu}}\]
	where   $\pi$ is a $\widehat{\calG}$-torsor and $q$ is $\widehat{\calG}$-equivariant. Here $\widehat{\calG}$ (resp.\ $\widehat{\bbM}^{\mathrm{loc}}_{\calG,\mu}$  ) denotes the completion of $\calG$ (resp $\Mloc_{\calG,\mu,\brZ_p}$) along its special fiber.	This induces a stratification on $\calN_{L,\red}$ indexed by $\Adm(\mu)_K$; in particular we obtain  a map $$\pi_{KR}:\calN_L(k)\rightarrow \Adm(\mu)_K.$$
	
We have the following description of the map $\pi_{KR}$. Let $y\in \calN_{L,\red}(k)$, and let $L'_y\subset \breve V$ denote the associated special lattice. Then $L'_y=g\bullet \breve L$ for some $g\in X(\mu^\sigma,b)$ so that $g^{-1}b\sigma(g)\in \calG \sigma(\dot{w})\calG$ for some $w\in \Adm(\mu)_K$. Then we have $\pi_{KR}(y)=w$.
	\begin{lemma}\label{lem: KR strata}
    We have $L'_y\nsubseteq L_y^{\vee}$ if and only if $\pi_{KR}(y)\in Z_K$.
	\end{lemma} 
	\begin{proof}It suffices to show that for $w\in\Adm(\mu)_K$, we have $\dot{w}\brL\nsubseteq \brL^\vee$ if and only if $w\in Z_K$. Recall that $$\brL=\Span(pe_1,\dotsc,pe_s,e_{s+1},\dotsc,e_n,f_1,\dotsc,f_n)\oplus \brL_{\an}$$ 
$$\brL^\vee=\Span(e_1,\dotsc,e_n,p^{-1}f_1,\dotsc,p^{-1}f_s,f_{s+1},\dotsm,f_n)\oplus \brL_{\an}$$  so that $\brL^\vee/\brL$ is spanned by $e_1,\dotsc,e_s,p^{-1}f_1,\dotsc,p^{-1}f_s$. Thus we have $\dot{w}\brL\nsubseteq \brL^\vee$ if and only if $\dim(\brL+\dot{w}\brL)/\brL=1$ and the image of $\dot{w}\brL$ in $p^{-1}\brL/\brL$ is not contained in  $\Span(e_1,\dotsc,e_s,p^{-1}f_1,\dotsc,p^{-1}f_s)\mod \brL$. By inspecting the possibilities for $\dot{w}x^{(s)}$ in Lemma \ref{lem: dim=1 implies permissible}, we find that this occurs if and only if $\dot{w}x^{(s)}=x^{(s)}\pm \epsilon_i$ for $i=s+1,\dotsc,n$, or equivalently $w\in Z_K$.
\end{proof}

\subsubsection{}Let $\Psi_{\calZ,\Lambda}:  \cZ^{(0)}(\Lambda)_{\red}\rightarrow S_{\Lambda}$ denote the morphism of schemes constructed in Proposition \ref{prop: morphism BT strata to DLV}.  Let $S_{\Lambda}^{\smallheartsuit}\subset S_{\Lambda}$ denote the subscheme consisting of subspaces $\scrV\subset \Omega_{\Lambda}$ with $\Phi(\scrV) \nsubseteq \scrV^\vee$ (equivalently, $\scrV+\Phi(\scrV)$ is not totally isotropic). We write $\cZ^{\smallheartsuit}(\Lambda)\subset \cZ_\red^{(0)}(\Lambda)$ for the preimage of $S_\Lambda^{\smallheartsuit}$ under the map induced by $\psi$.
	\begin{lemma}\label{lem: iso tangent space}
Assume $S_\Lambda$ (resp.  $\cZ_\red^{(0)}(\Lambda)$) has positive dimension. Then		$S_\Lambda^{\smallheartsuit}$  (resp. $\cZ^{\smallheartsuit}(\Lambda)$) is a Zariski open and dense subscheme of $S_\Lambda$ (resp. $\cZ_{\red}^{(0)}(\Lambda)$).
		
	\end{lemma}	
	\begin{proof}Since the morphism $\cZ_{\red}(\Lambda)^{(0)}\rightarrow S_\Lambda$ is a universal homeomorphism, it suffices to prove this for $S_\Lambda$. By  Theorem \ref{thm: decom pf SLambda}, we have $$S_\Lambda\cong 
     \begin{cases}
     S_\Lambda^{\pm} \cong 
      X_{P_0}(1)\cup X_{P_0}(t^{\mp}) & \text{ if $h=0$},\\
      S_\Lambda \cong 
      X_{P_0}(1)\cup X_{P_0}(s_{\frac{t-h}{2}}) & \text{ if $h=1$},\\ 
      S_\Lambda \cong 
      X_{P_0}(1)\cup X_{P_0}(t^+)\cup X_{P_0}(t^-) \cup X_{P_0}(w_1) & \text{ if $h=2$},\\
        S_\Lambda \cong 
      X_{P_0}(1)\cup X_{P_0}(s_{\frac{t-h}{2}}) \cup X_{P_0}(w_1) & \text{ if $h>2$}.
      \end{cases}
   $$
      Then under this isomorphism, we find that $$S_\Lambda^{\smallheartsuit}\cong\begin{cases}
          S_\Lambda &\text{ if $h=0$}\\
          X_{P_0}(w_1) & \text{ if $h>0$}
      \end{cases},$$ which is open and dense in $S_\Lambda.$ 
	\end{proof}

\subsubsection{}
Now let $y\in \calZ^{\smallheartsuit}(\Lambda)(k)$.  Recall we have the associated 
$\brZ_p$-submodules $$M_{y,1},M_y\subset D_{\breve \bbQ_p}$$ (see \S\ref{para: associated special lattice}) and $L_y',L_y\subset V_{\brQ}$, where $M_y\cong \bbD(\scrG_y)$, $M_{y,1}=V\bbD(\scrG_y)$, and we have $$L_y'=\Phi L_y=\{v\in 
V_{\brQ_p}:vM_y\subset M_y\}.$$
There is a unique $g\in X(\mu^\sigma,b)$ such that $g\brD=M_y$; then $\Phi L_y=g\bullet\brL$ and hence $\Phi L_y$ is a vertex lattice of type $h$ in $\breve V$.

We set $\overline{M}_y=M_y/pM_y$. We also set $\overline{\Phi L}_y=\Phi L_y/p\Phi L_y$ which is a possibly degenerate quadratic space over $k$ and let $\overline{\Phi L}_{y,1}$ denote the one-dimensional subspace given by the image of $p(L_y+\Phi L_y)$. Since $y\in \cZ^\circ(\Lambda)$, $\overline{\Phi L}_{y,1}$  does not lie in the radical of the quadratic space $\overline{\Phi L}_{y}$.

Now let $u\in T_y\calN_{L,k}$ which we consider as a point of $\calN_L(k[\epsilon])$ lifting $y$.   By Grothendieck--Messing theory, $u$ corresponds to a filtration $\overline{M}_{u,1}$ on $\overline{M}_y\otimes_k k[\epsilon]$ lifting the image of $M_{y,1}$, and we set $$\overline{\Phi L}_{u,1}=\{ v\in \overline{\Phi L}_y\mid v\overline{M}_{u,1}\subset \overline{M}_{u,1}\}.$$
As in  \S\ref{para: subspace local model}, we let $\scrD_y$ denote the set of rank one $k[\epsilon]$ submodules of $\overline{\Phi L}_y\otimes_k k[\epsilon]$ lifting $\overline{\Phi L}_{y,1}$,  and  $$\phi_y:\scrD_y\rightarrow \Hom_k(\overline{\Phi L}_{y,1}, \overline{\Phi L}_{y,1}^\perp/\overline{\Phi L}_{y,1})$$ is the map characterized by $F=\langle m+\phi_y(F)(m)\epsilon\rangle$ for $F\in \scrD_y$.    The following key proposition gives a description of the tangent spaces of $\calN_{L,\red}$ and its Bruhat--Tits strata.

\begin{proposition}\label{prop: tangent space RZ}
\begin{enumerate}(cf. \cite[Corollary 4.2.7]{LiZhu2018}) \item We have $\overline{\Phi L}_{1,u}\in \scrD_y$, and the map $$\Theta_y: T_y\calN_{L,k}\rightarrow \Hom_k(\overline{\Phi L}_{y,1}, \overline{\Phi L}_{y,1}^\perp/\overline{\Phi L}_{y,1})$$given by $u\mapsto \phi_y(\overline{\Phi L}_{1,u})$ is an isomorphism.
\item $\Theta_y$ restricts to an isomorphism $$T_y\calZ(\Lambda)_{k}\rightarrow \Hom_k(\overline{\Phi L}_{y,1},\langle \overline{\Phi L}_{y,1},\overline{\Lambda}\rangle^{\perp}/\overline{\Phi L}_{y,1})$$ where $\overline{\Lambda}$ denotes the image of $\Lambda$ in $\overline{\Phi L}_{y}$. In particular, we have $$\dim T_y\calZ(\Lambda)_k=(t+h)/2-1$$ where $t$ is the type of $\Lambda$.
\end{enumerate}
	\end{proposition}

 \subsubsection{}We will need the following preparatory lemmas. For $R$ a $\bbZ_p$-algebra, we have a morphism $L_R\rightarrow \widetilde{L}_R\rightarrow \End_R(D_R),$ where we consider $\widetilde{L}$ as a $\bbZ_p$-module by restriction of structure from $\calO_E$, and we recall that $D_R=\Hom_{\bbZ_p}(C(\widetilde{L}),\bbZ_p)_R$. For an element $\xi\in L_R$, we let $\ker(\xi)\subset D_R$ denote the kernel of the endomorphism corresponding to $\xi$.

 \begin{lemma}\label{lem: preserves kernel if and only orthogonal}
  Let $\xi\in L_R$ be the generator of an isotropic line in $L_R$ and assume there exists $\zeta\in L$ such that $(\zeta,\xi)=1$. Then for any $v\in L_R$, left multiplication by $v$ preserves $\ker(\xi)\subset D_R$ if and only if $v$ is orthogonal to $\xi$.\end{lemma}

 \begin{proof}
     Under our assumptions, this follows using a similar proof to \cite[Lemma 4.1.2]{LiZhu2018}, we recall the argument for the convenience of the reader. 
     
     We let $\mathrm{im}_C(\xi)\subset C(\widetilde{L})_R$ denote the  image $\xi$. Then by duality, it suffices to show $v$ preserves $\mathrm{im}_C(\xi)$ if and only if $v$ is orthogonal to $\xi$.
     
 If $v$ is orthogonal to $\xi$, then $v\xi=-\xi v$, so $v$ preserves $\mathrm{im}_C(\xi)$. Conversely, assume $\xi$ preserves $\im_C(\xi)$; then our assumption that $(\zeta,\xi)=1$ implies that $L_R=\xi^\perp\oplus R\zeta$. Upon replacing $\zeta$ by $\zeta-\frac{q(\zeta)\xi}{2}$, we may assume that $\zeta$ is isotropic. Then in $C(\widetilde L)$ we have \begin{equation}\label{eqn: Clifford eq}2=q(\zeta+\xi)=\zeta\xi+\xi\zeta.\end{equation}

Suppose  $\lambda\zeta$ preserves $\im_C(\xi),$ $\lambda\in R$. Then in $C(\widetilde L)$, we have \begin{align*}
         0&=\xi(\lambda\zeta)\xi\\
         &=\lambda(2-\zeta\xi)\xi\\
         &=2\lambda\xi
     \end{align*}
Here the first equality follows from the fact that $\xi\xi=0$ and $\lambda\zeta$ preserves $\im_C(\xi)$, the second follows from \eqref{eqn: Clifford eq}, and the third follows from $q(\xi)=0$. Thus $\lambda=0$; hence no non-zero multiple of $\zeta$ preserves $\im_C(\xi).$

    Now write $v=v_1+\lambda \zeta$ with $v_1\in \xi^\perp$, $\lambda\in R$. Then $v_1$ preserves $\im_C(\xi)$ by the first part of the proof, and hence $\lambda\zeta$ preserves $\im_C(\xi)$. Thus $\lambda=0$ and $v=v_1$ as desired.
 \end{proof}

   \begin{lemma}\label{lem: kernel clifford algebra}
 Let $x\in \calM^w(k), w\in Z_K$, and $v\in T_x\calM$, and let $\xi$ be a generator of $\overline{L}_{1,v}$. We have $\ker(\xi)=\overline{D}_{1,v}$.

\end{lemma}

\begin{proof} We have $\overline{D}_{1,v}=g\overline{D}_{x,1}:=g\overline{D}_{1,k}\otimes_k k[\epsilon]$ and $\overline{L}_{1,v}=g\bullet\overline{L}_{x,1}:=g\bullet\overline{L}_{1,k}\otimes_k k[\epsilon]$ for some $g\in \calG(k[\epsilon]).$
Thus it suffices to show that $\overline{D}_{1,k}=\ker(\xi_1)$, where $\xi_1$ is a generator of the isotropic line $\overline{L}_{1,k}$. By definition of $\overline{L}_{1,k}$, we have an inclusion $\overline{D}_{1,k}\subset \ker(\xi_1)$. Note that  $$\dim \ker(\xi_1)= \frac{1}{2}\dim \overline{D}=\overline{D}_{1,k}.$$ Thus we have equality by a dimension count.
\end{proof}

\begin{proof}[Proof of Proposition \ref{prop: tangent space RZ}]As above, the identification $M_y\cong \breve D$ induced by $g\in X(\mu^\sigma,b)$ gives rise to an embedding $\Mloc_{\calG,\mu}\rightarrow \Gr(M_y,d)$ such that $\overline{M}_{y,1}$ corresponds to some $x\in \Mloc_{\calG,\mu}(k)$. We fix an identification $\widehat{\calN}_{L,y}\cong \Spf R$ as in Proposition \ref{prop: KP deformation theory}, where $R$ is the complete local ring of $\Mloc_{\calG,\mu}$ at $x$. This induces an isomorphism $$T_y\calN_{L,k}\cong T_x\calM,$$ and we write $v\in T_x\calM$ for the image of $u$.  Then by Proposition \ref{prop: KP deformation theory} (2), the filtration on $\overline{M}_y\otimes k[\epsilon]$ corresponding to $u$, is equal to the filtration corresponding to $v$ under the embedding $\Mloc_{\calG,\mu}\rightarrow \Gr(M_y,d)$. Thus $\overline{\Phi L}_{1,u}\in \scrD_y$ by Proposition \ref{prop: tangent space LM} (1), and we have a commutative diagram
\[\xymatrix{T_y\calN_{L,k}\ar[r]\ar[d] & \scrD_y\ar[r]^{\!\!\!\!\!\!\!\!\!\!\!\!\!\!\!\!\!\!\!\!\!\!\!\!\!\!\!\!\!\!\!\!\!\!\!\!\phi_y}\ar[d]&\Hom_k(\overline{\Phi L}_{y,1}, \overline{\Phi L}_{y,1}^\perp/\overline{\Phi L}_{y,1})\ar[d]\\
T_x\calM\ar[r]	& \scrD \ar[r]^{\!\!\!\!\!\!\!\!\!\!\!\!\!\! \phi}& \Hom_k(\overline{L}_1,\overline{L}_1^\perp/\overline{L}_1)}\]
where the vertical maps are isomorphisms. Part (1) follows from this.

For part (2), let $u \in T_y\calN_{L,k}$. Then by Grothendieck-Messing theory, $u\in T_y\calZ(\Lambda)_{k}$ if and only if $\Lambda$ preserves $\overline{M}_{1,u}$. Under the identifications above and by Lemma \ref{lem: kernel clifford algebra}, we have $\overline{M}_{1,u}=\ker(\xi)$, where $\xi$ is a generator of  $\overline{\Phi L}_{1,u}$. By Lemma \ref{lem: preserves kernel if and only orthogonal} and the identifications above, $u\in T_y\calZ(\Lambda)_{k}$ if and only if $\overline{\Lambda}$ is orthogonal to $\overline{\Phi L}_{1,u}$, or equivalently $\Theta_y(u)$ has image in $\overline{\Lambda}^\perp$. Note that the assumptions of Lemma \ref{lem: preserves kernel if and only orthogonal} are satisfied since $\overline{\Phi L}_{y,1}$ does not lie in the radical of $\overline{\Phi L}_y$.  Then in particular part (2) follows from the fact that $\langle \overline{\Phi L}_{y,1},\overline{\Lambda}\rangle^{\perp}$  is isomorphic to $\Lambda^\vee/\Phi L_y$ which has dimension $t-h$.

\end{proof}

\subsubsection{} We apply the above to describe the strata $\calZ_\red^{(0)}(\Lambda).$

 \begin{proposition}\label{prop: bijection on tangent space}
	Let $x\in \cZ^{\smallheartsuit}(\Lambda)(k)$ and let  $z=\Psi_{\calZ,\Lambda}(x)\in S_\Lambda^{\smallheartsuit}(k)$. The map $\Psi_{\calZ,\Lambda}$  induces an isomorphism of tangent spaces  $$T_x\cZ_{\red}^{\smallheartsuit}(\Lambda)\cong T_zS_\Lambda^{\smallheartsuit}.$$
\end{proposition}	

\begin{proof}
First, note that the inclusion $$T_x\cZ_{\red}^{\smallheartsuit}(\Lambda)\subset T_x\cZ^{\smallheartsuit}(\Lambda)_k$$ is in fact an isomorphism.  Indeed, we have $$\dim T_x\cZ_{\red}^{\smallheartsuit}(\Lambda)\geq \dim \calZ_{\red}^{\smallheartsuit}(\Lambda)=\dim S_\Lambda^{\smallheartsuit}=(t+h)/2-1=\dim T_x\cZ^{\smallheartsuit}(\Lambda)_k.$$ 
Here the first equality follows since $\Psi_{\calZ,\Lambda}$ is a universal homeomorphism, the second equality is Theorem \ref{thm: decom pf SLambda} and the last equality is Proposition \ref{prop: tangent space RZ} (2). The equality $T_x\cZ_{\red}^{\smallheartsuit}(\Lambda)=T_x\cZ^{\smallheartsuit}(\Lambda)_k$ then follows by a dimension count.

Now let $\scrV\subset \Omega_{\Lambda}$ denote the image of $L_x$ so that $\scrV$ corresponds to the point $z\in S_\Lambda^{\smallheartsuit}(k)$. Recall by Proposition \ref{prop: tangent space of SLambda}, we have an isomorphism 
$$T_zS_{\Lambda}^{\smallheartsuit}\cong \Hom_k(\scrV+\Phi(\scrV)/\Phi(\scrV),\Omega_{\Lambda}/(\scrV+\Phi(\scrV))).$$
Now multiplication by $p^{-1}$ induces isomorphisms 
$$\overline{\Phi L}_{x,1}\cong(\scrV+\Phi(\scrV))/\Phi(\scrV)$$
$$\langle\overline{\Phi L}_{x,1},\overline{\Lambda}\rangle^\perp/\overline{\Phi L}_{x,1}\cong \Omega_{\Lambda}/(\scrV+\Phi(\scrV)).$$ 
where the second isomorphism follows from the isomorphism $\langle\overline{\Phi L}_{x,1},\overline{\Lambda}\rangle^\perp\cong p\Lambda^\vee/p\Phi L_x$. It follows that we have an isomorphism $$\Hom_k(\scrV+\Phi(\scrV)/\Phi(\scrV),\Omega_{\Lambda}/(\scrV+\Phi(\scrV)))\cong \Hom_k(\overline{\Phi L}_{x,1},\langle \overline{\Phi L}_{x,1},\overline{\Lambda}\rangle^{\perp}/\overline{\Phi L}_{x,1}).$$ It follows from the definitions that the following diagram commutes
\[\xymatrix{T_x\cZ^{\smallheartsuit}(\Lambda)_k \ar[r] \ar[d]& T_zS_\Lambda^{\smallheartsuit} \ar[d]\\
\Hom_k(\scrV+\Phi(\scrV)/\Phi(\scrV),\Omega_{\Lambda}/(\scrV+\Phi(\scrV)))\ar[r] & \Hom_k(\overline{\Phi L}_{x,1},\langle \overline{\Phi L}_{x,1},\overline{\Lambda}\rangle^{\perp}/\overline{\Phi L}_{x,1}).}\]
Here the top arrow is induced by $\Psi_{\calZ,\Lambda}$ and the vertical arrows are the isomorphisms given by Proposition \ref{prop: tangent space RZ} and Proposition \ref{prop: tangent space of SLambda}. Since the bottom arrow is an isomorphism, so is the top arrow and the result follows.
\end{proof}

\begin{theorem} \label{thm: iso}
The map $\Psi_{\calZ,\Lambda}:\calZ_\red^{(0)}({\Lambda})\rightarrow S_\Lambda$ is an isomorphism. 
	\end{theorem}

\begin{proof}

Let $\Psi_{\calZ,\Lambda}^{\smallheartsuit}: \cZ_\red^{\smallheartsuit}(\Lambda)\rightarrow S_\Lambda^{\smallheartsuit}$ denote the restriction of $\Psi_{\calZ,\Lambda}$.  We first show that $\Psi_{\calZ,\Lambda}^{\smallheartsuit}$ is an isomorphism. Indeed it is a proper morphism and induces a bijection on $k$-points. Moreover, by Proposition \ref{prop: bijection on tangent space}, $\Psi_{\calZ,\Lambda}^{\smallheartsuit}$ induces an isomorphism of tangent spaces at closed points. Thus $\Psi_{\calZ,\Lambda}^{\smallheartsuit}$ is a closed immersion, hence an isomorphism since both schemes are reduced. It follows that $\Psi_{\calZ,\Lambda}$ is a birational morphism.

Now $\Psi_{\calZ,\Lambda}$ itself is also proper and induces a bijection on closed points. Thus $\Psi_{\calZ,\Lambda}$ is a finite birational morphism and $S_\Lambda$ is normal by Theorem \ref{thm: decom pf SLambda}. It follows that $\Psi_{\calZ,\Lambda}$ is an isomorphism.
\end{proof}
\begin{theorem}\label{thm: isoY}
  The map $\Psi_{\calY,\Lambda}:\calY_\red^{(0)}({\Lambda}^\vee)\rightarrow R_\Lambda$ is an isomorphism. 
\end{theorem}

\begin{proof} Recall that $V^{(p)}$ denotes the quadratic space $(V,p(\ ,\ ))$ and $L^\sharp$ the lattice $L^\vee$ considered as a lattice in $V^{(p)}$, see Definition \ref{def: L^sharp}. Then we have the associated RZ-space $\calN_{L^\sharp}$ which is equipped with an isomorphism $\Upsilon:\calN_L\cong \calN_{L^\sharp}$ by the discussion in \ref{para: RZ space GSp}. For $\Lambda\in \cV^{\le h}$, we have $\Lambda^\sharp\in \cV^{\ge n-h}(\bV^\sharp)$.     Then we claim the following diagram is commutative.
     \[\xymatrix{
        \cY_{L,\red}^{(0)}(\Lambda^\vee)\ar[d]_{\Psi_{\calY,\Lambda}}\ar[r]& \cZ_{L^\sharp,\red}^{(0)}(  \Lambda^{\sharp})\ar[d]^{\Psi_{\calZ,\Lambda}}\\
        R_\Lambda      \ar[r]    &      S_{\Lambda^\sharp}}\] where the top map is induced by the duality morphism $\Upsilon$ and the bottom map is given by Lemma \ref{lem: duality DL var}. Indeed since all schemes are reduced, it suffices to check this on $k$-points. Assume $M$ is a special lattice corresponding to a $k$-point of $\cY_{L,\red}^{(0)}$ via Propositions \ref{prop: k points of N} and  \ref{prop: k points}. In particular, $\Lambda^\vee \subset M^\vee$. By the proof of Propositions \ref{prop: duality for Z(Lambda)} and \ref{prop: k points of N},  $\Upsilon(M)=\widetilde{\pi}M^\vee$. Note that $\Lambda^\sharp=\widetilde{\pi}\Lambda^\vee \subset \Upsilon(M)$. So $\Upsilon(M)\in \cZ_{L^\sharp,\red}^{(0)}(k)$ indeed.  Moreover, $\Psi_{\cY,\Lambda}(M)=pM^\vee/p\Lambda^\vee $ and $\Psi_{\cZ,\Lambda}(\Upsilon(M))= M^\sharp/\Lambda^\sharp$. The claim follows from the fact that the map sending $\Psi_{\cY,\Lambda}(M)=pM^\vee/p\Lambda^\vee $ to  $  M^\sharp/\Lambda^\sharp$ is exactly the map given by Lemma \ref{lem: duality DL var}.   \end{proof}

\subsection{The Bruhat-Tits stratification}

In this last subsection, we summarize the main results.

Recall that  $L\subset V$ is a vertex lattice of type $h$ and we use $\cV^{\ge h}$ (resp. $\cV^{\le h}$ ) to denote the set of vertex lattices in $\bV$ of rank $n$ and type $t \ge h$ (resp. $\le h$). We refer the reader to \S \ref{sec: DL for SO} for notations about Deligne--Lusztig varieties. 

Recall that we have a natural decomposition
\begin{align*} 
\cN_L=\bigsqcup_{\ell \in \mathbb{Z}} \cN_L^{(\ell)},
\end{align*}
where  $\cN_L^{(\ell)} \subset \cN_L$ be the open and closed formal subscheme on which $\operatorname{ord}_p(c(\rho))=\ell$.
\begin{theorem}\label{thm: main thm}\hfill
 \begin{enumerate}
\item    For each $\ell \in \Z$, we have
    \begin{align*}
        \cN_{L,\red}^{(\ell)} = \Big (\cup_{\Lambda \in \cV^{\ge h}} \cZ_\red^{(\ell)}(\Lambda) \Big ) \cup \Big(\cup_{\Lambda \in \cV^{\le h}} \cY_\red^{(\ell)}(\Lambda^\vee) \Big). 
    \end{align*}

\item For $\Lambda \in \cV^{\ge h}$ and  $\Lambda \in \cV^{\le h}$ respectively, we have 
\begin{align*}
    \Psi: \cZ_\red^{(\ell)}(\Lambda)\cong S_{\Lambda }\quad \text{and}\quad  \Psi: \cY_\red^{(\ell)}(\Lambda ^\vee)\cong R_{\Lambda },
\end{align*}
where $S_{\Lambda}$ and  $ R_{\Lambda}$ are irreducible normal generalized Deligne-Lusztig varieties of dimension $\frac{t_\Lambda+h}{2}-1$ and  $n-\frac{t_\Lambda+h}{2}-1$ respectively.

\item

Assume   $\Lambda \in \cV^{\ge h}$. Let
\begin{align*} 
\cZ_\red^\circ(\Lambda)&\coloneqq \cZ_\red^{(\ell)}(\Lambda)\setminus \cup_{\Lambda'\in \cV^{<t}}\cZ_\red^{(\ell)}(\Lambda').
 \end{align*}  
 Then  
    \begin{enumerate}
        \item $\cZ_\red^{(\ell)}(\Lambda)=\overline{\cZ_\red^\circ}(\Lambda)=\sqcup_{\Lambda\subset \Lambda'}\cZ_\red^\circ(\Lambda')$.
        
        \item   $\Psi_{\calZ,\Lambda}: \cZ_\red^{(\ell)}(\Lambda)\stackrel{\sim}{\to}$ $S_{\Lambda}$ induces the isomorphism 
        \begin{align*}
            \cZ_\red^{\circ}(\Lambda)\cong S_\Lambda^{\circ}=S_\Lambda^{\circ\smallheartsuit}\sqcup S_\Lambda^{\circ\smalldagger},
        \end{align*}
        where 
       \begin{align*}
   S_\Lambda^{\circ\smallheartsuit}&=  \begin{cases}
         X_{P_{\Lambda}}(w_{\Lambda}^+ ) \sqcup X_{P_{\Lambda}}(w_{\Lambda}^- ) & \text{ if $h=0$},\\
      X_{P_{\Lambda}}(w_{\Lambda}) & \text{ if $h=1$},\\ 
              X_{P}(w_{\Lambda}) & \text{ if $h=2$},\\
         X_{P_{\Lambda}}(w_{\Lambda}) & \text{ if $h>2$}.
      \end{cases}\\
     S_\Lambda^{\circ\smalldagger}&=\begin{cases}
     \emptyset& \text{ if $h=0$},\\
     \emptyset & \text{ if $h=1$}\\  
           X_{P_{\Lambda}}(w_{\Lambda}^{\prime +}) \cup    X_{P_{\Lambda}}(w_{\Lambda}^{\prime -})  & \text{ if $h=2$},\\
      X_{P_{\Lambda}}(w_{\Lambda}')    & \text{ if $h>2$}.
      \end{cases}
  \end{align*}
 
    \end{enumerate}

\item  Assume   $\Lambda \in \cV^{\ge h}$. Let  
\begin{align*} 
\cZ_\red^{0}(\Lambda)&\coloneqq \cZ_\red^{(\ell)}(\Lambda)\cap
 \Big(\cup_{\Lambda \in \cV^{\le h}} \cY_\red^{(\ell)}(\Lambda^\vee) \Big).
 \end{align*}  
Then
    \begin{enumerate}
        \item $\cZ_\red^{0}(\Lambda)\cong S_\Lambda^0=\cup_{\substack{\Lambda\subset \Lambda',\\ t(\Lambda')\le h}} S_{\Lambda',\Lambda}$.
        \item     $ 
       \cZ_\red^{(\ell)}(\Lambda)\cap \cY_\red^{(\ell)}((\Lambda')^\vee)\cong S_{\Lambda',\Lambda} \cong \sqcup_{\substack{\Lambda \subset \Lambda_1\subset \Lambda_2\subset \Lambda'\\
       t(\Lambda_2)\le h \le t(\Lambda_1)}} X_{I_{\Lambda_1,\Lambda_2}}(w_{\Lambda_1,\Lambda_2}).
     $

    \end{enumerate}
    \item Let $L^\sharp$  be $L^\vee$ with quadratic form $q(\,)_{L^\sharp}=p q(\,)_{L^\vee}$. We have a natural isomorphism $\Upsilon_{L,L^\sharp}:\cN_L\cong \cN_{L^\sharp}$ so that it induces  an isomorphism
         \begin{align*}
  \cY_{L,\red}^{(\ell)}(\Lambda^\vee) \cong  \cZ_{L^\sharp,\red}^{(\ell)}(\Lambda^\sharp),
\end{align*}
which becomes 
         \begin{equation*}
      R_{\Lambda}^{[h]} \cong S_{\Lambda^\sharp}^{[n-h]}
\end{equation*}
in Lemma \ref{lem: duality DL var} under the isomorphisms $\Psi.$  Then similar results of $(3)$ and $(4)$ for $\cY_\red^{(\ell)}(\Lambda^\vee)$ follow via   the above isomorphisms.
 \end{enumerate}
\end{theorem}
\begin{proof}
   (i) is Theorem \ref{thm: N=ZunionY}. 

   (ii) follows from Theorems \ref{thm: iso} and \ref{thm: isoY} when $\ell=0$. The isomorphism between $\cN_L^{(0)}$ and $\cN_L^{(\ell)}$ induces an isomorphism between $\cZ_\red^{(0)}$ and $\cZ_\red^{(\ell)}$ so that the general case follows from the case when $\ell=0$. The irreducibility, normality, and dimension of $S_\Lambda$ and $R_\Lambda$ are Theorem \ref{thm: decom pf SLambda}.

   (iii) follows from Theorem \ref{thm: decom pf SLambda} and Proposition \ref{prop: further decom of DL}.

   (iv) follows from Propositions \ref{prop: decomp of SLambdaLambda'} and \ref{prop: further decom of DL} and the definitions of $\cZ_\red^{0}(\Lambda)$ and $ S_\Lambda^0$.

   (v) follows from Proposition \ref{prop: duality for Z(Lambda)} and the proof of Theorem \ref{thm: isoY}.
\end{proof}

	\bibliographystyle{alpha}
	\bibliography{reference}
\end{document}